\documentclass{memo-l}

\usepackage{diagrams}
\usepackage{epsf,latexsym,amsmath,amssymb,amscd,datetime}
\newtheorem{theorem}{Theorem}[chapter]
\newtheorem{lemma}[theorem]{Lemma}

\newtheorem{corollary}[theorem]{Corollary}

\newtheorem{definition}[theorem]{Definition}

\newtheorem{exercise}[theorem]{Exercise}

\numberwithin{section}{chapter}
\numberwithin{equation}{chapter}

\newcommand{\cal}{\mathcal}

\newcommand{\mt}{\widetilde}
\newcommand{\free}{{\rm Free}\,}

\newcommand{\cH}{{\cal K}}
\newcommand{\cK}{{\cal L}}
\newcommand{\cG}{{\cal G}}
\newcommand{\cF}{{\cal F}}

\newcommand{\kone}{L}
\newcommand{\twist}{{\rm twist}}

\newcommand{\subgraph}{\subset}

\newcommand{\Ext}{{\rm Ext}}
\newcommand{\sheaves}{{\bf Sh}}

\newcommand{\Cat}[1]{{{\rm Cat}(#1)}}

\newcommand{\isom}{\simeq} 

\newcommand{\from}{\colon}
\newcommand{\ignore}[1]{}

\newcommand{\Hom}{{\rm Hom}}

\newcommand{\ca}{{\cal A}}

\newcommand{\ck}{{\cal K}}

\newcommand{\ce}{{\cal E}}

\newcommand{\cg}{{\cal G}}

\newcommand{\cm}{{\cal M}}

\newcommand{\cp}{{\cal P}}

\newcommand{\affine}{{\mathbb A}}
\newcommand{\field}{{\mathbb F}}

\newcommand{\integers}{{\mathbb Z}}

\newcommand{\complex}{{\mathbb C}}

\begin{document}

\frontmatter

\title{Sheaves on Graphs, Their Homological Invariants,
  and a Proof of the Hanna Neumann Conjecture}


\author{Joel Friedman}
\address{Department of Computer Science, 
        University of British Columbia, Vancouver, BC\ \ V6T 1Z4, CANADA,
        and Department of Mathematics, University of British Columbia,
        Vancouver, BC\ \ V6T 1Z2, CANADA. }
\curraddr{}
\email{{\tt jf@cs.ubc.ca} or {\tt jf@math.ubc.ca}}
\thanks{Research supported in part by an NSERC grant.  Research
        done in part at the Centre Bernoulli, funded by the Swiss
        National Science Foundation.}


\date{May 30, 2011}



\begin{abstract}
In this paper we establish some foundations regarding sheaves of vector
spaces on graphs and their invariants, such as homology groups and their
limits.  
We then use these ideas to prove the Hanna Neumann
Conjecture of the 1950's; in fact, 
we prove a strengthened form of the conjecture.

We introduce a notion of a sheaf of vector spaces on a graph, 
and develop the foundations of homology theories for such sheaves.
One sheaf invariant,
its ``maximum excess,'' has a number of remarkable
properties.
It has a simple definition, with no reference to homology theory,
that resembles graph expansion.
Yet it is a ``limit'' of Betti numbers, and hence has a
short/long exact sequence theory and resembles the $L^2$ Betti numbers
of Atiyah.
Also, the maximum excess is defined via a 
supermodular function, which gives the maximum excess
much stronger properties than
one has of a typical Betti number.

Our sheaf theory can be viewed as a vast generalization of algebraic graph
theory: each sheaf has invariants associated 
to it---such as Betti numbers and Laplacian matrices---that generalize
those in classical graph theory. 

We shall use ``Galois graph theory'' to reduce the Strengthened
Hanna Neumann Conjecture to showing that certain sheaves, that we
call $\rho$-kernels,
have zero maximum excess.
We use the symmetry in Galois theory to argue that if the
Strengthened Hanna Neumann Conjecture is false, then the maximum excess
of ``most of'' these $\rho$-kernels must be large.
We then give an inductive argument to show that this is impossible.
\end{abstract}

\maketitle


\setcounter{page}{4}

\tableofcontents

\chapter*{Introduction}

This memoir has two main goals.  First, 
we develop some foundations
on what we call ``sheaves on graphs'' and their invariants.  Second,
using these foundations,
we resolve the Hanna Neumann Conjecture of the 1950's.

Although our foundations of sheaves on graphs seem likely to impact
a number of areas of graph theory, the theme that is common to most
of this memoir is the
Hanna Neumann Conjecture (or HNC).  
Both this conjecture and a strengthening of it,
known as the Strengthened Hanna Neumann Conjecture (or SHNC)
have been extensively
studied 
(see
\cite{burns71,imrich76,imrich77,
servatius83,gersten83,stallings83,walter90,tardos92,dicks94,
tardos96,ivanov99,ar00,dicks01,ivanov01,khan02,meakin02,
jitsukawa03,walter07,everitt08,mineyev10}).
These conjectures are usually stated as
an inequality involving free groups, although both conjectures
have well-known
reformulations in term of finite graphs.
In this memoir we prove both conjectures, using the finite graph 
reformulations,
reducing both to the
vanishing of a homology group of certain sheaves on graphs.

This work was originally written and posted to {\tt arxiv.com}
as two separate articles.
The first aritcle, \cite{friedman_sheaves}, contains the foundational
material on sheaves of graphs, and comprises Chapter~1 of this 
manuscript.
The second article, \cite{friedman_sheaves_hnc}, resolves the SHNC
(and HNC), and Chapter~2 of this manuscript consists of this material.
This manuscript is easier to read than both articles separately, in that
redundant definitions have been discarded, and 
references in Chapter~2 to
material in Chapter~1 are now more specific.
Yet, as we now explain, Chapters~1 and 2 are largely
independently of one another, and Chapter~1, the foundations of
sheaves on graphs, is of interest beyond the HNC and SHNC.
To explain this interest, let us recall
a bit about sheaf theory and its connection to discrete mathematics.

Among many (co)homology theories of topological spaces,
the sheaf approach has many advantages.
For one, it works with non-Hausdorff
spaces, as done first by Serre
in algebraic geometry with the Zariski topology
(see \cite{hartshorne}).
Grothendieck's sheaf theory of \cite{sga4.1,sga4.2,sga4.3,sga5}
defined a notion of a sheaf on very general spaces
now called ``Grothendieck topologies.'' 
While Grothendieck's work has had remarkable success to cohomology
theories
in algebraic
geometry, we believe that graph theory and combinatorics may
greatly benefit by studying very special Grothendieck topologies
formulated from finite, discrete structures.
In particular, we will resolve the SHNC using a simple,
finite Grothendieck topology
associated to any finite graph.

Another aspect of sheaf cohomology is that it vastly generalizes
the cohomology of a space.  Each sheaf has injective resolutions
that give cohomology groups.  When the sheaf is take to be the
``structure sheaf'' of the space, we recover the cohomology groups
of the space.  However, there are many sheaves apart from the
structure sheaf, and the resulting cohomology groups can represent
a variety of aspects of the space.  In particular, each open subset
of a space, $X$, has an associated sheaf on $X$ that reflects many
properties of $X$; we will use
such sheaves in our proof of the SHNC.

One fundamental aspect of any (co)homology theory is that it expresses
relations between related (co)homology groups in terms of exact
sequences.  Furthermore, any exact sequence yields
a triangle inequality 
between the dimensions or ranks of any three consecutive
elements.
In \cite{friedman_cohomology,friedman_cohomology2,friedman_linear}
we began an investigation into applying such inequalities to complexity
theory, in particular to construct formal complexity measures to obtain
lower bounds for formula size.
Similarly, in this manuscript, we prove the SHNC from such an inequality.

In Chapter~1 we define a sheaf on a graph with no reference to sheaf
theory, rather as a collection of vector spaces indexed on the vertices
and edges of the graphs along with certain ``restriction'' maps.
We add that one can view such a sheaf as a simple genelization of an
incidence matrix of a graph; it follows that sheaves on graphs
can be viewed as a vast generalization of classical algebraic graph
theory (of adjacency matrices, Laplacians, etc.).
However,
our sheaves on graphs can also be viewed as
the very special case of sheaves
of finite dimensional vector spaces on a simple Grothendieck
topology that we associate to a finite graph.  In the case where
the graph has no self loops, the Grothendieck topology is equivalent
to a simple topological space.

Chapter~1 begins with a simple definition of sheaves on graphs and
some examples.  However, quickly we begin to study ``limits'' of
Betti numbers of these sheaves.
The most remarkable invariant that we study in Chapter~1 is the
{\em maximum excess} of a sheaf.
We give a number of strong results regarding the maximum excess,
and the related ``twisted'' homology.
These are related to the $L^2$ Betti numbers first studied by Atiyah
(see \cite{atiyah,luck02}); however, the results we
obtain in the case of finite graphs, especially regarding the maximum
excess, seem especially strong.

To summarize the above few paragraphs, 
here are some reasons that Chapter~1 is of interest
independent of the HNC:
\begin{enumerate}
\item sheaf theory on graphs
generalizes algebraic graph theory and, therefore, may strengthen its
applications;
\item our results on maximum excess give tools to study certain graph
invariants such as the ``reduced cyclicity'' and number of ``acyclic
components;''
\item our results on the maximum
excess of sheaves may indicate what one can expect of limits of
Betti numbers on more general structures;
\item any results on sheaves on graphs may give new results and examples
of what to expect on more general finite Grothendieck topologies, 
such as those of
possible interest to complexity theory;
\item any results on Betti numbers of sheaves may yield new inequalities
on other integers that can be viewed as akin to Betti numbers on some
discrete Grothendieck topology.
\end{enumerate}
Of course, despite the above reasons for interest in sheaves on graphs,
the reader will see that Chapter~1 is largely developed with
an eye toward the reduced cyclicity and the HNC.

Let us summarize aspects of Chapter~2, our proof of the SHNC, in general terms.
This will serve to highlight our approach to this problem via 
sheaves on graph,
which is very different than previous approaches.
We use a graph theoretic formulation of the SHNC that involves the
reduced cyclicity of three graphs.
However, using what we call
``Galois graph
theory'' (of \cite{friedman_geometric_aspects,st1}, but also \cite{gross}),
the SHNC amounts to showing that the reduced cyclicity of one graph is
less than that of another graph, and both of these graphs admit a
natural map to the same Cayley graph.

We wish to emphasize that, to the best of our knowledge, 
our manuscript represents the first application of Galois graph theory 
to other parts of graph theory.  That is, Galois graph theory occurs
for its own interest (in \cite{friedman_geometric_aspects}) and for its
connection to number theory (in \cite{st1}).  
However, in this manuscript we make essential use of Galois graph
theory to two
independent questions not obviously related to Galois graph theory.
First, in Chapter~1 we use Galois graph theory to show that maximum excess
scales under pulling back by a covering map; first we prove this for
Galois morphisms, making essential use of the symmetry in Galois theory,
and then we deduce the general case by the Normal Extension Theorem of
Galois graph theory.
Second, 
Galois graph theory is the basis of our construction of
$\rho$-kernels, upon which our approach to the HNC and SHNC is based,
and the symmetry of these $\rho$-kernels is used constantly in Chapter~2.

Let us return to the SHNC, and recall that exact sequences give triangle
inequalities on the dimensions of consecutive terms.
The reduced cyclicity
is a type of limiting first Betti number.  Hence,
one graph has smaller reduced cyclicity than a second graph
provided that there is a surjection from the first graph to the second,
such that the kernel of this surjection has vanishing 
limiting first
Betti number.
Unfortunately there is no such graph surjection in the graphs that arise from
the SHNC.
However, both graphs admit a natural map to the same Cayley graph, and
hence can be viewed as sheaves on this Cayley graph (much as open subsets of
a topological space have associated sheaves).
Remarkably, 
there is a surjection
from the first graph to the second when viewed as sheaves.
The kernel of such a sujection (generally a sheaf)
will be called a $\rho$-kernel, and the SHNC turns out to be implied
by the vanishing limiting first Betti number, or maximum excess,
of an appropriate collection of $\rho$-kernels.

We emphasize that the $\rho$-kernels that we build seem almost forced
upon us, once we look for the surjections described above.  However,
it does not seem to be an easy question, essentially of linear algebra,
to determine whether or not these $\rho$-kernels have vanishing
maximum excess.
In fact, if we define a $\rho$-kernel as the kernel of any surjection
of the two graphs of interest, then there are $\rho$-kernels whose maximum
excess does not vanish.

To complete the proof of the SHNC, we shall show that the maximum excess
of a ``generic'' $\rho$-kernel vanishes.  This main idea is that there
is a symmetry property of the ``excess maximizer,'' which implies that
maximum excess of a generic $\rho$-kernel must be a multiple of the order of 
an associated
Galois group (the group associated to the Cayley graph mentioned above).
From this point one knows that
if the generic maximum excess doesn't vanish, it would
be large; one can then use two different inductive arguments to show that
this is impossible.

For the reader interested only in a proof of the HNC, we mention that
can read its proof in Chapter~2 while skipping most of the material in
Chapter~1.
Indeed, Chapter~2 is based on the ``stand alone'' paper,
\cite{friedman_sheaves_hnc}, written without explicit
reference to homology theory, using only sheaves and maximum excess.
So to read Chapter~2, one needs 
the definitions of sheaves and maximum excess, of
Section~\ref{se:shmain}, the Galois graph theory of Section~\ref{se:shgalois},
and the submodularity of the excess in Section~\ref{se:shme}.
Aside from these results, the proof in \cite{friedman_sheaves_hnc}
needed the fact that the maximum excess is a ``first quasi-Betti number,''
which relies on the main (and most difficult) theorem of Chapter~1.
However, we have recently found a variant of the proof in
\cite{friedman_sheaves_hnc} which does not require this fact.
Hence one can read a complete proof of the HNC and SNHC in this manuscript,
without most of
Chapter~1 and any reference to homology.
However, as explained in Chapter~2, homology still gives valuable insight
into the proof.

We mention that as of writing
\cite{friedman_sheaves,friedman_sheaves_hnc},
Mineyev has informed us of his independent proof of
the HNC and SHNC, first using Hilbert modules
(\cite{mineyev11.1}, based on \cite{mineyev10}), and then
using only combinatorial group theory (\cite{mineyev11.2}).
His approaches seem very different from ours.


We wish to thank 
Laurant Bartholdi, for conversations and introducing us to the SHNC,
and Avner Friedman, for comments on a draft of this manuscript.
We thank Luc Illusie, for an inspiring 
discussion regarding our ideas involving sheaf theory, homology, and the
SHNC; this discussion was a turning point in our research.
We wish to thank for following people 
for conversations:
Goulnara Arjantseva, Warren Dicks,
Bernt Everitt, Sadok Kallel, Richard Kent, 
Igor Mineyev, Pierre Pansu, and Daniel Wise.
Finally, we thank
Alain Valette and the Centre Bernoulli at the EPFL for hosting us
during a programme on limits of graphs, where we met Bartholdi and
Pansu and began this work.

\mainmatter
\chapter[Foundations of Sheaves on Graphs]{Foundations of Sheaves on Graphs and Their Homological Invariants}

\section{Introduction}
\label{se:shintro}

The main goal of this chapter is to 
introduce a notion of a sheaf on a graph and 
to establish some foundational
results regarding the homology groups 
of such sheaves and related invariants.
After developing some general points we shall focus on a remarkable
invariant of a sheaf that we call the {\em maximum excess}.

The {\em maximum excess} of a sheaf 
arises naturally as a ``limit'' of Betti numbers,
akin to $L^2$ Betti number defined by
Atiyah.  Although such limits have been studied in many contexts, we are able
to show some compellingly strong results about these limits in the
case of sheaves on graphs.
First,
the maximum excess
can be defined, with no reference to homology theory, in a manner that
makes it resemble quantities seen in matching theory or expander graphs.
Second, this definition amounts to 
the maximum of an ``excess'' function that is supermodular; this gives
additional structure to the
maximum excess that is not apparent from homology theory.  
Third, for any given sheaf, the limit is attained from ``twisted Betti
numbers'' by passing to a finite
cover (as opposed to an infinite limit of covers).

Our motivation for studying the maximum excess and certain
Betti numbers arose from studying an important graph invariant
that we call the {\em reduced cyclicity} of a graph.
This invariant arises in one formulation of 
the much studied Hanna Neumann Conjecture of the 1950's 
(see
\cite{burns71,imrich76,imrich77,
servatius83,gersten83,stallings83,walter90,tardos92,dicks94,
tardos96,ivanov99,ar00,dicks01,ivanov01,khan02,meakin02,
jitsukawa03,walter07,everitt08,mineyev10}); in Chapter~2
we shall use the results of this chapter to prove this conjecture.
Moreover, our methods
will prove what is known as
the Strengthened Hanna Neumann Conjecture (or SHNC) of \cite{walter90}.

Our sheaf theory on graphs is based on the sheaf theory of Grothendieck 
(see \cite{sga4.1,sga4.2,sga4.3,sga5}), built upon
what are now known as
Grothendieck topologies.
In the special case when the graph has no self-loops, the sheaf theory
we describe is equivalent to the sheaf theory on certain topological
spaces (see \cite{hartshorne}).  
The basic definition of sheaves on graphs and their homology groups
are special cases of theory developed in
\cite{friedman_cohomology,friedman_cohomology2,friedman_linear} 
and are probably special cases of situations arising in the fields
of toric varieties and quivers.
However, in this chapter we study a special case of this general notion
of sheaf theory, proving especially strong theorems particular to
sheaves on graphs and obtaining new theorems in graph theory.
In this process we also introduce new invariants in sheaf theory---such as 
``maximum excess''
and ``twisted homology''---and establish theorems about these invariants
that
may become useful to sheaf theories in other settings.

In this chapter we explore primarily those aspects of
sheaf theory directly related to our future study of the SHNC, namely
general properties of the maximum excess.  However, we believe sheaf
theory is a concept fundamental to graph theory, and that there will
probably emerge other applications of these ideas.
One reason for this belief is that many areas in graph theory, such as
expanding graphs, work with the adjacency matrix of a graph.
Any sheaf on a graph, $G$, has an adjacency matrix (and incidence
matrix, Laplacian, etc.) with many of the properties that graph adjacency
matrices have.  A graph has a particularly simple sheaf that we call
its ``structure sheaf.''
The adjacency matrix of the structure sheaf
turns out to be the adjacency matrix of $G$.
In this way the adjacency matrix of a graph, and all of
traditional algebraic graph theory, can be generalized to
sheaf theory; the sheaf theory, given its more general nature and
expressiveness,
may shed new light on traditional algebraic graph theory and its
applications.

New graph theoretic inequalities arise in our sheaf theory out of
``long exact sequences,'' analogous to long exact sequences that
appear in virtually any homology theory.
Indeed, relations between different homology groups are often
expressed in exact sequences, and in any exact sequence of vector spaces,
the dimensions of three consecutive elements
satisfy a triangle inequality.
It is such triangle inequalities that inspire and form the basis of
our approach to 
the SHNC.

One remarkable aspect of our sheaf theory is that it
adds ``new morphisms''
between graphs. 
In other words, consider two graphs, $G_1$ and $G_2$ that each admit a
morphism to another graph, $G$.
It is possible to associate with each $G_i$ a sheaf,
${\cal S}(G_i)$, over $G$, that contains all the information present in
$G_i$.
Any $G$-morphism from $G_1$ to $G_2$ gives rise to a morphism of sheaves,
from ${\cal S}(G_1)$ to ${\cal S}(G_2)$;
however, there are sheaf morphisms from ${\cal S}(G_1)$ to ${\cal S}(G_2)$
that do not arise from any graph morphism.
For example, there may be a surjection from
${\cal S}(G_1)$ to ${\cal S}(G_2)$ when there is no graph theoretic
surjection $G_1\to G_2$.
Some such ``new surjections'' are
crucial to our proof of the SHNC; the kernel of such ``new surjections''
give a type of sheaf that we call a {\em $\rho$-kernel}, which is the basis
of our approach to the SHNC.
Said otherwise, for any graph, $G$,
there is a faithful functor from the category of ``graphs
over $G$'' to the category of ``sheaves over $G$;'' however this functor is
not full, and some of the ``new morphisms'' between graphs over $G$,
viewed as sheaves over $G$, ultimately yield new
concepts in graph theory needed in our proof of 
the SHNC.

This chapter will focus on four types of invariants of sheaves:
(1) homology groups and resulting Betti numbers, 
(2) twisted homology groups and resulting twisted Betti numbers,
(3) the maximum excess, and
(4) limiting twisted Betti numbers.
Let us briefly motivate our interest in these invariants and describe the
main theorems in this chapter.  This discussion will be made more precise,
with more background, in Section~\ref{se:shmain}.

Our first type of invariant, homology groups of sheaves and resulting
Betti numbers, will not involve any difficult theorems.
The main novelty of this type of invariant is in its definition;
it is chosen in a way that it has appropriate properties for our needs
and
can express some
traditional invariants of a graph; these invariants include 
its Euler characteristic and the traditional zeroth and first Betti numbers.
In sheaf theory, usually sheaf cohomology based on the
global section functor is a central object of study;
however, these cohomology groups do not yield the invariants of interest to 
us in this chapter.
Instead, our homology groups are based on global cosections;
i.e., our homology groups are essentially
Ext groups in the first variable, where the second variable is
fixed to be the
structure sheaf.

The SHNC conjecture can be reformulated in graph theoretic terms, involving
a more troubling graph invariant, $\rho(G)$, of a graph, $G$,
which we call the
{\em reduced cyclicity} of $G$.  
The reason this graph invariant is troubling is that its usual definition
seems to require that we know how many connected components of
$G$ are acyclic, i.e., are isolated vertices or trees.
Prior to this paper, all non-trivial techniques we know
to bound $\rho(G)$ either presuppose something about the number of
acyclic components of $G$, or else they overlook such components; as such,
previous results on the reduced cyclicity usually either require special
assumptions or give results that are not sharp.
Our second set of invariants, the twisted homology groups and their dimensions,
i.e., the
twisted Betti numbers, give $\rho(G_1)$ as the first twisted Betti
number of 
a certain sheaf on $G$, for any graph, $G_1$, with a graph morphism to $G$.  
As such, the long exact sequences arising in twisted homology
give the first sharp relations between values
of $\rho$; however, these relations usually involve sheaves and
not just graphs alone.

Let us sketch the idea of why reduced cyclicity is a special case of
a twisted Betti number.
In this chapter we observe that
$\rho(G)$ is the limit of $h_1(K)/[K\colon G]$
over ``generic Abelian 
coverings maps,'' $K\to G$, where the degree, $[K\colon G]$, of the covering
map tends to infinity.
It is well known that for Abelian covering maps $K\to G$, we can recover
spectral properties of the adjacency matrix of $K$ by 
working with that of $G$ and ``twisting its entries,'' i.e.,
multiplying certain entries by roots of unity that appear in the
characters of the underlying Abelian group.
So we form ``twisted'' homology groups by ``generically twisting'' a sheaf,
with twists that are parameters or indeterminates,
and compute that the reduced connectivity, $\rho(G)$, equals
the  first ``twisted'' Betti
number of the structure sheaf of $G$.
This gives a generalization of the definition of $\rho$ from graphs to
sheaves, and the resulting twisted Betti numbers 
satisfy triangle inequalities coming from
the long exact sequences in twisted homology.

Another promising fact about twisted Betti numbers is that, via
the theory of long exact sequences, one can reduce the
SHNC to the vanishing of the first twisted Betti number of a 
collection of sheaves that we call $\rho$-kernels.

The problem is that the twisted homology approach often seems to be the
``wrong'' way to view the reduced cyclicity, mainly for the following reason.
The
Euler characteristic and reduced cyclicity have a remarkable scaling
property under covering maps, $\phi\from K\to G$, i.e.,
$$
\chi(K)=\chi(G)\deg(\phi), \quad
\rho(K)=\rho(G)\deg(\phi).
$$
Twisted Betti numbers do not always scale in this way; this makes us
suspect that the twisted Betti number is not always
a good generalization of
the reduced cyclicity.

The remedy comes in our third type of invariant, a single invariant of
a sheaf that we shall define and call its {\em maximum excess}.
This is an integer that
one can define simply and with no reference to homology theory.
Its definition resembles combinatorial 
invariants arising in matching theory or
expander graphs.
The maximum excess of any sheaf is at most the first twisted Betti number,
and the two are equal on many types of sheaves, including all constant
sheaves.  Hence the two concepts are related but not identical.
Furthermore,
the SHNC is implied by
the ({\em a priori} weaker)
vanishing maximum excess of $\rho$-kernels, and the maximum excess
satisfies stronger properties that yield better bounds than what
one would
get for the first twisted Betti number.
So for the SHNC, we largely
abandon the idea of using twisted Betti numbers to
generalize $\rho$ from graphs to sheaves, and instead use the
maximum excess.
The problem is that
to proof the SHNC we require inequalities involving the
maximum
excess akin to those holding of Betti numbers of homology theories via
long exact sequences;
there is no {\em a priori} reason that such inequalities should hold.

The main theorem of this chapter, Theorem~\ref{th:shmain},
says that for any fixed sheaf on a graph, $G$, there is an integer,
$q$, with the following property: the maximum
excess and first twisted Betti number agree when the sheaf is
``pulled back'' along a covering map $G'\to G$, provided that the
girth of $G'$ is at least $q$.

The main theorem implies that the maximum excess is a
{\em first quasi-Betti number}, meaning that the maximum excess
satisfies certain triangular inequalities that
we use to prove the SHNC.
However, in Chapter~2 we see that a variation of our proof avoids
these triangular inequalities.

Another view of our main theorem is that there exists a ``limit'' to
the ratio of a twisted Betti number of a pullback of a fixed sheaf
along a graph covering
to the degree of the covering.
We shall call this limiting ratio a ``limiting twisted Betti number,''
which is our fourth type of invariant.  Our main theorem
can be rephrased as saying that the first limiting
twisted Betti number
is just the maximum excess.
It is easy to see that limiting twisted Betti numbers satisfy the 
triangular inequalities we desire for the maximum excess; hence
proving the main theorem proves the desired inequalities for the
maximum excess. 
However, as a limiting Betti number, the
maximum excess actually has associated homology groups whose dimensions
divided by the covering degree approximate the maximum excess.
And
it may turn out that the homology groups themselves may contain useful
information beyond knowing merely their dimension; however,
for our proof of the SHNC,
all that we need is
the dimensions of these homology groups, 
i.e., their Betti numbers.

Lior Silbermann has pointed out to us that our notion of 
limiting twisting Betti numbers
is a discrete analogue of ``$L^2$ Betti numbers'' introduced by
Atiyah on manifolds (\cite{atiyah}); the theory involved in the study of
$L^2$ Betti numbers
(see\cite{luck02}), especially the von Neumann dimension of certain
``matrices'' of this theory, may already imply that our limiting twisting Betti
numbers do have a limit and that it is an integer (because the fundamental
group of a graph is a free group).
So part of our results can be viewed as a very explicit type of
$L^2$ or limiting Betti number calculation (for the very special case
sheaves on graphs), 
that includes stronger information; indeed, we
give a simple interpretation of this number (the maximum excess)
and a finite procedure for computing it (pulling back to a graph
of sufficiently large girth and computing a
twisted Betti number).

We note that for the purpose of proving the SHNC, the main results
needed from this chapter are the definitions of a sheaf and its maximum
excess, and a few properties we prove regarding the
maximum excess.  If we could
prove such properties without using homology theory, we could study
the SHNC without homology theory.  Nonetheless, we find that twisted homology
gives important intuition for the maximum excess; for example, we first
proved the SHNC using twisted homology, and only discovered during the
writing of \cite{friedman_sheaves_hnc} that the proof could be written
entirely in terms of the maximum excess.
As we remark at the end of Chapter~2, there is a way to prove the SHNC
with no reference to homology theory, but this requires some extra
combinatorial analysis (namely Appendix~A).

The rest of this chapter is organized as follows.
In Section~\ref{se:shmain} we give precise definitions and statements of
the theorems in this chapter.
In Section~\ref{se:shgalois} we review part of what might be called
``Galois theory of graphs'' that we will use in this chapter.
In Section~\ref{se:shsheaf} we give the basic properties of sheaves 
and homology,
pullbacks and their adjoints; then we explain everything in terms of 
cohomology of
Grothendieck topologies (this explanation will help the reader to understand
the context of our definitions, but this explanation
is not necessary to read the rest
of this paper).
In Section~\ref{se:shtwist} we define the twisted homology and compute
the twisted homology of the constant sheaf of a graph; we also
interpret twisted homology in terms of Abelian covers.
In Section~\ref{se:shme} we establish the basic properties of the
maximum excess, including its bound on the twisted homology.
The next two sections establish our main theorem.
In Section~\ref{se:shabel} we show how to interpret elements of the
first twisted homology group of a graph in terms of the first 
homology group of the
maximum Abelian covering of the graph.
In Section~\ref{se:shequal} we prove Theorem~\ref{th:shmain}, that says that
the first twisted Betti number and the maximum excess agree after
an appropriate pullback.
In Section~\ref{se:shconclude} we make some concluding remarks.


\section{Basic Definitions and Main Results}
\label{se:shmain}

In this section we will define sheaves and all the main invariants 
of sheaves that we
use in this paper.
We will state the
main theorem in this chapter, and state or describe other results in
this chapter.
In most of this paper we work with directed graphs (digraphs), which
makes things notationally simpler;
as we 
remark in Section~\ref{se:shconclude},
all this sheaf and homology theory works just as well
with undirected graphs, although it is slightly more cumbersome if one wants
to avoid orienting the edges.

\subsection{Definition of Sheaves and Homology}
\label{sb:sheaf_def}
We will allow directed graphs
to have multiple edges and
self-loops;
so in this paper a directed graph (or digraph) consists of tuple
$G=(V_G,E_G,t_G,h_G)$ where $V_G$ and $E_G$ are 
sets---the vertex and edge sets---and
$t_G\from E_G\to V_G$ is the ``tail'' map and 
$h_G\from E_G\to V_G$ the ``head'' map.
Throughout this paper, unless otherwise indicated,
a digraph is assumed to be finite, i.e., the vertex
and edge sets are finite.

Recall that a morphism of digraphs, $\mu\from K\to G$, is a pair
$\mu=(\mu_V,\mu_E)$ of maps $\mu_V\from V_K\to V_G$ and
$\mu_E\from E_K\to E_G$ such that
$t_G\mu_E=\mu_V t_K$ and $h_G\mu_E=\mu_V h_K$.  
We can usually drop the subscripts from $\mu_V$ and $\mu_E$, although
for clarity we shall sometimes include them.

Recall that fibre products exist for directed graphs
(see, for example, \cite{friedman_geometric_aspects}, or
\cite{stallings83}, where fibre products are called ``pullbacks'')
and the fibre
product, $K=G_1\times_G G_2$, of morphisms
$\mu_1\from G_1\to G$ and $\mu_2\from G_2\to G$ has 
$$
V_K=\{(v_1,v_2) \;|\; v_i\in V_{G_i},\;\mu_1 v_1=\mu_2 v_2 \}, 
$$
$$
E_K=\{(e_1,e_2) \;|\; e_i\in E_{G_i},\;\mu_1 e_1=\mu_2 e_2 \},
$$
$$
t_K=(t_{G_1},t_{G_2}),
\quad\mbox{and}\quad h_K=(h_{G_1},h_{G_2}).
$$
For $i=1,2$, respectively,
there are natural digraph
morphisms, $\pi_i\from G_1\times_G G_2\to G_i$ called
projection onto the first and second component, respectively, given by
the respective set theoretic projections on $V_K$ and $E_K$.

We say that $\nu\from K\to G$ is a covering map
(respectively, \'etale\footnote{Stallings, in \cite{stallings83}, uses
the term ``immersion.'' 
}) if 
for each $v\in V_K$, $\nu$ gives a bijection
(respectively, injection)
of incoming edges of $v$ (i.e., those edges whose head is $v$) with those
of $\nu(v)$, and a bijection (respectively, injection)
of outgoing edges of $v$ and $\nu(v)$.
If $\nu\from K\to G$ is a covering map and $G$ is connected, then the
{\em degree} of $\nu$, denoted $[K\colon G]$,
is the number of preimages of a vertex or edge in $G$
under $\nu$ (which does not depend on the vertex or edge); if $G$ is not
connected, one can still write $[K\colon G]$ when the number of preimages
of a vertex or edge in $G$ is the same for all vertices and edges.

Given a digraph, $G$, we view $G$ as an undirected graph (by forgetting
the directions along the edges), and
let $h_i(G)$ denote the $i$-th Betti number of $G$,
and $\chi(G)$ its Euler characteristic; hence $h_0(G)$ is the number
of connected components of $G$, $h_1(G)$ is the minimum number of edges 
needed to be removed from $G$ to leave it free of cycles, and
$$
h_0(G)-h_1(G)=\chi(G) = |V_G|-|E_G|.
$$
Let ${\rm conn}(G)$ denote the connected components of $G$, and let
\begin{equation}\label{eq:shrho}
\rho(G) = \sum_{X \in {\rm conn}(G)} \max(0,h_1(X)-1),
\end{equation}
which we call the {\em reduced cyclicity of $G$}.

For each digraph, $G$, and field, $\field$, our sheaf theory is the
theory of sheaves of finite dimensional $\field$-vector spaces on
a certain
finite Grothendieck topology 
(see \cite{sga4.1,sga4.2,sga4.3,sga5}, where a Grothendieck topology is
called a ``site'')
that we associate to $G$;
this Grothendieck topology
has many properties in common
with topological spaces; in \cite{friedman_cohomology}
we have called these spaces {\em semitoplogical}, and have worked out
the structure of their injective and projective modules, which allows
us to compute derived functors (e.g., cohomology, Ext groups), used in
\cite{friedman_cohomology,friedman_cohomology2,friedman_linear}.
Here we define sheaves and describe a homology theory 
``from scratch,'' without appealing to projective or injective modules; 
later we explain how our homology theory
fits into standard sheaf theory as the derived functors of global
cosections.

\begin{definition} 
\label{de:sheaf}
Let $G=(V,E,t,h)=(V_G,E_G,t_G,h_G)$ 
be a directed graph, and $\field$ a field.  By a {\em sheaf of 
finite dimensional $\field$-vector
spaces on $G$}, or simply {\em a sheaf on $G$},
we mean the data, ${\cal F}$, consisting of 
\begin{enumerate}
\item a finite dimensional
$\field$-vector space, ${\cal F}(v)$, for each $v\in V$, 
\item a finite dimensional
$\field$-vector space, ${\cal F}(e)$, for each $e\in E$, 
\item a linear map, ${\cal F}(t,e)\from{\cal F}(e)\to{\cal F}(te)$ for
each $e\in E$, 
\item a linear map, ${\cal F}(h,e)\from{\cal F}(e)\to{\cal F}(he)$ for
each $e\in E$, 
\end{enumerate}
The vector spaces ${\cal F}(P)$, ranging over all $P\in V_G\amalg E_G$
($\amalg$ denoting the disjoint union), are called the {\em values} of
${\cal F}$.  The morphisms ${\cal F}(t,e)$ and ${\cal F}(h,e)$ are
called the {\em restriction maps}.
If $U$ is a finite dimensional vector space over $\field$,
the {\em constant sheaf associated to $U$},
denoted $\underline{U}$, is the sheaf comprised of the value
$U$ at each vertex and edge, with all restriction maps being the
identity map.
The constant sheaf $\underline{\field}$ will be called the
{\em structure sheaf} of $G$ (with respect to the field, $\field$),
for reasons to be explained later.
\end{definition}

The field, $\field$, is arbitrary, although at times we insist
that it not be finite, and at times that it have characteristic zero.

Now we define homology groups.
To a sheaf, ${\cal F}$, on a digraph, $G$, we set
$$
{\cal F}(E)=\bigoplus_{e\in E} {\cal F}(e), 
\quad
{\cal F}(V)=\bigoplus_{v\in V} {\cal F}(v).
$$
We associate a transformation
$$
d_h=d_{h,{\cal F}}\from {\cal F}(E) \to  {\cal F}(V)
$$
defined by taking ${\cal F}(e)$ (viewed as a component of ${\cal F}(E)$)
to ${\cal F}(he)$ (a component of ${\cal F}(V)$) via the map 
${\cal F}(h,e)$.
Similarly we define $d_t$.  We define the {\em differential of ${\cal F}$}
to be
$$
d=d_{\cal F} = d_h - d_t.
$$
\begin{definition} We define the {\em zeroth} and {\em first homology
groups of ${\cal F}$} to be, respectively,
$$
H_0(G,{\cal F}) = {\rm cokernel}(d),\quad
H_1(G,{\cal F}) = {\rm kernel}(d).
$$
We denote by $h_i(G,{\cal F})$ the dimension of $H_i(G,{\cal F})$ 
as an
$\field$-vector space, and call it the {\em $i$-th Betti number of
${\cal F}$}.
We often just write $h_i({\cal F})$ and $H_i({\cal F})$ if $G$ is clear
from the context (when no confusion will arise between $h_i({\cal F})$,
the dimension, and $h$ the head map of a graph).
We call $H_i(\underline\field)$ 
the $i$-th homology group of $G$ with coefficients
in $\field$, denoted $H_i(G)$ or, for clarity, $H_i(G,\underline\field)$.
\end{definition}
For ${\cal F}=\underline\field$, $d$ is just the usual incidence matrix; thus,
if $\field$ is of characteristic zero, then the $h_i(G)$,
i.e., the dimension of the $H_i(G)$, 
are the usual
Betti numbers of $G$.

Define the {\em Euler characteristic of ${\cal F}$} to be
$$
\chi({\cal F})=\dim\bigl({\cal F}(V)\bigr) -
\dim\bigl( {\cal F}(E) \bigr).
$$
Since $d_{\cal F}$ has domain ${\cal F}(E)$ and codomain
${\cal F}(V)$, we have
$$
h_0({\cal F}) - h_1({\cal F}) = \chi({\cal F}).
$$

If $j\from G'\to G$ is a digraph morphism, there is a naturally defined
sheaf $j_!\underline\field$ on $G$ such that $H_i(j_!\underline\field)$ is naturally
isomorphic to $H_i(G')$ ($j_!$ will be defined as a functor from sheaves
on $G'$ to sheaves on $G$ in Subsection~\ref{sb:pullbacks}); 
when $j$ is an inclusion, then
$j_!\underline\field$ is just the sheaf whose
values are $\field$ on $G'$ and $0$ outside
of $G'$ (i.e., on vertices and edges not in $G'$);
we will usually use $\underline\field_{G'}$ to denote $j_!\underline\field$ 
(which is somewhat abusive unless $j$ is understood).
If $\phi\from G'\to G''$ is a morphism of digraphs over $G$, then
$\phi$ gives rise to a natural morphism of sheaves $\underline\field_{G'}\to
\underline\field_{G''}$.  In this way the functor $G'\mapsto \underline\field_{G'}$ 
includes the category of digraphs over $G$ as a subcategory of
sheaves over $G$.
As mentioned before,
one key aspect of sheaf theory is that the functor is not full, i.e., there
exist (very important) morphisms of 
sheaves $\underline\field_{G'}\to\underline\field_{G''}$
that do not arise from a morphism of digraphs $G'\to G''$; such
morphisms will be needed to define sheaves (their kernels) that we call
$\rho$-kernels, which will be crucial to our approach to the SHNC.

Next we give the long exact sequence in homology associated to
a short exact sequence of sheaves.

\begin{definition} A morphism of sheaves $\alpha\from{\cal F}\to{\cal G}$
on $G$ is a collection of linear maps
$\alpha_v\from{\cal F}(v)\to{\cal G}(v)$ for each $v\in V$ and
$\alpha_e\from{\cal F}(e)\to{\cal G}(e)$ for each $e\in E$ 
such that for each $e\in E$ we have
${\cal G}(t,e)\alpha_e=\alpha_{te}{\cal F}(t,e)$ and 
${\cal G}(h,e)\alpha_e=\alpha_{he}{\cal F}(h,e)$.
\end{definition}

It is not hard to check that
all Abelian
operations
on sheaves, e.g., taking kernels, taking direct sums, checking exactness,
can be done ``vertexwise and edgewise,'' i.e.,
${\cal F}_1\to{\cal F}_2\to{\cal F}_3$ is exact iff
for all $P\in V_G\amalg E_G$, we have
${\cal F}_1(P)\to{\cal F}_2(P)\to{\cal F}_3(P)$ is exact.
This is actually well known,
since our sheaves are presheaves of vector spaces on a category
(see \cite{friedman_cohomology} or Proposition~I.3.1 of
\cite{sga4.1}).

The following theorem results from a straightforward application
of classical homological algebra.
\begin{theorem}\label{th:shshort_long}
To each ``short exact sequence'' of sheaves, i.e.,
$$
0\to {\cal F}_1\to {\cal F}_2 \to {\cal F}_3 \to 0
$$
(in which the kernel of each arrow is the image of the preceding arrow),
there is a natural long exact sequence of homology groups
$$
0\to 
H_1({\cal F}_1)\to
H_1({\cal F}_2)\to
H_1({\cal F}_3)\to
H_0({\cal F}_1)\to
H_0({\cal F}_2)\to
H_0({\cal F}_3)\to
0 .
$$
\end{theorem}

\subsection{Quasi-Betti Numbers and Maximum Excess}

For any digraph, $G$, we have that the pair
$h_0,h_1$ assign non-negative integers to
each sheaf over $G$, and these integers satisfy certain properties.
In this chapter we introduce other pairs of invariants,
essentially variations of $h_0,h_1$, that satisfy the same
properties.  Our proof of the SHNC will use the fact that
the ``maximum excess'' is
part of such a pair.  Let us make these notions precise.

\begin{definition} A sequence of real numbers, $x_0,\ldots,x_r$ is
a {\em triangular sequence} if for any $i=1,\ldots,r-1$ we have
$$
x_i\le x_{i-1}+x_{i+1}.
$$
\end{definition}

\begin{definition} Given a digraph, $G$, and a field, $\field$,
consider the category of sheaves of $\field$-vector spaces on $G$.
Let
$\alpha_0,\alpha_1$ be 
two functions
from sheaves 
to the non-negative reals.
We shall say that $(\alpha_0,\alpha_1)$ is a
{\em quasi-Betti number pair (for $G$ and $\field$)} provided that:
\begin{enumerate}
\item for each sheaf, ${\cal F}$, we have
\begin{equation}\label{eq:shquasi_Euler}
\alpha_0({\cal F})-\alpha_1({\cal F}) = \chi({\cal F});
\end{equation}
\item for any sheaves, ${\cal F}_1,{\cal F}_2$ on $G$ we have
$$
\alpha_i({\cal F}_1\oplus {\cal F}_2) =
\alpha_i({\cal F}_1) +
\alpha_i({\cal F}_2) \quad\mbox{for $i=0,1$;}
$$
\item for any short exact sequence of sheaves on $G$
$$
0\to{\cal F}_1\to{\cal F}_2\to{\cal F}_3\to 0,
$$
the sequence of integers
$$
0,
\alpha_1({\cal F}_1),\alpha_1({\cal F}_2),\alpha_1({\cal F}_3),
\alpha_0({\cal F}_1),\alpha_0({\cal F}_2),\alpha_0({\cal F}_3),
0
$$
is triangular.
\end{enumerate}
Moreover, we say that a function, $\alpha$, from sheaves to non-negative
reals is a {\em first quasi-Betti number} if the pair 
$(\alpha_0,\alpha_1)$ with
$$
\alpha_1({\cal F})=\alpha({\cal F}), \quad
\alpha_0({\cal F})=\chi({\cal F})+\alpha({\cal F})
$$
\end{definition}
are quasi-Betti number pair.
The relationship between quasi-Betti numbers and a first quasi-Betti
numbers is forced by equation~(\ref{eq:shquasi_Euler}).

Notice that $(h_0,h_1)$ is a quai-Betti number pair; the only issue
in establishing this is property~(3) of the definition, and
this follows from the long exact sequence given by
Theorem~\ref{th:shshort_long}.

Of course, if $(\alpha_0,\alpha_1)$ is a quasi-Betti number pair, then
clearly $\alpha_1$ is a first quasi-Betti number.

Let us give other quasi-Betti number pairs, beginning with the one
of main interest in this paper.

\begin{definition}
\label{de:excess}
Let ${\cal F}$ be a sheaf on a digraph, $G$.
For any $U\subset {\cal F}(V)$ we define the 
{\em head/tail neighbourhood of $U$}, denoted 
$\Gamma_{\rm ht}(G,{\cal F},U)$, or simply
$\Gamma_{\rm ht}(U)$, to be
\begin{equation}\label{eq:shhead/tail}
\Gamma_{\rm ht} (U) = \bigoplus_{e\in E_G} \{ w\in{\cal F}(e) \;|\;
d_h(w),d_t(w)\in U \};
\end{equation}
we define the {\em excess of ${\cal F}$ at
$U$} to be
$$
{\rm excess}({\cal F},U) = \dim\bigl( \Gamma_{\rm ht}(U) \bigr) - \dim(U).
$$
Furthermore we define the {\em maximum excess} of ${\cal F}$ to be
$$
{\rm m.e.}({\cal F}) = \max_{U\subset{\cal F}(V_G)} {\rm excess}
({\cal F},U).
$$
\end{definition}

We shall see that the excess is a supermodular function, and hence
the maximum excess occurs on a lattice of subsets of ${\cal F}(V)$.
It is not hard to see that for the structure sheaf, $\underline\field$, we have
$$
{\rm m.e.}(\underline\field) = \rho(G).
$$
It is instructive to determine which subsets of $\underline\field(V_G)$
obtain this maximum excess of $\rho(G)$.
So
let $G'$ be obtained from $G$ by discarding all components
with positive Euler characteristic and, optionally,
discarding some components of zero Euler characteristic, and then, optionally
repeatedly pruning any of its leaves 
(i.e., removing a vertex of degree one and
its incident edge); then the excess of $U(G')=\oplus_{v\in V_{G'}}\field(v)$
of $\underline\field$ on $G$ is $\rho(G)$, and, conversely, any
subspace $U\subset\underline\field(V_G)$ achieving the maximum excess of
$\rho(G)$ is of the form $U(G')$ for a $G'$ as above.  
The reader can easily see that such $U(G')$ form a lattice (i.e., are
closed under
intersection and sum).

\begin{theorem}\label{th:shSHNC_main}
The maximum excess is a first quasi-Betti number.
\end{theorem}

Theorem~\ref{th:shSHNC_main} will be crucial to our proof of
the SHNC (although, as mentioned before, in an alternate proof we
avoid the need for this theorem).
Somewhat surprisingly, the statement of 
this theorem and all the necessary definitions
do not involve any homology theory.

We shall show Theorem~\ref{th:shSHNC_main}
by identifying the maximum excess with a certain
``limit'' Betti number.

\subsection{Twisted Homology}
\label{sb:twisted_def}
One graph theoretic reformulation of the SHNC involves the reduced
cyclicity defined in equation~(\ref{eq:shrho}).
This definition seems difficult to deal with, because of the 
$\max(0,h_1(X)-1)$ term, and of the possibility of $h_1(X)=0$ for
some components, $X$, of $G$.
For a digraph, $G$, one can realize $\rho(G)$ as a ``twisted first Betti
number;'' constructing this
``twisted homology theory'' is our first step towards showing
that the maximum excess is a first quasi-Betti number.

Let us first briefly motivate our definitions of twisted homology.
We begin by noticing that for $G$ connected we have
\begin{equation}\label{eq:shrho_limit}
\rho(G)  = \lim_{n\to\infty} h_1(L_n)/n,
\end{equation}
where for each positive integer $n$ we choose a covering
$L_n\to G$ of degree $n$ such that $L_n$ is connected (for then
$h_0(L_1)=1$ and $h_1(L_n)=h_0(L_n)-\chi(L_n)=1+n\rho(G)$).

One way of choosing $n$ and $L_n\to G$ of degree $n$ such that $L_n$ is
connected is to take $n=p$ a prime number, and take $L_p\to G$ to be
a ``generic'' $\integers/p\integers$ covering of $G$ (see 
Section~\ref{se:shgalois}).  
It is well known that for $\integers/p\integers$ coverings $G'\to G$,
or for any Abelian covering, the eigenvalues of the adjacency matrix
of $G'$ can be computed from those of $G$ after ``twisting'' appropriately;
here ``twisting'' means
multiplying the
entries of $G$'s adjacency matrix by appropriate roots of unity,
according to the characters of the ``Galois group'' of $G'$ over $G$
(see Section~\ref{se:shgalois}).
The same holds for homology groups.

This leads us to a new homology theory, as follows.
Let ${\cal F}$ be a sheaf of $\field$-vector spaces on a digraph, $G$,
and let $\field'$ be a field containing $\field$.
A {\em twist} or {\em $\field'$-twist}, $\psi$, on $G$ is a map
$$
\psi \from E_G\to \field'.
$$
By the {\em twisting} of
${\cal F}$ by $\psi$, denoted ${\cal F}^\psi$, we mean the
sheaf of $\field'$-vector spaces given via 
$$
{\cal F}^\psi(P)=\bigl( {\cal F}(P)\bigr) \otimes_{\field}\field'
$$
for all $P\in V_G
\amalg E_G$, and 
$$
{\cal F}^\psi(h,e)={\cal F}(h,e), \quad
{\cal F}^\psi(t,e)=\psi(e){\cal F}(t,e),
$$
where ${\cal F}(h,e)$ and ${\cal F}(t,e)$ are viewed as $\field'$-linear maps
arising from their original $\field$-linear maps.
In other words, ${\cal F}^\psi$ is the sheaf
on the same vector spaces extended to $\field'$-vector spaces,
but with the tail restriction maps
twisted by $\psi$.
The map, $d_{{\cal F}^\psi}$, viewed as a
matrix, has entries in the field $\field'$.
The groups $H_i({\cal F}^\psi)$ are defined as $\field'$-vector spaces.

Now let $\psi=\{\psi(e)\}_{e\in E_G}$ be 
viewed as $|E_G|$ indeterminates, and let
$\field(\psi)$ denote the field of rational functions in the $\psi(e)$
over $\field$.  Then $d=d_{{\cal F}^\psi}$ can be viewed as a morphism
of finite dimensional vector spaces over $\field(\psi)$, given by a matrix
with entries in $\field(\psi)$.
\begin{definition}
We define the {\em $i$-th twisted
homology group of ${\cal F}$}, denoted by
$$
H_i^\twist({\cal F})=H_i^\twist({\cal F},\psi),
$$
for $i=0,1$, respectively, to be the cokernel and kernel, respectively,
of $d_{{\cal F}^\psi}$ described above as a morphism of $\field(\psi)$
vector spaces.  
We define
the
{\em $i$-th twisted Betti number of ${\cal F}$}, denoted
$h_i^{\rm twist}({\cal F})$, to be dimension of 
$H_i^\twist({\cal F})$.
\end{definition}

We easily see, akin to equation~(\ref{eq:shrho_limit}), that 
$$
\rho(G) = h_1^{\rm twist}(\underline\field).
$$
The analogous short/long exact sequences theorem holds in twisted homology,
and this easily implies that
$h_1^{\rm twist}$ is a quasi-Betti number.
We wish to mention that we can interpret
$$
h_0^{\rm twist}(\underline\field)=\chi(\underline\field)+
h_1^{\rm twist}(\underline\field)=\chi(G)+\rho(G)
$$
as the number of ``acyclic components'' of $G$, i.e., the number of connected
components that are free of cycles.

\subsection{Maximum Excess Versus Twisted Betti Numbers, and The Unhappy
$4$-Bundle}
\label{sb:unhappy}

Note that for the constant sheaf, $\underline\field$, on a digraph, $G$,
the values of $h_1^{\rm twist}$
and the maximum excess agree and equal $\rho(G)$.  Notice also that
it is immediate that $h_1^{\rm twist}$ is a first quasi-Betti number,
but it seems to us more difficult to show that the maximum excess is
a first quasi-Betti number.
This indicates that it would be easier to work with $h_1^{\rm twist}$
rather than the maximum excess in studying the SHNC (and this can be
done).
We give two reasons why we nonetheless use the maximum excess.

First, the SHNC is more directly related to the vanishing maximum
excess of a certain sheaves we call $\rho$-kernels; and this vanishing
is weaker (at least {\em a priori}) than the vanishing of
$h_1^{\rm twist}$ of the $\rho$-kernels.  
Second, the Euler characteristic, reduced cyclicity,
and the maximum excess have a nice
scaling property under ``pullbacks'' via covering maps, that
$h_1^\twist$ does not share.  This makes $h_1^\twist$ seem to be,
at times, the ``wrong'' invariant for certain situations, like those
arising in the SHNC.

Let us discuss the above remarks in more precise terms.
It is easy to see that
$$
h_1^{\rm twist}({\cal F})\ge {\rm m.e.}({\cal F}),
$$
and one can
show that equality holds if for each $e\in E_G$, ${\cal F}(e)$ is either
zero or one dimensional.  In particular, this holds for 
${\cal F}=\complex_L$ for any subgraph, $L$ of $G$. 
However, there are sheaves, such as the ``unhappy 4-bundle,'' that
we will soon describe,
which have maximum excess zero but positive $h_1^\twist$.
The above inequality does show that if $h_1^\twist$ vanishes then so
does the maximum excess; in the case of the SHNC and $\rho$-kernels
this means that vanishing $h_1^\twist$ of $\rho$-kernels is at least
as strong a condition as the SHNC.

We now describe a sheaf we call the {\em unhappy $4$-bundle}.  It is
a highly instructive
example that illustrates a number of points on maximum
excess and twisted homology.
Let $B_2$ be the bouquet of two self-loops, 
i.e., the digraph with one vertex, $v$,
and two self-loops, $e_1,e_2$.
Let ${\cal U}$ be defined as
\begin{equation}\label{eq:shunhappy_values}
{\cal U}(v) = \field^4, \quad {\cal U}(e_i) = \field^2 \quad
\mbox{for $i=1,2$},
\end{equation}
and
\begin{equation}\label{eq:shunhappy_restrictions}
d_h = 
\left[
\begin{matrix} 1&0&1&0 \\ 0&1&0&0 \\ 0&0&0&1 \\ 0&0&0&0 \end{matrix}
\right]
, \quad
d_t = 
\left[
\begin{matrix} 0&0&0&0 \\ 0&0&1&0 \\ 1&0&0&0 \\ 0&1&0&1 \end{matrix}
\right],
\end{equation}
where these matrices multiply 
the coordinates of ${\cal U}(E)$ arranged as
a column vector (the column vector to the right of the matrix),
where ${\cal U}(E)$'s coordinates are  ordered as 
${\cal U}(e_1)\oplus{\cal U}(e_2)$.
The twisted incidence matrix of ${\cal U}$
(which characterizes ${\cal U}$) is given by
\begin{equation}\label{eq:shunhappy_twist}
d_{{\cal U}^{\psi}} =
\left[
\begin{matrix} 1&0&1&0 \\ 0&1&-\psi(e_2)&0 \\ 
-\psi(e_1)&0&0&1 \\ 0&-\psi(e_1)&0&-\psi(e_2) \end{matrix}
\right].
\end{equation}
This matrix has a kernel of dimension one in $\field(\psi)$, however
its maximum excess is zero. Equivalently, 
if ${\cal F}(v)=\field^4$ has $\alpha,\beta,\gamma,\delta$ as its
standard basis (i.e., $\alpha=(1,0,0,0)$, $\beta=(0,1,0,0)$, etc.), then
the image of the four standard coordinates on ${\cal F}(E)$ via 
$d_{{\cal U}^\psi}$ is
\begin{equation}\label{eq:shfour_vectors}
\nu_1= \alpha - \psi(e_1)\gamma, \quad \nu_2=\beta-\psi(e_1)\delta,
\quad \nu_3=\alpha-\psi(e_2)\beta, \quad\nu_4=\gamma-\psi(e_2)\delta.
\end{equation}
The fact that $h_1^\twist({\cal U})\ne 0$ follows from the simple 
computation that
$$
\nu_1\wedge\nu_2\wedge\nu_3\wedge\nu_4=0
$$
or the linear dependence relation
$$
\nu_1-\psi(e_2)\nu_2 - \nu_3 + \psi(e_1)\nu_4
=0
$$
The reason we call ${\cal U}$ a $4$-bundle is that is four dimensional at
the vertex of $B_2$, and it is has properties akin to a vector bundle;
this will be explained more fully in a sequel to this paper.

For any sheaf, ${\cal F}$, on a digraph, $G$, and 
any morphism $\phi\from K\to G$ of directed graphs, we define the
{\em pullback of ${\cal F}$ via $\phi$} to be the
sheaf $\phi^*{\cal F}$ on $K$ given via
$$
(\phi^*{\cal F})(P) = {\cal F}(\phi(P)) \qquad
\mbox{for all $P\in V_K\amalg E_K$},
$$
and for all $e\in E_K$,
$$
(\phi^*{\cal F})(h,e) = {\cal F}(h,\phi(e)), \quad
(\phi^*{\cal F})(t,e) = {\cal F}(t,\phi(e)).
$$
It is easy to see that
if $\mu$ is a covering map of degree $\deg(\mu)$ then
$$
\chi(\mu^*{\cal F}) = \deg(\mu) \chi({\cal F}),
$$
and, with a little more work (and using Galois graph theory, oddly
enough), that 
\begin{equation}\label{eq:shme_scale}
{\rm m.e.}(\mu^*{\cal F}) = \deg(\mu) {\rm m.e.}({\cal F}).
\end{equation}
The ``unhappy 4-bundle'' also shows that $h_1^\twist$ does not
enjoy this ``scaling by $\deg(\mu)$ under pullback''
property.  Indeed, $h_1^\twist({\cal U})=1$;
however
if $\phi\from G'\to B_2$ (recall ${\cal U}$ is defined on the graph $B_2$)
is the degree two cover of $B_2$ in which the $G'$ edges
mapping to $e_1$ are self-loops, and the edges mapping to $e_2$ are
not, then $h_1^\twist(\phi^*{\cal U})=0$.  In other words, via taking
wedge products or solving for a linear relation, it is 
straightforward to verify
the linear independence of the eight vectors
$$
\nu_1^1= \alpha^1 - \psi(e_1^1)\gamma^1, 
\ \nu_2^1=\beta^1-\psi(e_1^1)\delta^1,
\ \nu_3^1=\alpha^1-\psi(e_2^1)\beta^2, 
\ \nu_4^1=\gamma^1-\psi(e_2^1)\delta^2.
$$
$$
\nu_1^2= \alpha^2 - \psi(e_1^2)\gamma^2, 
\ \nu_2^2=\beta^2-\psi(e_1^2)\delta^2,
\ \nu_3^2=\alpha^2-\psi(e_2^2)\beta^1, 
\ \nu_4^2=\gamma^2-\psi(e_2^2)\delta^1.
$$

\subsection{The Fundamental Lemma and Limit Homology}

The following is the main and most difficult
theorem in this chapter; it allows us
to connect twisted homology and maximum excess.
For any digraph we shall define the notion of its {\em Abelian girth},
which is always at least as large as its girth.

\begin{theorem}\label{th:shmain}\label{TH:SHMAIN}
For any sheaf, ${\cal F}$, on a digraph, $G$, let
$\mu\from G'\to G$ be a covering map where the Abelian girth
of $G'$ is at
least 
$$
2 \Bigl( \dim\bigl({\cal F}(V)\bigr) + 
\dim\bigl({\cal F}(E)\bigr) \Bigr) + 1.
$$
Then
$$
h_1^\twist(\mu^*{\cal F}) = {\rm m.e.}(\mu^*{\cal F}).
$$
\end{theorem}

From this lemma it is easy to see that the maximum excess
is a first quasi-Betti number.

\subsection{Limits and Limiting Betti Numbers}

In this subsection we give a new interpretation to our main theorem,
Theorem~\ref{th:shmain}.
For any two covering maps,
$$
\phi_1\from G_1\to G \quad\mbox{and}\quad 
\phi_2\from G_2\to G ,
$$
their fibre product
$$
\phi\from G_1\times_G G_2 \to G
$$
factors through both $\phi_1$ and $\phi_2$, i.e., $\phi$ is a
``common cover.''  It follows that the set, ${\rm cov}(G)$,
of covering maps of a
fixed digraph, $G$, is a directed set, under the partial order
$\phi_1\le\phi_2$ if $\phi_2$ factors through $\phi_1$.  
As such we may speak of limits
in the usual sense of limits of a directed sets; i.e., if $f$ is, say,
a real-valued function on covering maps, then we write
$$
\lim_{\phi\in{\rm cov}(G)} f(\phi) = L
$$
if for any $\epsilon>0$ there is a $\phi_\epsilon\in{\rm cov(G)}$ such
that $|f(\phi')-L|\le \epsilon$ provided that $\phi'$ factors through
$\phi_\epsilon$ (such a limit, $L$, is necessarily unique).

Theorem~\ref{th:shmain} implies that for any sheaf, ${\cal F}$,
on $G$, we have
$$
{\rm m.e.}({\cal F}) = \lim_{\phi\in{\rm cov}(G)} 
\frac{ h_1^\twist(\phi^*{\cal F})}{\deg(\phi)}.
$$
Of course, Theorem~\ref{th:shmain} amounts to saying that this limiting
value
is exactly
attained at any $\phi\from G'\to G$ with $G'$ of sufficiently large 
girth or Abelian girth.

For a sheaf, ${\cal F}$, on a digraph, $G$, we define
$$
\lim_{\phi\in{\rm cov}(G)}
\frac{ h_i^\twist({\cal F})}{\deg(\phi)}
$$
to be the $i$-th limiting Betti number, which we denote
$h_i^{\rm lim}({\cal F})$.  Evidently,
$$
h_1^{\rm lim}({\cal F}) = {\rm m.e.}({\cal F}), \quad
h_0^{\rm lim}({\cal F}) = \chi({\cal F}) + {\rm m.e.}({\cal F}).
$$
It is easy to see that the limit of quasi-Betti pairs is also a
quasi-Betti pair, and that for any fixed covering map
$\phi\from G'\to G$, the functions for $i=0,1$ given by
$$
h_i^{\rm twist}(\phi^*{\cal F})/\deg(\phi)
$$
form a quasi-Betti pair.  This is another way of saying that
Theorem~\ref{th:shmain} implies
Theorem~\ref{th:shSHNC_main}.

%
%
%
 
\subsection{Sheaves, Adjacency Matrices, and Laplacians}

We remark that from the incidence matrix, $d_{\cal F}=d_h-d_t$, of
a sheaf, ${\cal F}$, one
can define a Laplacians, adjacency matrices, and related matrices
that are analogues of those used for graphs.
This construction can also be viewed as a very special, discrete case
of Hodge theory.
We require that for each $P\in V_G\amalg E_G$, we have that each ${\cal F}(P)$
be endowed with an inner product.  In that way ${\cal F}(V),{\cal F}(E)$
become inner product spaces, and we have adjoint operators
$d_h^*,d_t^*$ and $d^*=d_h^*-d_t^*$ from ${\cal F}(V)$ to ${\cal F}(E)$.
We define
$$
\Delta_0 = d d^*,\quad \Delta_1=d^* d
$$
to be the {\em Laplacians} of ${\cal F}$, which, of course, depend
on the inner products chosen for the values, ${\cal F}(P)$, of
${\cal F}$; we easily see
that $\Delta_i$ is an operator on ${\cal F}(V)$
and ${\cal F}(E)$ respectively for $i=0$ and $i=1$ respectively;
if $\field$ is of characteristic zero, then the $\Delta_i$ are positive
semi-definite operators, and 
the kernel of $\Delta_i$ is $H_i({\cal F})$.
In the special case ${\cal F}=\underline{\field}$, with the same,
standard inner
products on all ${\cal F}(P)=\field$, the Laplacians become
the usual Laplacians of the graph.

Furthermore, given ${\cal F}$ and inner products on the values of ${\cal F}$,
we get generalizations of the adjacency matrix and degree matrix.
For example, if we set
$$
D_0 = d_h d_h^* + d_t d_t^*,\quad
A_0 = d_h d_t^* + d_t d_h^* ,
$$
we have that $\Delta_0 = D_0 - A_0$;  
in the case ${\cal F}=\underline{\field}$ and standard inner products,
$D_0,A_0$, respectively amount to the usual degree and adjacency matrices,
respectively.  One can define $D_1,A_1$ analogously.

One could define a sheaf to be {\em regular} in the way that one would 
define a graph to be regular, i.e., if both $D_0$ and $D_1$ are both
multiples of the identity.  One could measure the {\em expansion} of
a sheaf by the eigenvalues of $A_0,A_1$ or $\Delta_0,\Delta_1$.

We believe that the spectral theory of such matrices and related properties
such as expansion could be quite interesting to pursue.
However, we shall not pursue them further in this paper.

\section{Galois and Covering Theory} 
\label{se:shgalois}

In this section we establish a number of important definitions and
facts concerning
graph coverings, Abelian coverings, and Galois coverings.

There is a collection of facts about number fields that may be called
Galois theory; this would include classical Galois theory, but
also more recent statements such as if $k'$ is a Galois extension
field of $k$, then
$$
k'\otimes_k k' \isom \bigoplus_{{\rm Aut}(k'/k)} k'
$$
(see \cite{sga4.5}, Section~I.5.1).
Such facts have analogues in graph theory, which one might call
``graph Galois theory.'' 
Such facts were described in \cite{friedman_geometric_aspects,
st1};
at least some of these some of these facts were known much earlier, in
\cite{gross}; since these facts are fairly simple and quite powerful,
we presume they may occur elsewhere in the literature (perhaps only
implicitly).

\subsection{Galois Theory of Graphs}

We shall summarize some theorems of \cite{friedman_geometric_aspects};
the reader is referred to there and \cite{st1} for more
discussion.
In this article Galois group actions, when written
multiplicatively (i.e., not viewed as functions or morphisms)
will be written
on the right, since our Cayley
graphs are written with its generators acting on the left.

Let $\pi\from K\to G$ be a covering map of digraphs.
We write ${\rm Aut}(\pi)$, or somewhat abusively
${\rm Aut}(K/G)$ (when $\pi$ is understood),
for the automorphisms of $K$ over $G$, i.e.,
the digraph automorphisms $\nu\from K\to K$ such that $\pi=\pi\nu$.

Now assume that $K$ and $G$ are connected.  Then it is easy to see
(\cite{friedman_geometric_aspects,st1}) that for every $v_1,v_2\in V_K$
there is at most one $\nu\in{\rm Aut}(K/G)$ such that $\nu(v_1)=v_2$;
the same holds with edges instead of vertices.
It follows that $|{\rm Aut}(K/G)|\le [K:G]$, with equality iff
${\rm Aut}(K/G)$ acts transitively on each vertex and edge fibre of
$\pi$.  In this case we say that $\pi$ is {\em Galois}.

If $\pi\from K\to G$ is Galois but $K$ is not connected, $|{\rm Aut}(K/G)|$
can be as large as $[K\colon G]$ factorial (if $K$ is a number of copies of
$G$).
So when $K$ is not connected, we say that a covering map $\pi\from K\to G$
is {\em Galois} provided that we additionally specify a subgroup, $\cg$,
of ${\rm Aut}(K/G)$ of that acts simply (without fixed points) and
transitively on each of the
vertex and edge fibres of $\pi$; we declare $\cg$ to be the Galois group.
Again, this additional specification does not change any of the theorems
here, although it does mean that certain $\pi\from K\to G$ can be 
Galois on each component of $G$ without being Galois in our sense
(consider $G=G_1\amalg G_2$, and $K_i=\pi^{-1}(G_i)$, where $G_1,G_2$
are connected and ${\rm Aut}(K_i/G_i)$ are non-isomorphic groups).

\begin{theorem}[Normal Extension Theorem]
\label{th:shnormal_extension_theorem}
If $\pi\from G\to B$ is a covering map of digraphs, there 
is a covering map $\mu\from K\to G$ such that $\pi\mu$ is Galois.
\end{theorem}
In this situation we say that $K$ is a {\em normal extension} of
$G$ (assuming the maps $\mu$ and $\pi$ are understood).
By convention, all graphs are finite in the paper unless otherwise specified.
Generally speaking, we will not address infinite graphs in the context
of Galois theory; however, if the $\pi\from G\to B$ in this theorem
is a morphism of finite degree, even if $G$ and $B$ are infinite digraphs,
then the proof of the Normal Extension Theorem due to Gross is still
valid.

Let us outline two proofs of the Normal Extension Theorem.
The proof in \cite{friedman_geometric_aspects} uses the fact that
$G$ corresponds to a subgroup, $S$, of index $n=|V_G|$ of the
group $\pi_1(B)$, the fundamental group
of $B$ (which is the free group on $h_1(B)$ elements).  The intersection
of $xSx^{-1}$ over a set of coset representatives of $\pi_1(B)/S$
is a normal subgroup, $N$, of finite size (at worst $n^n$, since there are
$n$ cosets and each $xSx^{-1}$ is of index $n$); $\pi(B)/N$ then
naturally corresponds to a Galois cover $K\to B$ of at most $n^n$
vertices.

There is a very pretty proof of the Normal Extension Theorem
discovered earlier by Jonathan Gross 
in \cite{gross}, giving a better bound on the
number of vertices of $K$.  For any positive integer
$k$ at most $n=|V_G|$, let $\Omega^k(G)$ be the subgraph of
$G\times_{B}G\times_{B}\cdots\times_{B}G$ (multiplied $k$ times)
induced on the set of vertices of the form $(v_1,\ldots,v_k)$ where
$v_i\ne v_j$ for all $i,j$ with $i\ne j$.  Each $\Omega^k(G)$ 
admits a covering map to
$G$ by
projecting onto any one of its components.  But $\Omega^n(G)$, which
has edge and vertex fibers of size $n!$, 
is Galois by the natural, transitive action of $S_n$
(the symmetric
group on $n$ elements) on $\Omega^n(G)$.  So $\Omega^n(G)$ is a Galois
cover of degree at most $n!$ over $B$.

\subsection{Galois Coordinates}
\label{sb:galois_coordinates}

Given a graph, $G$, and a group, $\cg$, consider the task of
describing all
Galois covering maps $\pi\from K\to G$ with Galois group $\cg$;
consider also the task of giving
a meaning to a ``random'' such Galois covering
(i.e., describe a natural probability space whose atoms are such
coverings).
This can be done in a number of ways, via Galois coordinates or the
monodromy map.  Here we shall review these ideas and apply them.
These ideas occur (in parts) in many places in the literature;
see, for example, 
\cite{friedman_alon,friedman_relative,amit_random,friedman_geometric_aspects}.

Again, fix a graph, $G$, and a group, $\cg$.  By {\em Galois coordinates
on $G$ with values in $\cg$}
we mean a choice of $a_e\in \cg$ for each $e\in E_G$.  From the $\{a_e\}$
we build a covering map $\phi\from K\to G$ by taking $V_K=V_G\times \cg$
and taking $E_K=E_G\times \cg$ with the head and tail, respectively, of
an edge $(e,a)$ being 
\begin{equation}\label{eq:shaes}
h_K(e,a)=(h_G e,a_e a ),\quad t_K(e,a)=(t_G e, a),
\end{equation}
respectively.
We define a $\cg$ action on $K$ via 
$g\in \cg$ is the morphism
such that for $P\in V_G\amalg E_G$ and $a\in\cg$, $g$ sends $(P,a)$ to
\begin{equation}\label{eq:shgonK}
(P,a)g = (P,ag);
\end{equation}
in view of the fact that $a_e$ multiplies to the left in 
equation~(\ref{eq:shaes}), we see that the right multiplication
of $g$ on $a$ in equation~(\ref{eq:shgonK}) actually defines a digraph
morphism.
Let $\phi$ be projection onto the first coordinate.
Clearly $\phi$ is a Galois covering with Galois group $\cg$.

Conversely, let $\phi\from K\to G$ be any $\cg$ Galois covering.
We may identify $V_K$ with $V_G\times\cg$ by choosing for each
$v\in V_G$ an element $v'\in V_K$ such that $\phi(v')=v$ and declaring
$v'$ to have coordinates 
$(v,1)$ where $1$ is the identity in $\cg$; we say that $v'$ is the
{\em origin for $v$ in $K$}; then for all
$v''\in V_K$ with $\phi(v'')=v$ there is a unique $g\in\cg$ with
$v''=v'g$, and we declare $v''$ to have coordinates $(v,g)$.  
For any $g'\in\cg$ we have $v''g'=vgg'$ which has coordinates $(v,gg')$;
hence $g'$ acts on coordinates by right multiplication.
Now choose an edge $e'\in E_K$, and let $e=\phi(e')$; there exist
unique $a_{e'},g\in \cg$ for which the endpoints of $e'$ have coordinates
$$
te'=(te,g),\quad he'=(he,a_{e'}g).
$$
But the $\cg$ action on $K$ then shows that for any $g'$ we have
$$
t(e'g')=(te,gg'),\quad h(e'g')=(he,a_{e'}gg').
$$
It follows that $a_{e'}$ depends only on $e=\phi(e')$, i.e.,
$a_{e'}=a_{e'g'}$ for all $e'\in K$ and $g'\in \cg$.
In other words, there is a unique $a_e$ for each $e\in E_G$ such 
that the $\phi$ fibres of $e$
join $(t_G e,g)$ to $(h_G e,a_e g)$ for each $g\in \cg$.
In summary, for each choice of an element in the vertex fibres we get
Galois coordinates (and conversely).

Notice that in setting the coordinates on $V_K$, if for $v\in V_G$
we choose a different origin, namely $v'g_v$ instead of $v'$, then 
we have $v'g=(v' g_v)g_v^{-1}g$ for any $g\in\cg$; it follows that the
vertx $v'g$, which would have had coordinates $(v,g)$ with $v'$ as origin,
will have coordinates $(v, g_v^{-1}g)$ with $v'g$ as origin.  In particular, if
for $e'\in V_K$ and $e=\phi(e)$ we have
$te'=(te,g)$ and $he'=(he,a_e g)$ in one set of coordinates for some
$g$, and the
origins of $te$ and $he$ are respectively translated by $g_{te}$ and $g_{he}$,
then in the new coordinates
$$
te' = (te, g_{te}^{-1}g), \quad
he' = (he, g_{he}^{-1}a_e g).
$$
Setting $g'=g_{te}^{-1}g$, it follows that in the new, translated 
coordinates we have $te'=(te,g')$ and $he'=(he,\mt a_e g')$, where
$$
\mt a_e = g_{he}^{-1} a_e g_{te}.
$$
So changing Galois coordinate origins as such amounts to a transformation
of Galois coordinates
\begin{equation}\label{eq:shcoord_change}
a_e \mapsto \mt a_e = g_{he}^{-1} a_e g_{te}
\end{equation}
for a family $\{g_v\}_{v\in V_G}$ of $\cg$ values indexed on $V_G$.

Galois coordinates give a nice model of a random Galois cover of a given
graph with given Galois group---just choose the each $a_e$ uniformly in
$\cg$, assuming $\cg$ is finite, and independently over the $e\in E_G$.
If one wants a model of a random cover, one that is not Galois, one often
chooses $V_K$ to have vertices $V_G\times\{1,\ldots,n\}$, where $n$ is 
the degree of the cover, and chooses random matchings over each $G$ edge
(random permutations over self-loops); see, e.g.,
\cite{friedman_alon,friedman_relative,amit_random}.

\subsection{Walks and Monodromy}
\label{sb:walks_monodromy}
Another type of coordinates for Galois coverings are the monodromy maps.
For this we need to fix some notation regarding walks in a digraph.

\begin{definition} Let $G$ be a digraph.  
By an {\em oriented edge of $G$} we mean a formal symbol
$e^+$ or $e^-$ where $e\in E_G$.  We extend the head and tail map
to oriented edges via $he^+=te^-=he$ and $te^+=he^-=te$.
We say that the {\em inverse} of $e^+$ is $e^-$ and vice versa.
An {\em undirected walk} (or simply {\em walk}) in $G$
is an alternating sequence of vertices and oriented edges
$w=(v_0,f_1,v_1,f_2,v_2,\ldots, f_r,v_r)$ with $hf_i=v_i$, $tf_i=v_{i-1}$
for $i=1,\ldots,r$; we call $r$ its {\em length};
we say that $w$ is {\em closed} if $v_r=v_0$;
we say that $w$ is {\em non-backtracking} or {\em reduced}
if for each $i=1,\ldots,r-1$, $f_i$ and $f_{i+1}$ are not inverses of
each other.
\end{definition}

If $G$ is a digraph and $v\in V_G$, then we define $\pi_1(G,v)$ to
be the group of non-backtracking closed walks about $v$, where 
the group operation is concatenation of walks (which we reduce
until they are non-backtracking).  This, of course,
is isomorphic to
the usual fundamental group, $\pi_1(\mt G,v)$, where $\mt G$ is the
geometric realization of $G$, where vertices of $G$ correspond to points and 
edges of $G$ correspond to unit
intervals.
If $G$ is connected, then $\pi_1(G,v)$ is a free group on $h_1(G)$
generators.
We may also describe $\pi_1(G,v)$ as the classes of closed walks about
$v$, where two walks are equivalent if they reduce to the same
non-backtracking word (``reduce'' meaning repeatedly eliminating any
two consecutive steps of the walk that traverse an edge and then its inverse).

Let $\phi\from G'\to G$ be Galois with Galois
group $\cg$, with $G$ connected,
and let $\{a_e\}$ be Galois coordinates for $\phi$.  Extend the $\{a_e\}$
to be defined on oriented edges via
$a_{e^+}=a_e$, $a_{e^-}=a_e^{-1}$.
Fix a $v\in V_G$.  Then for any closed walk, $w$, about $v$ in $G$, we
let $e_i$ be the oriented edge traversed by $w$ on the $i$-th step and set
$$
{\rm Mndrmy}_{\phi,\{a_e\}}(w) = a_{e_k}\ldots a_{e_1},
$$
where $\{a_e\}_{e\in E_G}$ are Galois coordinates on $\phi$.  
We call ${\rm Mndrmy}_{\phi,\{a_e\}}$ the {\em monodromy} map with respect
to $\{a_e\}$; it is a group morphism from $\pi_1(G,v)$ to ${\cal G}$.
Conversely, given a group morphism
$$
M\from \pi_1(G,v)\to \cg
$$
with $G$ connected, 
we can form a covering $\phi\from G'\to G$ with Galois coordinates
$\{a_e\}$ such that ${\rm Mndrmy}_{\phi,\{a_e\}}=M$; indeed,
we let $T$ be an undirected spanning tree for $G$, define $a_e=1$
for $e\in E_T$ (where $1$ denotes the identity in $\cg$), and
define $a_e$ for $e\in E_G\setminus E_T$ by taking an element 
$\gamma\in \pi_1(G,v)$ composed entirely of $E_T$ edges except
for one edge $e$ (traversed in the same orientation as $e$) and
set $a_e=M(e)$; since $\pi_1(G,v)$ is a free group on $E_G\setminus E_T$,
this implies that $M$ is well-defined and equals
${\rm Mndrmy}_{\phi,\{a_e\}}$.

If we change
Galois coordinates on $\phi$, then according to 
equation~(\ref{eq:shcoord_change}) we get a conjugate element.  
Hence there is a natural map:
$$
{\rm Mndrmy}_\phi\from \pi_1(G,v) \to {\rm ConjClass}(\cg).
$$
If $v'\in V_G$ has a path, $p$, to $v$, then
the map $\gamma\mapsto p\gamma p^{-1}$ gives a homomorphism
$\pi_1(G,v)\to\pi_1(G,v')$, and the two monodromy maps, respectively,
send $\gamma$ and $p\gamma p^{-1}$ to the same conjugacy class; hence
we get a map
$$
{\rm Mndrmy}_\phi\from \pi_1(G)\to {\rm ConjClass}(\cg)
$$
independent of the base point (for $G$ connected).
Any notion defined on conjugacy classes of $\cg$ becomes defined on
$\pi_1(G)$ via monodromy.  For example, if $\cg$ is Abelian, then
the conjugacy classes of $\cg$ are the same as $\cg$, and we get a
homomorphism
$$
{\rm Mndrmy}_\phi\from \pi_1(G)\to \ca,
$$
for any cover $\phi\from G'\to G$ with Abelian Galois group
$\ca$
(compare this to the discussion of torsors in Section~5.2 of
\cite{friedman_geometric_aspects}).
We remark that if the monodromy map is onto $\ca$, and $G$ is connected
then $G'$ is connected; indeed, this means that any two vertices in the
same fiber are connected, since any vertex in $G'$ has a path to
a vertex in any vertex fibre (lifted from the element of 
$\pi_1(G)$ that maps to the appropriate element of $\ca$); 
hence we can connect any two vertices via a
path.

\subsection{Covering maps and $\rho$}

Here we describe
a remarkable property of $\rho$ under covering maps.
\begin{theorem} 
\label{th:shremarkable_rho}
For any covering map $\pi\from K\to G$ of degree $d$,
we have $\chi(K)=d\chi(G)$ and $\rho(K)=d\rho(G)$.
\end{theorem}
\begin{proof} The claim on $\chi$ follows since $d=|V_K|/|V_G|=|E_K|/|E_G|$.
To show the claim on $\rho$, it suffices to consider the case of $G$
connected, the general case obtained by summing over connected components; 
but similarly it suffices to consider the case of $K$ connected.
In this case
$$
\rho(G)=h_1(G)-1 = -\chi(G) = -d\chi(K) = d\bigl(h_1(K)-1\bigr)=d\rho(K).
$$
\end{proof}

\section{Sheaf Theory and Homology}
\label{se:shsheaf}

In this section we define sheaves of vector spaces over a graph, $G$,
and their homology groups, and give their basic properties.  
Then we explain the definitions and properties in terms of sheaf
theory on Grothendieck topologies; in case $G$ has no self-loops,
we describe a topological space, ${\rm Top}(G)$, whose sheaves give an
equivalent description of our notion of sheaf.

In the first subsection we describe everything in simple terms,
giving some claims without proof; the reader can either prove them
from scratch, or wait until the second subsection where we explain
that all of these claims are special cases of well-known results.

\subsection{Homology and Pullbacks}
\label{sb:pullbacks}

The basic definitions of sheaves were given in 
Subsection~\ref{sb:sheaf_def}.  In this subsection we 
prove Theorem~\ref{th:shshort_long} and discuss pullbacks and related
functors.

\begin{proof}[Proof (of Theorem~\ref{th:shshort_long})]
By the ``vertexwise and edgewise'' nature of taking images and
kernels, we see that we have a diagram
\begin{diagram}[nohug,height=2em,width=3em,tight]
&& 0 && 0 && 0 && \\
&& \dTo && \dTo && \dTo && \\
0 & \rTo & {\cal F}_1(E) & \rTo & {\cal F}_2(E) & \rTo & {\cal F}_3(E)
& \rTo 0 \\
&& \dTo && \dTo && \dTo && \\
0 & \rTo & {\cal F}_1(V) & \rTo & {\cal F}_2(V) & \rTo & {\cal F}_3(V)
& \rTo 0 \\
&& \dTo && \dTo && \dTo && \\
&& 0 && 0 && 0 && \\
\end{diagram}
The theorem follows from the standard ``delta'' or ``connecting'' map
in homological algebra, via the ``snake lemma'' 
(see \cite{lang,atiyah_book,hilton}).
\end{proof}

Next we describe the functoriality of sheaves.
For any sheaf, ${\cal F}$, on a graph, $G$, and 
any morphism $\phi\from K\to G$ of directed graphs, recall from
Subsection~\ref{sb:unhappy} that the
``pullback'' sheaf $\phi^*{\cal F}$ on $K$ is defined via
$$
(\phi^*{\cal F})(P) = {\cal F}(\phi(P)) \qquad
\mbox{for all $P\in V_K\amalg E_K$},
$$
and for all $e\in E_K$,
$$
(\phi^*{\cal F})(h,e) = {\cal F}(h,\phi(e)), \quad
(\phi^*{\cal F})(t,e) = {\cal F}(t,\phi(e)).
$$

If ${\cal F}$ is a sheaf on $G$ and $K$ is a subgraph of $G$, then
there is a sheaf on $G$ denoted ${\cal F}_K$ called ``${\cal F}$ restricted
to $K$ and extended by zero,'' defined by $({\cal F}_K)(P)$ is $0$ if
$P\notin V_K\amalg E_K$, and otherwise ${\cal F}(P)$; the restriction maps
are inherited from ${\cal F}$ (when $0$ is not involved).
Notice that in case ${\cal F}=\underline{\field}$, then we have
\begin{equation}\label{eq:shexplicit_adjoint}
\underline{\field}_K(V_G)=\field^{V_K},\quad
\underline{\field}_K(E_G)=\field^{E_K},
\end{equation}
and $d=d_h-d_t$ is the standard incidence matrix of $K$; hence
$H_i(\underline{\field}_K)\isom H_i(K)$.

If $\phi\from K\to G$ is an arbitrary map, and ${\cal F}$ a sheaf on
$K$, there is a natural sheaf $\phi_!{\cal F}$ on $G$ defined as follows:
$$
(\phi_!{\cal F})(P)= \bigoplus_{Q\in\phi^{-1}(P)} {\cal F}(Q),\quad
\forall P\in V_G\amalg E_G,
$$
with the restriction maps induced from those of ${\cal F}$, i.e.,
$(\phi_!{\cal F})(h,e)$ is the sum of the maps taking, for $e'\in
\phi^{-1}(e)$, the
${\cal F}(e')$ component of $(\phi_!{\cal F})(e)$ to the
${\cal F}(he')$ component of $(\phi_!{\cal F})(he)$ via the map
${\cal F}(h,e')$.
The reader can now observe that
\begin{equation}\label{eq:shshriek}
(\phi_!{\cal F})(V_G) \isom {\cal F}(V_K),\quad 
(\phi_!{\cal F})(E_G) \isom {\cal F}(E_K),
\end{equation}
and $d_{\phi_!{\cal F}}$ is the same map as $d_{\cal F}$ modulo
these isomorphisms; hence
\begin{equation}\label{eq:shspecial_adjoint}
H_i(\phi_!{\cal F}) \isom H_i({\cal F})
\end{equation}
for $i=0,1$.
In Subsection~\ref{sb:adjoint} we prove that $\phi_!$ is the left
adjoint of $\phi^*$, and in particular the isomorphisms of homology groups
above
are immediate;
in Subsection~\ref{sb:vanishing} we explain the role of $\phi_!$ in
certain ``vanishing theorems'' (of sheaf invariants).
We shall make special use of $\phi_!$ for $\phi$ \'etale in our approach
to the SHNC (see Theorems~\ref{th:etale_contagious} and
\ref{th:motivate}).

If $\phi\from K\to G$ is the inclusion of a subgraph, and ${\cal F}$
is a sheaf on $G$, then ${\cal F}_K$, defined before, equals
$\phi_!\phi^*{\cal F}$.  More generally we write ${\cal F}_K$ for
$\phi_!\phi^*{\cal F}$ for arbitrary $\phi$, provided that $\phi$
is understood in context.  
Since $\phi^*\underline{\field}=\underline{\field}$ for arbitrary
$\phi$, we always have $\underline{\field}_K = \phi_!\underline{\field}$.
This observation, combined with equation~(\ref{eq:shspecial_adjoint}),
gives another proof that
$H_i(\underline{\field}_K)$ is canonically isomorphic to $H_i(K)$ for
$i=0,1$; this proof, based on adjoints,
is less explicit than the proof based on 
equation~(\ref{eq:shexplicit_adjoint}) and the remarks just below it.

The tensor product of two sheaves on $G$ is defined as the tensor product
their
values at each point and each vertex, as $\field$-vector spaces.
Note that if $\phi\from K\to G$ is an arbitrary morphism of digraphs,
and ${\cal F}$ is a sheaf on $G$, we easily verify that
$$
{\cal F}_K = \phi_!\phi^*{\cal F} = {\cal F}\otimes \underline\field_K,
$$
and if $K'\to G$ is another morphism we have an isomophism of sheaves on
$G$
\begin{equation}\label{eq:tensor_product}
\underline\field_{K} \otimes \underline\field_{K'} \isom
\underline\field_{K\times_G K'}.
\end{equation}
Furthermore,
if $L\to G$ is an arbitrary digraph morphism, we have an equality
of sheaves on $K$,
$$
\phi^*\underline\field_L = \underline\field_{K\times_G L}.
$$

If $G'\subset G$, then there is a natural inclusion of sheaves on $G$,
$\underline\field_{G'}\to\underline\field$ (but not generally any nonzero
morphism from $\underline\field=\underline\field_G$ 
to $\underline\field_{G'}$).

If $\phi\from K\to G$ is a morphism of graphs, 
and 
$\alpha\from {\cal F}_1\to{\cal F}_2$ is a morphism of sheaves on $K$,
then we have natural a natural morphism
$$
\phi_!\alpha \from \phi_!{\cal F}_1\to\phi_!{\cal F}_2,
$$
that make $\phi_!$ a functor on the category of sheaves.  Similarly
for $\phi_*$, and for the pullback,
$\phi^*$ (which acts the other way, from sheaves
and their morphisms on $G$ to those on $K$).

\subsection{Standard Sheaf Theories}

In this subsection we explain the connections with classical sheaf
theory on topological spaces.  We then describe our definitions
and particular choice of homology theory (and the role of $\phi_!$)
in terms of the view of Grothendieck et al. (\cite{sga4.1,sga4.2,sga4.3,
sga5}).

First consider an arbitrary topological space on a finite set, $X$.
Say that an open set, $U$, in $X$ is {\em irreducible} if $U$ is 
nonempty\footnote{If the empty set were considered irreducible, the
subcategory of irreducible open sets would have an initial element, making
the structure sheaf injective and giving the wrong homology groups.
One can say that the empty set is the union of proper subsets, namely
the empty union; as such the empty set is reducible ``by definition.''}
and
not the union of its proper subsets.
It is known that the category
of sheaves on $X$ is equivalent to the category of presheaves on
the irreducible open subsets; this can be proven directly---the
essential idea is that if a set is not irreducible, then we can
construct its value at a sheaf from those on its subsets;
there is also a proof in Section~2.5 of 
\cite{friedman_cohomology}, where this fact
follows easily from the Comparison Lemma
of
\cite{sga4.1}, Expos\'e III, 4.1. 
As is pointed out in \cite{friedman_cohomology}, this 
theorem is valid for any finite {\em semitopological} Grothendieck
topology,
where semitopological means that the underlying category has only
one morphism from any object to itself.

For example, if $X=\{A,B,C,D\}$ with irreducible open sets
being $\{A\}$, $\{C\}$, $\{A,B,C\}$, and $\{A,D,C\}$.  Then one
can recover a sheaf on $X$ (which has seven open sets) on the basis
of its values on these four sets, and any presheaf on these four sets
extends to a sheaf on $X$.
We remark that $X$ geometrically
corresponds (see \cite{friedman_cohomology})
to a circle, $X$, covered by two overlapping intervals, the intervals
corresponding to $\{A,B,C\}$ and $\{A,D,C\}$.  We have
$h_i(X)=1$ for $i=0,1$.

Let $G$ be a digraph with no self-loops.  In this case our sheaf
theory agrees with a standard topological one.  Namely,
let ${\rm Top}(G)$ be the topological space on $V_G\amalg E_G$,
whose open sets are subgraphs of $G$.  There are two types of
open irreducible sets: those of the form $\{v\}$ with $v\in V_G$, and
those of the form $\{he,e,te\}$ with $e\in E_G$; for each $e$ we have
$\{he\}$ and $\{te\}$ are subsets of $\{he,e,te\}$, and hence a sheaf
on ${\rm Top}(G)$ is determined by its values on the sets of type
$\{v\}$ and $\{he,e,te\}$ and the restrictions from the values on
$\{he,e,te\}$ to both $\{he\}$ and $\{te\}$.
We therefore recover our definition of a sheaf on a graph
(i.e., Definition~\ref{de:sheaf}).

Note that in the above $X=\{A,B,C,D\}$ definition, this is equivalent
to ${\rm Top}(G)$ with $V_G=\{A,C\}$ and $E_G=\{B,D\}$ and any heads/tails
correspondences making this a graph of two vertices joined by two
edges.

Notice that the above construction also gives a space, ${\rm Top}(G)$,
when $G$ has self-loops.  But this space has the wrong properties and
homology groups.  For example, if $G$ has one vertex and one self-loop,
then $h_i(G)=1$ for $i=0,1$ as defined in the previous section; however,
${\rm Top}(G)$ amounts to one irreducible open lying in another (with
only one inclusion, not the desired two), and we have
$h_1({\rm Top}(G))=0$.  So we now give a Grothendieck topology for
every digraph, $G$, that gives our sheaf and homology theory.

For each digraph, $G$, let ${\rm Cat}(G)$ be the category whose
objects are $V_G\amalg E_G$ and where the $2|E_G|$ non-identity morphisms are
given by $he\to e$ and $te\to e$ ranging over all $e\in E_G$ (with two
distinct morphisms
$he\to e$ and $te\to e$, even when $he=te$).
Then a sheaf over ${\rm Cat}(G)$ with the {\em grossi\`ere topologie},
i.e., a presheaf over the category ${\rm Cat}(G)$, is just the notion
of a sheaf given earlier.
Again, if $e$ is a self-loop, then this category has two morphisms between
two distinct objects; it is easy to see that the category of
sheaves over a graph with a self-loop cannot
be equivalent to the category of sheaves over any topological space.

Notice that earlier definitions regarding sheaves on $G$ and related
matters often involve
a $P$ in $V_G\amalg E_G$, giving vertices and edges a somewhat equal
treatment; this happens because $V_G$ and $E_G$ comprise the objects of
${\rm Cat}(G)$, and only the morphisms of ${\rm Cat}(G)$ distinguish
them.

At this point we will use explain certain features of the homology
theory we use here.
The proofs are in or are easy consequences of material in
\cite{friedman_cohomology}, and is mostly
easily derivable from material in
\cite{sga4.1,sga4.2,sga4.3,sga5} (which contains a lot of other 
material$\ldots$).
We shall assume the reader is familiar with basic
sheaf and cohomology theory
found in any algebraic geometry text, such as \cite{hartshorne}, and
we will just list a few points that are not standard,
or where the finite graph situation is 
different.
Let $\sheaves(G)$ be the category of sheaves of vector spaces (over some
fixed field, $\field$) on $G$.

\begin{enumerate}
\item $\sheaves(G)$ have enough projectives as well as injectives.
(See \cite{friedman_cohomology} for a simple 
characterization of all injectives
or projectives.)
\item If $u\from K\to G$ is a morphism of graphs, the pullback,
$u^*\from\sheaves(G)\to\sheaves(K)$ is defined via
$$
(u^*{\cal F})(P) = {\cal F}(u(P))
$$
for $P\in V_K\amalg E_K$, with its natural restriction maps
inherited from ${\cal F}$ (this is the same pullback defined in
Subsections~\ref{sb:unhappy} and
\ref{sb:pullbacks}); $u^*$ has a left adjoint, $u_!$ (defined in
Subsection~\ref{sb:pullbacks}), and a right
adjoint, $u_*$ (see \cite{sga4.1}, Expos\'e I, Proposition~5.1).  
In other words, 
\begin{equation}\label{eq:shadjoint}
\Hom_G(\phi_!{\cal F},{\cal L})  \isom \Hom_K({\cal F},\phi^*{\cal L})\quad
\forall {\cal F}\in\sheaves(G),\ {\cal L}\in\sheaves(K),
\end{equation}
and similarly for $\phi_*$.  
\item As a consequence we have
\begin{equation}\label{eq:shadjoint2}
\Ext^i_G(\phi_!{\cal F},{\cal L})  \isom \Ext^i_K({\cal F},\phi^*{\cal L})\quad
\forall {\cal F}\in\sheaves(G),\ {\cal L}\in\sheaves(K),
\end{equation}
and similarly for $\phi_*$.
\item If $u\from G'\to G$ is an inclusion of graphs, then
$u_!{\cal F}$ is just ${\cal F}_{G'}$, i.e., the sheaf that is zero
outside $G'$ and ${\cal F}$ when restricted to $G'$.
\item Any sheaf, ${\cal F}$, over $G$ has an injective resolution
$$
\Biggl( \bigoplus_{v\in V_G} (k_v)_* {\cal F}(v) 
\oplus \bigoplus_{e\in E_G} (k_e)_* {\cal F}(e) \Biggr) 
\longrightarrow
\Biggl( \bigoplus_{e\in E_G} (k_e)_* \bigl({\cal F}(te)\oplus{\cal F}(he) \bigr)
\Biggr)
$$
where for $P\in V_G\amalg E_G$, $k_P$ denotes the morphism from the 
category, $\Delta_0$, of one object and one (identity) morphism, to
$\Cat{G}$ sending the object of $\Delta_0$ to $P$.
In our case, this means that for a vector space, $W$, we have
$(k_P)_*W$ has the value $W^{d(Q)}$ at $Q$, where $d(Q)$ is the number
of morphisms from $Q$ to $P$.  For ${\cal F}=\underline{\field}$ this is homotopy
equivalent to a simpler resolution, namely
\begin{equation}\label{eq:shbetter_than_standard}
\underline{\field}\to \bigoplus_{v\in V_G} (k_v)_*\field
\to \bigoplus_{e\in E_G} (k_e)_*\field
\end{equation}
(see the paragraph about greedy resolutions and ``rank'' order in
Section~2.11 of \cite{friedman_cohomology}).
\item Similarly, any sheaf, ${\cal F}$, over $G$ has a projective resolution
$$
\Biggl( \bigoplus_{e\in E_G} \bigl( 
(k_{te})_! {\cal F}(e) \oplus
(k_{he})_! {\cal F}(e) \bigr) \Biggr)
\longrightarrow
\Biggl( \bigoplus_{v\in V_G} (k_v)_! {\cal F}(v) 
\oplus \bigoplus_{e\in E_G} (k_e)_! {\cal F}(e) \Biggr) 
$$
Again, $\underline{\field}$ (and numerous other sheaves encountered in practice)
have a simpler (``rank'' order) resolution:
\begin{equation}\label{eq:shbetter_than_standard2}
\bigoplus_{v\in V_G} (k_v)_!\field^{d_v-1} \to
\bigoplus_{e\in E_G} (k_e)_!\field \to \underline{\field},
\end{equation}
where $d_v$ is the degree of $v$ (the sum of the indegree and
outdegree), and the $d_v-1$ represents the fact that 
$\field^{d_v-1}$ is really the kernel of the map
$\field^{d_v}\to\field$ which is addition of coordinates;
similarly, in equation~(\ref{eq:shbetter_than_standard}), the $\field$
in $(k_e)_!\field$ is really the cokernel of the diagonal
inclusion $\field\to\field^2$, with the $2$ in $\field^2$ coming from
the fact that each edge is incident upon two vertices.

\item This means that the derived functors, 
${\rm Ext}^i({\cal F}_1,{\cal F}_2)$, of ${\rm Hom}({\cal F}_1,{\cal F}_2)$
can be computed as the cohomology groups of
$$
\bigoplus_{v\in V_G} {\rm Hom}\bigl( {\cal F}_1(v),{\cal F}_2(v) \bigr) 
\oplus 
\bigoplus_{e\in E_G}{\rm Hom}\bigl( {\cal F}_1(e),{\cal F}_2(e) \bigr) 
$$
$$
\longrightarrow
\bigoplus_{e\in E_G}{\rm Hom}\bigl( {\cal F}_1(e),{\cal F}_2(te)\oplus
{\cal F}_2(he) \bigr)
$$ 
\end{enumerate}

Now we can understand our choice of homology groups.
From equations~(\ref{eq:shbetter_than_standard}) and
(\ref{eq:shbetter_than_standard2}), we see that the constant
sheaf, $\underline\field$,
has a simple injective resolution but a more awkward projective
resolution.  So the homology theory that we've defined earlier amounts to
$$
H_i({\cal F}) = \bigl( \Ext^i({\cal F},\underline{\field}) \bigr)^\vee,
$$
where $\,^\vee$ denotes the dual space;
we have
$$
h_0({\cal F})-h_1({\cal F}) = \chi({\cal F}) = 
\dim\bigl( {\cal F}(V) \bigr)
-
\dim\bigl( {\cal F}(E) \bigr).
$$
As an alternative, one could study the standard cohomology theory
$$
H^i({\cal F})=\Ext^i(\underline{\field},{\cal F}).
$$
But we easily see that
$$
\dim\bigl(H^0({\cal F})\bigr)-\dim\bigl(H^1({\cal F})\bigr)
=
\dim\bigl({\cal F}(E)\bigr) - \sum_{v\in V_G} (d_v-1) \dim\bigl(
{\cal F}(v) \bigr) .
$$
This is another avenue to study, but does not seem
to capture in a simple way the invariant $\rho=\rho(G)$ of a digraph,
$G$.

We remark that we could reverse the role of open and closed sets in 
this discussion.
Indeed, to any sheaf, ${\cal F}$, of finite dimensional
$\field$-vector spaces
on a finite category,
${\cal C}$, we can take
the spaces dual to the ${\cal F}(P)$ for objects, $P$, of ${\cal C}$,
thereby getting a sheaf, ${\cal F}^\vee$, defined on ${\cal C}^{\rm opp}$,
the category opposite to ${\cal C}$ (i.e., the category obtained by
reversing the arrows).  Taking the opposite category has the effect of
exchanging open and closed sets, exchanging projectives and injectives,
etc.

Let us briefly explain the name ``structure sheaf.'' Generally speaking,
in sheaf theory each topological space or Grothendieck topology comes
with a special sheaf called the ``structure sheaf'' that has several
properties.  One key property is that the ``global sections'' of
a sheaf, ${\cal F}$, should reasonably be interpreted as the sheaf
homomorphisms to ${\cal F}$ from the structure sheaf.
This makes ``global cosections,'' on which our homology theory is based,
to be sheaf homomorphisms from ${\cal F}$ to the ``structure sheaf.''
Hence we call $\underline{\field}$ the structure sheaf.



\subsection{$\nu_!$, the left adjoint to $\nu^*$}
\label{sb:adjoint}

As mentioned in the previous subsection, if $\nu\from G'\to G$ is
an arbitrary graph morphism, then $\nu^*$ has a left adjoint, $\nu_!$.
In this subsection we show that $\nu_!$ is the left adjoint to $\nu_*$,
based on the general construction given in \cite{sga4.1}.
Although $\nu^*$ has a right adjoint, $\nu_*$, for our homology theory
it is $\nu_!$ that seems more important.



The general construction of $\nu_!$ is given in 
\cite{sga4.1}, Expos\'e I, Proposition~5.1).  
Alternatively, the reader can simply take the $\nu_!$ that we describe
and verify that it satisfies equation~(\ref{eq:shadjoint}).

According to \cite{sga4.1}, Expos\'e I, Proposition~5.1, given 
a sheaf, ${\cal F}$, on a graph $G$, i.e., a presheaf
on ${\rm Cat}(G)$, the value $\nu_!{\cal F}(P)$ for $P\in V_G\amalg E_G$ 
is determined as follows:
form the category $I^P_\nu$ whose objects are
$$
\{ (m,X)\ |\ X\in V_{G'}\amalg E_{G'},\ 
\mbox{$m\from P\to\nu(X)$ is a morphism in ${\rm Cat}(G)$} \},
$$
with a morphism from $(m,X)$ to $(m',X')$ being a morphism
$\mu\from X\to X'$ in ${\rm Cat}(G')$ such that $m'=\nu(\mu)m$;
then the projection $(m,X)\mapsto X$ followed by ${\cal F}$ gives
a contravariant functor from $I^P_\nu$ to $\field$-vector spaces, and
we take the inductive limit in $I^P_\nu$.  
It follows that if $e\in E_G$, then $I^e_\nu$ is category whose objects
are $({\rm id}_e,e')$ where $e'$ lies over $e$, and ${\rm id}_e$ is the
identity at $e$.  It follows that
$$
(\nu_!{\cal F})(e) = \bigoplus_{e'\in\nu^{-1}(e)} {\cal F}(e').
$$
If $v\in V_G$, then $I^v_\nu$ contains the following:
\begin{enumerate}
\item $({\rm id}_v,v')$ for each
$v'$ over $v$; 
\item $(\mu,e')$ for every $e'\in E_{G'}$ over an $e\in E_G$
with $he=v$, with $\mu$ the morphism from $v$ to $e$ given by the head
relation; and
\item the same with ``tail'' replacing ``heads.''
\end{enumerate}
We claim that each object $(\mu,e')$ has a unique morphism in $I^v_\nu$
to an element $({\rm id}_v,v')$, where $v'=he'$ in part~(2) and
$v'=te'$ in part~(3).  So the inductive limit for $(\nu_!{\cal F})(v)$
can be restricted to the subcategory of objects in part~(1), and we again
get a direct sum:
$$
(\nu_!{\cal F})(v) = \bigoplus_{v'\in\nu^{-1}(v)} {\cal F}(v').
$$
We leave it to the reader to verify that
the restriction maps of $\nu_!{\cal F}$ are just the
natural maps induced by ${\cal F}$.

Now we see that 
$$
(\nu_!{\cal F})(V_G) \isom
{\cal F}(V_{G'}),\quad
(\nu_!{\cal F})(E_G) \isom
{\cal F}(E_{G'}),
$$
with $d_{\nu_!{\cal F}}$ and $d_{{\cal F}}$ identified under the 
isomorphism.  Hence they have the same homology groups, same adjacency matrix,
etc.  The main difference is that one is a sheaf on $G$, the other a sheaf
on $G'$.

\subsection{$\nu_!$ and Contagious Vanishing Theorems}
\label{sb:vanishing}
In this section, we comment that vanishing of homology groups of a sheaf
implies the vanishing certain homology groups of related sheaves.
We call such results ``contagious vanishing'' theorems.  
This gives a nice use of the sheaves $\nu_!\underline{\field}$.
Let us 
first explain our interest in such results, as motivated by the
SHNC.

As mentioned before, we will show
that the SHNC is implied by 
the vanishing maximum excess of a sheaf that we call a $\rho$-kernel.
The $\rho$-kernel actually arises when considering a trivial and very
special case of
the SHNC; however it turns out that the vanishing of the maximum
excess these $\rho$-kernels
actually imply the entire SHNC.  What happens is that the trivial
case of the SHNC, when expressed as a short/long exact sequence, can
be ``tensored'' with sheaves of the form $\nu_!\underline{\field}$;
then a general ``contagious vanishing theorem'' implies that the
maximum excess of the $\rho$-kernel tensored with 
$\nu_!\underline{\field}$ vanishes; this proves {\em all cases} of the SHNC.
In other words, the vanishing of a homology group of a sheaf 
or of a related group can
be more powerful than it first seems.
Let us describe the underlying ideas, which are not specific to the SHNC.

Let $G'\subset G$ be digraphs, and let $G$ be a sheaf on
${\cal F}$.  Then we have an exact sequence
$$
0\to {\cal F}_{G'}\to {\cal F} \to {\cal F}/{\cal F}_{G'} \to 0.
$$
Of course, when $G$ has no self-loops, then this is a special case
of the general short exact sequence
$$
0\to {\cal F}_U \to {\cal F} \to {\cal F}_Z \to 0,
$$
where ${\cal F}$ is a sheaf on a topological space, $U$ is an open subset,
and $Z$ is the closed complement (see \cite{hartshorne}, Chapter II, 
Exercise 1.19 or Chapter III, proof of Theorem~2.7).
The long exact sequence implies that if $h_1({\cal F})=0$, then
$h_1({\cal F}_{G'})=0$.
Of course,
the same is true of any first quasi-Betti number, and so we have the
following simple but useful theorem.

\begin{theorem}
\label{th:open_contagious}
If $\alpha_1$ is any first quasi-Betti number for sheaves of $\field$-vector
spaces
on a graph, $G$, and if $\alpha_1({\cal F})=0$ for such a sheaf, ${\cal F}$, 
then for any subgraph, $G'$,
of $G$ we have $\alpha_1({\cal F}_{G'})=0$.
\end{theorem}

The intuition is clear in case $\alpha_1$ is $h_1$ or $h_1^\twist$ or
the maximum excess, and ${\cal F}=\underline{\field}$: passing to a subgraph
cannot increase the first Betti number or the reduced cyclicity of a graph.

One way in which a sheaf ${\cal F}_{G'}$ can naturally arise is when we
take a short exact sequence of sheaves in $G$,
$$
0\to {\cal F}_1\to {\cal F}_2\to {\cal F}_3 \to 0,
$$
and take the tensor product with 
$\underline{\field}_{G'}$; the tensor product preserves exactness
(i.e., all higher Tor groups vanish in sheaves of vector spaces over graphs),
so we get a new short exact sequence
$$
0\to {\cal F}_1\otimes\underline{\field}_{G'} 
\to {\cal F}_2\otimes\underline{\field}_{G'} 
\to {\cal F}_3 \otimes\underline{\field}_{G'} \to 0;
$$
now note that for any sheaf, ${\cal F}$, on $G$ we have
$$
{\cal F}\otimes\underline{\field}_{G'} = {\cal F}_{G'}.
$$

As a consequence, if one has an exact sequence of sheaves on $G$,
$$
0\to {\cal F}_1\to {\cal F}_2\to {\cal F}_3 \to 0,
$$
and one expects that
${\rm m.e.}({\cal F}_2)\le {\rm m.e.}({\cal F}_3)$, then a
simple homological explanation for this inequality would be
that ${\rm m.e.}({\cal F}_1)=0$.  But this would, in turn, imply
that ${\rm m.e.}(({\cal F}_2)_{G'}) \le
{\rm m.e.}(({\cal F}_3)_{G'})$ for all open subsets, $G'$, of $G$,
which could be a much stronger inequality (and is much stronger
for the setting of the
SHNC).

Let us state a
slightly stronger ``contagious vanishing'' theorem that we
shall apply to the maximum excess.
\begin{definition} By a {\em scaling} first quasi-Betti number, $\alpha_1$,
we mean a rule that, for some field, $\field$, and {\em any} digraph, $G$, 
assigns a non-negative real number to each sheaf of $\field$-vector spaces
over $G$, such that
\begin{enumerate}
\item $\alpha_1$ is a first quasi-Betti number when restricted to
sheaves on $G$ for any digraph, $G$; 
\item for any covering map $\phi\from K\to G$ of digraphs and any sheaf,
${\cal F}$, on $G$ we have
$$
\alpha_1(\phi^*{\cal F}) = \alpha_1({\cal F})\,\deg(\phi);
$$
and
\item for an \'etale $\phi\from K\to G$ and any sheaf, ${\cal F}$,
on $K$ we have
$$
\alpha_1(\phi_!{\cal F}) = \alpha_1({\cal F}).
$$
\end{enumerate}
\end{definition}
By the end of this chapter we will know that the maximum excess is a
scaling first quasi-Betti number: 
condition~(2) follows
from Theorem~\ref{th:shme_scale};
conditions~(1) and (3) follow from Theorem~\ref{th:shmain} by taking
limits; since condition~(1) is almost immediate, we prove only
condition~(3).

For arbitrary $\phi\from K\to G$, and arbitrary
$\mu\from G'\to G$, let $K'=
G'\times_G K$,
and let $\mu'\from K'\to K$ 
and $\phi'\from K'\to G'$ be the projections.
We easily see (on each vertex and edge of $G$) a natural isomorphism
\begin{equation}\label{eq:base_change}
(\phi')_!(\mu')^*{\cal F} \isom \mu^*\phi_!{\cal F}.
\end{equation}
Using equation~(\ref{eq:shspecial_adjoint}) we have
$$
h_1\bigl((\mu')^*{\cal F}\bigr) = 
h_1\bigl((\phi')_!(\mu')^*{\cal F}\bigr) = 
h_1(\mu^*\phi_!{\cal F}).
$$
Now we take $\mu\from G'\to G$ to be a covering map; 
then $\mu'\from K'\to K$ is a covering map of the same degree 
as $\mu$; since $\phi$ is \'etale, so is $\phi'\from K'\to G'$,
and hence the girth of $K'$ is at least that of $G'$ (since any
closed, non-backtracking walk on $K'$ pushes down, via $\phi'$, to one
on $G'$ of equal length).  
Hence if the girth of $G'$ is sufficiently large we have
$$
{\rm m.e.}({\cal F})=h_1\bigl((\mu')^*{\cal F}\bigr)/\deg(\mu')
=
h_1(\mu^*\phi_!{\cal F})/\deg(\mu)={\rm m.e.}(\phi_!{\cal F}).
$$

\begin{theorem}\label{th:etale_contagious}
Let $\alpha_1$ be a scaling first quasi-Betti number.
If $\alpha_1({\cal F})=0$ for a sheaf, ${\cal F}$, on a digraph,
$G$, and $\nu\from G'\to G$ is
\'etale,
then $\alpha_1({\cal F}_{G'})=0$ where
${\cal F}_{G'}=\nu_!\nu^*{\cal F}\isom {\cal F}\otimes\underline\field_{G'}$.
\end{theorem}
\begin{proof}
Since $\nu$ is \'etale, it factors as an open inclusion
$j\from G'\to G''$ followed by a covering map 
$\mu\from G''\to G$.
Since $\alpha_1$ scales, we have $\alpha_1({\cal F})=0$ implies that
$\alpha_1(\phi^*{\cal F})$.
which implies
$\alpha_1({\cal F}')=0$, where 
$$
{\cal F}'=(\mu^*{\cal F})_{G'}=j_!j^*\mu^*{\cal F}
$$
by Theorem~\ref{th:open_contagious}.
Hence
$$
\alpha_1(\mu_!{\cal F}')=\alpha_1({\cal F}')=0.
$$
But
$$
\mu_!{\cal F}' = \mu_!j_!j^*\mu^*{\cal F} \isom \nu_!\nu^*{\cal F}
={\cal F}_{G'},
$$
so 
$$
\alpha_1({\cal F}_{G'})=\alpha_1(\mu_!{\cal F}')=0.
$$
\end{proof}

Note that if $\phi\from K\to G$ is not \'etale, then for a sheaf,
${\cal F}$, on $K$,
the maximum excess of ${\cal F}$ and $\phi_!{\cal F}$ need not agree.
For example,
consider $\phi\from B_2\to B_1$,
the unique morphism of digraphs, where $B_i$ is the graph with one
vertex and $i$ self-loops, and ${\cal F}={\cal U}$, the unhappy
$4$-bundle of Subsection~\ref{sb:unhappy}.  Then
${\rm m.e.}({\cal U})=0$ but ${\rm m.e.}(\phi_!{\cal U})=1$, as the
head/tail neighbourhood of the span of $\beta-\gamma$ is two dimensional.
Moreover, notice how our proof 
that ${\rm m.e.}({\cal F})={\rm m.e.}(\phi_!{\cal F})$
for $\phi$ \'etale would fail to work for arbitrary $\phi$:
for arbitrary $\phi$, not necessarily \'etale, we would
not be able to assert
that $K'$ has large girth; in the example in this paragraph, 
$\phi$ takes two edges to one,
and as a result $K'$ always has girth at most two (and Abelian girth
at most eight).

We finish with a remark that may be useful when generalizing sheaves
to discrete structures beyond graphs.
Equation~(\ref{eq:base_change}) is known
as 
a ``base change'' morphism.  In more general contexts, there is usually
a natural ``base change'' morphism
$$
\mu^*\phi_!{\cal F}
\to
(\phi')_!(\mu')^*{\cal F} 
$$
that is not generally injective or surjective, not even for presheaves
of vector spaces on finite categories; see
\cite{friedman_cohomology}.
However if $\phi\from K\to G$ is any digraph morphism, then $\phi$
determines a functor $\Phi\from{\rm Cat}(K)\to{\rm Cat}(G)$ of the 
associated categories, and this functor is always ``target liftable,''
i.e., for any morphism, $\alpha$, of ${\rm Cat}(G)$ 
and object, $P$, in ${\rm Cat}(K)$ with $\Phi(P)$ being
the target of $\alpha$, there is a morphism, $\alpha'$, for ${\rm Cat}(K)$
whose target is $P$ and with $\Phi(\alpha')=\alpha$.
In \cite{friedman_cohomology} we see that the ``target liftable'' property
for $\phi$ (or, more precisely, $\Phi$) guarantees the isomorphism
in the base change morphism (actually \cite{friedman_cohomology} speaks
of the dual morphism $\mu^*\phi_*{\cal F}\to (\phi')_*(\mu')^*{\cal F}$
and ``source liftable'' in the discussion after Theorem~10.2 there). 
So if we, for example, spoke of graphs without requiring the edges to have
both endpoints, then the simple example of \cite{friedman_cohomology}
shows that equation~(\ref{eq:base_change}) would fail.
Of course, if $K$ is a subgraph obtained from $G$ by deleting any number
of edges, then the inclusion $\phi$ (or, more precisely, associated 
$\Phi$) is not source liftable.
Hence ``target liftability'' (and not ``source liftability'') should
guide us in generalizing sheaves on graphs to more general discrete
structures, if we wish to have similar theorems about analogues of
the maximum excess being invariant under $\phi_!$ for \'etale $\phi$.

\section{Twisted Cohomology}
\label{se:shtwist}
In this section we describe a number of aspects 
of twisted homology,
and give its relationship to the homology of pullbacks under Abelian
covers.
We show that the first twisted Betti number
of the structure sheaf of a graph, $G$, agrees with $\rho(G)$.
We then prove a number of related results, such as giving a 
condition under which the maximum excess
agrees with the first Betti number.

\subsection{Remarks on the Definition}

Twists and twisted homology were defined in Subsection~\ref{sb:twisted_def}.
In this subsection we make a few remarks on the definitions.

In our definition of twists, 
for symmetry we could have also specified a multiplier (like $\psi(e)$) for
${\cal F}^\psi(h,e)$, not just ${\cal F}^\psi(t,e)$; i.e., we could have
defined a twists to be a map $E_G\times\{t,h\}\to\field'$.
But there is no real need for a
${\cal F}^\psi(h,e)$ multiplier,
since all twisted
homology groups would be isomorphic.

Note that $h_i^\twist({\cal F})$ could be alternatively described as
the ``generic dimension of $h_i({\cal F}^\psi)$;'' more precisely,
there is a polynomial, $f$, in $\{\psi_e\}$ over $\field$ such that
the dimension of $h_i({\cal F}^\psi)$ for any fixed twist, $\psi$, with
$\psi_e\in\field$, is $h_i^\twist({\cal F})$ provided
that $f(\psi)\ne 0$.  Furthermore, for any particular $\psi\in\field^{E_G}$,
the dimension of $h_i({\cal F}^\psi)$ is at least the generic dimension.
All these facts follow from the fact that the rank of a matrix is the
size of the largest square submatrix whose determinant does not vanish.
This discussion assumes either that $\field$ is infinite or that
$\field$ is considered as
embedded in an infinite or sufficiently large extension field
of itself (it is not clear how to give an interesting meaning to ``generic''
when dealing with finite dimensional spaces over finite fields).
Of course, the advantange of our original definition, which involves
$\field(\psi)$ with the $\psi$ being indeterminates, is that it gives
a simple, usable definition for arbitrary $\field$, even when $\field$
is finite.

\subsection{Twists and Abelian Coverings}

We now wish to describe twisting as giving the homology of pullbacks
under Abelian coverings.  
Given an Abelian group, ${\cal A}$, say that a field, $\field$, is a 
{\em Fourier field for ${\cal A}$} if $\field$ contains $n=|{\cal A}|$
distinct $n$-th roots of $1$ (which holds, for example, when the 
characteristic of $\field$ is relatively prime to $n$ and $\field$ is
algebraically closed).  In this case,
if ${\cal A}$, acts on
a vector space, $S$, over a field, $\field$, then we have a canonical
isomorphism
$$
\bigoplus_\nu S^\nu \isom S,
$$
where $\nu\from\ca\to\field$ ranges over all characters on $\ca$ and
$$
S^\nu = \{ s\in S \ |\  as=\nu(a)s \mbox{\ for all $a\in\ca$} \} ;
$$
indeed, for each $\nu$ we have $S^\nu\subset S$, and these inclusions
give a map
from the direct sum of the $S^\nu$ to $S$; the inverse map, from $S$ to the
direct sum of the $S^\nu$, is given as the sum of the maps from
$S$ to any particular $S^\nu$ via
\begin{equation}\label{eq:shinverse_fourier}
s \mapsto (1/n)\sum_{\alpha\in\ca} \nu^{-1}(\alpha) (\alpha s);
\end{equation}
the values $1/n$ and $\nu^{-1}(\alpha)$ all lie in $\field$ for any
$\field$ that is a Galois field for $\ca$.

\begin{lemma}\label{lm:abelian}
Let $\phi\from G'\to G$ be an Abelian covering map with Galois group
${\cal A}$.
Let ${\cal F}$ be a sheaf of $\field$-vector spaces on $G$ such that
$\field$ is Fourier field for ${\cal A}$.  Then
\begin{equation}\label{eq:shhom_decomp}
H_i(\phi^*{\cal F}) \isom
\bigoplus_\psi \bigl( H_i(\phi^*{\cal F}) \bigr)^\nu,
\end{equation}
the sum is over all characters, $\nu$, of $\ca$.
Let $\vec a=\{a_e\}_{e\in E_G}$ be any 
Galois coordinates for $\phi\from G'\to G$, and
for any character, $\nu$, of ${\cal A}$, let
$\nu(\vec a)$ denote the $\field$-twist taking $e\in E_G$ to
$\nu(a_e)$.
Then for each $\nu$ we have
$$
\bigl( H_i(\phi^*{\cal F}) \bigr)^\nu \isom 
H_i\bigl({\cal F}^{\nu(\vec a)}\bigr).
$$
\end{lemma}
\begin{proof}
We have an $\ca$ action on $(\phi^*{\cal F})(E_{G'})$ via
$$
(af)(e) = f(ea)
$$
for all $a\in\ca$, $f\in (\phi^*{\cal F})(E_{G'})$, and $e\in E_{G'}$.
Similarly $(af)(v)=f(va)$ defines an
$\ca$ action on $(\phi^*{\cal F})(V_{G'})$.
The map in equation~(\ref{eq:shinverse_fourier}) gives isomorphisms
$$
(\phi^*{\cal F})(E) \to \bigoplus_{\nu} 
\bigl( (\phi^*{\cal F})(E) \bigr)^\nu,
\quad
(\phi^*{\cal F})(V) \to \bigoplus_{\nu} 
\bigl( (\phi^*{\cal F})(V) \bigr)^\nu,
$$
and $d_{\phi^*{\cal F}}$ intertwines with these maps, which
establishes
equation~(\ref{eq:shhom_decomp}).
It remains to identify
$$
\bigl( H_i(\phi^*{\cal F}) \bigr)^\nu 
$$
with $H_i$ of the appropriately twisted ${\cal F}$.
So choose Galois coordinates, $\{a_e\}$, and therefore
identify $V_{G'}$ with $V_G\times \ca$ and $E_{G'}$ with $E_G\times \ca$
so that
$$
h(e,a)=(he,a_e a ) \quad\mbox{and}\quad
t(e,a)=(te,a)
$$
(as in Subsection~\ref{sb:galois_coordinates}).
Given an $f\in (\phi^*{\cal F})(E)$, define $\mt f\in{\cal F}(E)$ via
$$
\mt f(e) = f(e,{\rm id}_\ca),
$$
where ${\rm id}_\ca$ is the identity of $\ca$ and we identify
$E_{G'}$ with $E_G\times\ca$ as above.  Similarly define a linear
map $f\mapsto \mt f$ from $(\phi^*{\cal F})(V)$  to ${\cal F}(V)$.
Now consider
$$
f\in \bigl( H_1(\phi^*{\cal F}) \bigr)^\nu.
$$
For all $v'\in V_{G'}$ we have
$$
\sum_{e'\ {\rm s.t.}\ te'=v'} f(e') =
\sum_{e'\ {\rm s.t.}\ he'=v'} f(e') .
$$
Taking $v'=(v,{\rm id}_\ca)$ yields
$$
\sum_{te=v} f(e,{\rm id}_\ca) =
\sum_{he=v} f(e,a_e^{-1}) 
= \sum_{he=v} (a_e^{-1} f) (e,{\rm id}_\ca),
$$
which, since $f\in ( H_1(\phi^*{\cal F}) )^\nu$,
$$
= \sum_{he=v} \nu(a_e^{-1})  f (e,{\rm id}_\ca).
$$
It follows that
$$
\sum_{te=v} \mt f(e) =
\sum_{he=v} \nu(a_e^{-1})  \mt f (e).
$$
In other words, if we set $f'(e)=\nu(a_e^{-1})\mt f(e)$, then we have
$$
\sum_{te=v} \nu(a_e) f'(e) =
\sum_{he=v} f'(e).
$$
Hence $f'\in H_1({\cal F}^{\nu(\vec a)})$.
Clearly given $f'$ we can reconstruct $\mt f$ and then $f$, namely
$$
f(e,a) = \nu(a_e)\nu(a)f'(e).
$$
Hence $f\mapsto f'$ is an isomorphism
$$
\bigl( (\phi^*{\cal F})(E_{G'}) \bigr)^\nu \to {\cal F}^{\nu(\vec a)}(E_G).
$$
Furthermore we have an analogous map $f\mapsto \mt f$
$$
\bigl( (\phi^*{\cal F})(V_{G'}) \bigr)^\nu \to {\cal F}^{\nu(\vec a)}(V_G),
$$
namely, $\mt f(v)=f(v,{\rm id}_\ca)$, which likewise is an isomorphism.
Hence we get a commutative diagram:
$$
\begin{diagram}
\bigl( (\phi^*{\cal F})(E_{G'}) \bigr)^\nu & \rTo^{f\mapsto f'} &
{\cal F}^{\nu(\vec a)}(E_G) \\
\dTo^{d_{\phi^*{\cal F}}} & & \dTo_{d_{{\cal F}^{\nu(\vec a)}}} \\
\bigl( (\phi^*{\cal F})(V_{G'}) \bigr)^\nu & \rTo^{f\mapsto \mt f} &
{\cal F}^{\nu(\vec a)}(V_G)  \\
\end{diagram}
$$
Since the horizontal arrows are isomorphisms, this diagram sets up
isomorphisms between the kernel and cokernel
of the vertical arrows.  Hence for $i=0,1$ we have
$$
\bigl( H_i(\phi^*{\cal F}) \bigr)^\nu \isom
H_i({\cal F}^{\nu(\vec a)})
$$
\end{proof}

Lemma~\ref{lm:abelian} shows that if $\field$ is any infinite field,
${\cal F}$ is any sheaf on a digraph, $G$,
and we take a random $\integers/p\integers$ cover $\mu\from G'\to G$,
then we have that $h_i(\mu^*{\cal F})/p$ tends to $h_i^\twist({\cal F})$
in probability as $p\to\infty$.

Lemma~\ref{lm:abelian} also shows that if $\mu\from G'\to G$ is an Abelian
cover with covering group $\ca$, then $H_i(\mu^*{\cal F})$ is the sum
of $|\ca|$ groups, each isomorphic to an $H_i({\cal F}^\psi)$ for a
particular value of $\psi$, and hence of dimension at least 
$h_i^\twist({\cal F})$.  We conclude the following lemma.

\begin{lemma}\label{lm:twist_bound}
If $\mu\from G'\to G$ is any Abelian cover of $G$, and
${\cal F}$ is any sheaf on $G$, then
$$
h_i(\mu^*{\cal F}) \ge \deg(\mu) h_i^\twist({\cal F}).
$$
\end{lemma}
This can be viewed as an upper bound for $h_i^\twist({\cal F})$.
Now we note the trivial lower bound
$$
h_1^\twist({\cal F}) \ge -\chi({\cal F}),
$$
since $h_1^\twist({\cal F})$ is the kernel of a matrix whose dimension
of domain minus that of codomain is $-\chi({\cal F})$.

If $G$ is any connected
digraph, then for any prime, $p$, we claim that $G$ has an Abelian cover 
of degree $p$ that is connected; indeed,
just take the monodromy map to map any
generator of $\pi_1(G)$ to $1\in\integers/p\integers$ and use the remark at
the end of Subsection~\ref{sb:walks_monodromy}.  
In this case we have $h_1(G')=1-\chi(G')=1-p\chi(G)=p\rho(G)+1$.
But by Lemma~\ref{lm:twist_bound} with ${\cal F}=\underline{\field}$
(so that $\mu^*{\cal F}=\underline{\field}$ on $G'$) we have
$$
h_1^\twist(\underline{\field}) \le  h_1(G',\underline{\field})/p
= h_1(G')/p =\rho(G)+(1/p).
$$
Letting $p\to\infty$ we conclude $h_1^\twist(\underline{\field})\le \rho(G)$.
But the ``trivial lower bound'' gives
$$
h_1^\twist(\underline{\field})\ge -\chi(\underline{\field}) = \rho(G).
$$
If $G$ is not connected then we apply the above to each of its connected
components and conclude the following theorem.
\begin{theorem} For any digraph, $G$, we have $\rho(G)=h_1^\twist(\underline{\field})$.
\end{theorem}

\subsection{The Maximum Excess Bound}

Let ${\cal F}$ be a sheaf of $\field$-vector spaces
on a digraph, $G$, and let $U\subset{\cal F}(V)$.
Let $\psi=\{\psi(e)\}_{e\in E_G}$ be a twist of indeterminates.
Then $d=d_{{\cal F}^\psi}\from {\cal F}(E)\to{\cal F}(V)$ can 
be restricted as a 
morphism 
$$
\Gamma_{\rm ht}(U)\otimes_\field \field'\to U\otimes_\field \field'.
$$
By the ``trivial bound,'' the kernel of this morphism has dimension at least
$$
\dim\bigl( \Gamma_{\rm ht}(U) \bigr) - \dim(U) = {\rm excess}({\cal F},U).
$$
Hence the kernel of $d$ has at least this dimension.  This gives the
following simple bound.

\begin{lemma}\label{lm:trivial}
For any sheaf, ${\cal F}$, on a digraph, $G$, we have
$$
h_1^\twist({\cal F}) \ge {\rm m.e.}(\cal F).
$$
\end{lemma}

We wish to show that this holds with equality in certain cases;
Theorem~\ref{th:shmain} says that
equality will hold if ${\cal F}$ is pulled back appropriately.

\begin{definition} If ${\cal F}$ is a sheaf on a digraph, $G$, we say 
that ${\cal F}$ is {\em edge simple} if ${\cal F}(e)$ is of dimension
$0$ or $1$ for each $e\in E_G$.
\end{definition}

\begin{theorem}\label{th:shedge_simple}
Let $\field$ be an infinite field.
Let ${\cal F}$ be an edge simple sheaf of $\field$-vector spaces on
a digraph, $G$.
Then
$$
h_1^\twist({\cal F}) = {\rm m.e.}({\cal F}).
$$
\end{theorem}

\begin{proof} 
Let $\{e_1,\ldots,e_r\}\subset E$ be the edges where ${\cal F}(e)\ne 0$.
Let $\psi=\{\psi_i\}_{i=1,\ldots,r}$ be
indeterminates, and let
$$
{\cal F}(V)(\psi)  = \bigl( {\cal F}(V) \bigr) \otimes_\field \field(\psi).
$$
For each $e_i$ choose a $w_i\in {\cal F}(e_i)$ with $w_i\ne 0$, and 
let 
$$
v_i = a_i + \psi_i b_i \in {\cal F}(V)(\psi), \quad\mbox{with}\quad
a_i = {\cal F}(h,e_i)(w_i),\quad
b_i = {\cal F}(t,e_i)(w_i).
$$
Say that a $v_j$ is {\em critical} for $v_1,\ldots,v_r$ if the span
of $\{v_i\}_{i\ne j}$ is of dimension one less than $\{v_i\}_{i=1,\ldots,r}$.
Let us first prove the lemma assuming that
no vector is critical.
Let $r'$ be the dimension of the span of the $v_i$, so
$h_1^\twist({\cal F})=r-r'$.
In view of Lemma~\ref{lm:trivial}, suffices to show that 
$$
{\rm m.e.}({\cal F}) \ge r-r'.
$$
If $r-r'=0$ there is nothing to prove.  So we may assume $r'<r$.

We wish to show that there exists a $U\subset{\cal F}(V)$ such that
$$
|\{ i \ |\ a_i,b_i \in U \}| \ge \dim(U)+ r-r'.
$$
Let us first assume that for any $I$ with $\{v_i\}_{i\in I}$ 
independent (over $\field(\psi)$) we also have that $\{a_i\}_{i\in I}$ are
independent (over $\field$).

By reordering the $v_i$, we may assume
that
$$
v_1,v_2,\ldots v_{r'}
$$
are linearly independent.
Let $A$ be the span of $a_1,\ldots,a_{r'}$.
Consider that
\begin{equation}\label{eq:shwedgeab}
(a_1 + \psi_1 b_1) \wedge \cdots \wedge (a_{r'+1} + \psi_{r'+1} b_{r'+1}) = 0.
\end{equation}
Considering the constant coefficient (i.e., with no $\psi_i$'s)
of this wedge product,
we have $a_1\wedge\cdots\wedge a_{r'+1}=0$, and therefore
$a_{r'+1}\in A$; similarly considering the $\psi_{r'+1}$
coefficient shows that $b_{r'+1}\in A$.  Replacing $v_{r'+1}$ with
any $v_s$ with $s>r'+1$ shows that
$$
b_{r'+1},\ldots,b_r,a_1,\ldots,a_r\in A.
$$

In other words, we have shown that
if $U$ is the span of the $a_1,\ldots,a_r$, we have that $U$ is
$r'$ dimensional and contains any $b_j$ such that $j$ lies outside
a set, $I$, such that $|I|=r'$ and $\{v_i\}_{i\in I}$ are independent.
But no vector, $v_i$, is critical for $\{v_i\}$; hence for any $j$ there
is an $I$ of size $r'$ such that $j$ lies outside $I$ and
$\{v_i\}_{i\in I}$ are independent.  Hence $b_j\in U$ for any $j=1,\ldots,r$.
Hence ${\rm excess}({\cal F},U)\ge r-r'$.  This establishes the lemma
when no vector, $v_i$, is critical, and when for all $I$, $\{v_i\}_{i\in I}$
are independent implies that $\{a_i\}_{i\in I}$ are as well.

Now let us establish the lemma assuming no vector, $v_i$, is critical but
without assuming $\{v_i\}_{i\in I}$ independent implies
$\{a_i\}_{i\in I}$ is independent.
Note that since $\field$ is infinite, any
generic set in $\field^n$ (i.e., complement of the set of zeros of a
polynomial) is nonempty.
For each $I$ for which $\{v_i\}_{i\in I}$ is independent, we have
$$
\bigwedge_{i\in I} (a_i+\psi_i b_i) \ne 0 
\quad\mbox{(in $\Lambda^{|I|}({\cal F}(V)\otimes_\field \field(\psi))$\;)}.
$$
So for a generic set, $G_I$, of $\theta\in\field^r$ we have
$$
\bigwedge_{i\in I} (a_i+\theta_i b_i) \ne 0.
$$
So choose a $\theta\in\field^r$ in the intersection of all $G_I$ for all $I$
with $\{v_i\}_{i\in I}$ independent.  Let $\mt\psi=\psi+\theta$
(where $\theta\in\field^r$ and $\psi$ is a collection of $r$
indeterminates), and let 
$$
\mt v_i= a_i + \mt\psi_i b_i = \mt a_i + \psi_i b_i,
$$
where $\mt a_i=a_i+\theta b_i$.
We have $\{v_i\}_{i\in I}$ is independent precisely when
$\{\mt v_i\}_{i\in I}$ is, since they differ by a parameter translation,
but whenever this holds we also have that the $\{\mt a_i\}_{i\in I}$
are independent.  
But we have already proven the lemma in this case, i.e., the case of
$\mt v_i=\mt a_i + \psi_i b_i$, since each independent subset of
$\{\mt v_i\}$ has the corresponding subset of $\{\mt a_i\}$ being
independent.
Hence we can apply the lemma to conclude that there
is a subspace $U$ of ${\cal F}(V)$
of dimension $r'$, namely the span of the $\mt a_i$, such that
$$
\mt a_1,\ldots,\mt a_r,b_1,\ldots,b_r\in U.
$$
But $a_i$ is an $\field$-linear combination of $\mt a_i$ and $b_i$,
so $\mt a_i,b_i\in U$ also implies $a_i\in U$.  Hence, again,
${\rm excess}({\cal F},U)\ge r-r'$.

Let us finish by proving the lemma in general, i.e., without the
assumption that each $v_i$ is critical.  
Again, let $r'$ be the dimension of the span of $v_1,\ldots,v_r$
as above.  
If some element of $v_1,\ldots,v_r$ is critical, we may assume it is 
$v_1$; in this case, if some element of $v_2,\ldots,v_r$ is critical for
that set, 
we may assume it is $v_2$; continuing in this fashion, there is an
$s$ such that for all $i<s$, $v_i$ is critical for $v_i,\ldots,v_r$,
and no element of $v_s,\ldots,v_r$ is critical for that set.
Consider the sheaf ${\cal F}'$ which agrees with ${\cal F}$ everywhere
except
that ${\cal F}'(e_i)=0$ for $i<s$
(and so 
${\cal F}$ and ${\cal F}'$ agree at all vertices and all $e_i$ with $i\ge s$).
Then
$\{v_s,\ldots,v_r\}$ is of size $r-s+1$, but also
the span of $\{v_s,\ldots,v_r\}$ is of size $r'-s+1$ (by the criticality
of the $v_i$ with $i<s$), and hence
$h_1^\twist({\cal F}')=r-r'$.
But since no element of $v_s,\ldots,v_r$ is critical for that set,
the lemma holds for the case of ${\cal F}'$ (as shown by the end of the
previous paragraph).
We therefore construct a $U$ such
that ${\rm excess}({\cal F}',U)\ge r-r'$. Since
${\cal F}'(V)\subset{\cal F}(V)$, we can view $U\subset{\cal F}(V)$
and it is clear that $\Gamma_{\rm ht}(U)$ in ${\cal F}'$ is a subset
of $\Gamma_{\rm ht}(U)$ in ${\cal F}$.  Hence
$$
{\rm excess}({\cal F},U)\ge {\rm excess}({\cal F}',U)=r-r'.
$$
\end{proof}

\section{Maximum Excess and Supermodularity}
\label{se:shme}
In this section we prove that pulling back a sheaf
via $\phi$ multiplies the maximum excess by
$\deg(\phi)$.
To prove this we will prove supermodularity of the excess function,
which has a number of important consequences.
Before discussing this, we develop some terminology and simple
observations about what we call ``compartmentalized subspaces;''
this development will be used in this section and
in Section~\ref{se:shequal}. 
We finish this section with some additional remarks about the
maximum excess.

\subsection{Compartmentalized Subspaces}
\label{sb:compartment}

In this subsection we mention a few important definitions, and some
simple theorems we will use regarding these definitions.

\begin{definition}
\label{de:compartment}
Let $W$ be a finite dimensional vector space
over a field, $\field$.  By a {\em decomposition} of $W$
we mean an isomorphism a direct sum of vector spaces with $W$, i.e.,
$$
\pi \from \bigoplus_{s\in S}W_s \to W.
$$
For any $s\in S$ and 
any $v\in W_s$, let 
{\em the extension of $v$ of index $s$ by zero}, denoted
${\rm extend}(v,s)$,
to be the element of $\oplus_{s\in S}W_s$
that is $v$ on $W_s$ and 
zero on $W_q$ with $q\ne s$.
For $s\in S$ and
a subspace $W'\subset W$, let {\em the portion of $W'$ supported in $s$} be
$$
{\rm supportedIn}(s,W') = 
\Bigl\{ v\in W_s \ |\ \pi\bigl( {\rm extend}(v,s)\bigr) \in W' \Bigr\},
$$
and let the {\rm compartmentalization of $W'$} be
$$
(W')_{\rm comp} =
\pi\left( \bigoplus_{s\in S} {\rm supportedIn}(s,W') \right), 
$$
which is a subspace of $W'$.
We say that a subspace $W'$ is {\em compartmentalized} if
$(W')_{\rm comp} = W'$.
We say that $w_1,\ldots, w_m\in W$ are {\em compartmentally distinct}
if for any $s\in S$ there is at most one $j$ between $1$ and $m$
for which the $W_s$ component of $w_j$ is non-zero.
\end{definition}

So $W'\subset W$ as above is compartmentalized iff $W'$ is the image
under $\pi$ of a set of the form
$$
\bigoplus_{s\in S} W_s'.
$$

The intuitive point of the definition of compartmentalized subspaces
is that certain constructions,
such as maximum excess, are performed over the direct
summands of a vector space; in some such constructions,
the compartmentalized subspaces are
the subspaces of key interest.

In this section we will use only these definitions.
In Section~\ref{se:shequal}, we use two simple observations about
the situation of Definition~\ref{de:compartment}.
First, if $w_1,\ldots,w_m$ are compartmentally distinct, then
$w_1,\ldots,w_m$ are linearly independent if (and only if) they
are each non-zero.
Second,
$W'\subset W$ is compartmentalized only if (and if) there exist quotients,
$Q_s$, of $W_s$ for $s\in S$ such that $\pi$ induces an isomorphism
\begin{equation}\label{eq:shquotient_sum}
\bigoplus_{s\in S} Q_s \to W/W'.
\end{equation}
It will be helpful to formally combine these two observations into a
theorem that follows immediately; we will use this theorem
repeatedly in Section~\ref{se:shequal}, in our
proof of Theorem~\ref{th:shmain}.

\begin{theorem}\label{th:shindependent_nonzero}
Let $W$ be a finite dimensional vector space with a decomposition.
Let $w_1,\ldots,w_m$ be compartmentally distinct,
and let $W'\subset W$ be a compartmentalized subspace of $W$.
Then the images of $w_1,\ldots,w_m$ in $W/W'$ are linearly independent 
(in $W/W'$) iff they are nonzero (in $W/W'$).
\end{theorem}

Compartmentalization is a key to our definition of maximum excess.
Indeed, for a sheaf, ${\cal F}$, on a digraph, $G$, both
${\cal F}(V)$ and ${\cal F}(E)$ are defined as direct sums, and hence
come with natural decompositions.
The head/tail neighbourhood is a compartmentalized
space by its definition in equation~(\ref{eq:shhead/tail}); this is crucial
to the resulting definition of excess and maximum excess, in
Definition~\ref{de:excess}.  
Note that $d_h,d_t$ (but not $d$ in
general) are 
``compartmentalized morphisms'' in that they take vectors supported in
one component of ${\cal F}(E)$ to those supported in one component
of ${\cal F}(V)$.
This means that with our definition of head/tail neighbourhood,
for any $U\subset{\cal F}(V)$ and any twist, $\psi$, on $G$, 
the twisted differential, $d_{{\cal F}^\psi}$
takes 
$\Gamma_{\rm ht}(U)\otimes_{\field}\field(\psi)$ 
to $U\otimes_{\field}\field(\psi)$.

\subsection{Supermodularity and Its Consequences}

First we make some simple remarks on the maximum excess.
For any sheaf, ${\cal F}$, we have
$$
{\rm excess}({\cal F},0)= 0,
\quad
{\rm excess}({\cal F},{\cal F}(V))=-\chi({\cal F}),
$$
and hence
$$
{\rm m.e.}({\cal F}) \ge \max\bigl(0,-\chi({\cal F})\bigr).
$$


We now show that if $U$ achieves the maximum excess of ${\cal F}$, then
$U$ must be compartmentalized.  

\begin{theorem} Let the maximum excess of a sheaf, ${\cal F}$, on a
digraph, $G$, be achieved on a space $U\subset{\cal F}(V)$.  Then
$U$ is compartmentalized with respect to the identification $\pi$ given by
$$
\pi\from \bigoplus_{v\in V_G} {\cal F}(v) \to {\cal F}(V).
$$
\end{theorem}
\begin{proof}
For $e\in E_G$
and $w\in{\cal F}(e)$, if we have $d_t w\in U$, then
$$
d_t w = \pi\Bigl( {\rm extend}\bigl({\cal F}(t,e)w,te \bigr)\Bigr)
\in U_{\rm comp};
$$
similarly if $d_h w\in U$, then $d_h w\in U_{\rm comp}$.  Hence,
in view of equation~(\ref{eq:shhead/tail}), we have
$$
\Gamma_{\rm ht}(U_{\rm comp}) = \Gamma_{\rm ht}(U).
$$
Hence, if $U_{\rm comp}$ is a proper subspace of $U$, then
$$
{\rm excess}({\cal F},U_{\rm comp}) < 
{\rm excess}({\cal F},U) .
$$
So if $U$ maximizes the excess, then $U_{\rm comp}=U$; i.e., $U$ is
compartmentalized.
\end{proof}

The main results in this section stem from the following easy theorem.
\begin{theorem} 
\label{th:shsupermodular}
Let ${\cal F}$ be a sheaf on a graph, $G$.  Then the
excess, as a function of $U\subset {\cal F}(V)$, is supermodular, i.e.,
\begin{equation}\label{eq:shsupermodular}
{\rm excess}(U_1) + {\rm excess}(U_2) \le 
{\rm excess}(U_1\cap U_2) + {\rm excess}(U_1+U_2)
\end{equation}
for all $U_1,U_2\subset {\cal F}(V)$.
It follows that the maximizers of the excess function of ${\cal F}$,
$$
{\rm maximizers}({\cal F})=\{ U\subset{\cal F}(V)
\ | \  {\rm excess}(U)={\rm m.e.}({\cal F}) \},
$$
is a sublattice of the set of subsets of ${\cal F}(V)$,
i.e., is closed under intersection and sum
(and therefore has a unique maximal element and a unique minimal element).
Finally, if $U_1,U_2$ are maximizers of the excess function of ${\cal F}$,
then
$$
\Gamma_{\rm ht}(U_1+U_2) = 
\Gamma_{\rm ht}(U_1) + 
\Gamma_{\rm ht}(U_2) .
$$
\end{theorem}
\begin{proof}
We use the fact that if $A_1,A_2$ are any subspaces of 
an $\field$-vector space,
then
$$
\dim(A_1)+\dim(A_2) = \dim(A_1\cap A_2)+\dim(A_1+A_2).
$$
In particular, for $U_1,U_2\subset{\cal F}(V)$ we have
\begin{equation}\label{eq:shincexc}
\dim(U_1)+\dim(U_2) = \dim(U_1\cap U_2)+\dim(U_1+U_2).
\end{equation}
On the other hand
$$
\Gamma_{ht}(U_1\cap U_2) = \Gamma_{ht}(U_1)\cap \Gamma_{ht}(U_2)
$$
and
\begin{equation}\label{eq:shsupset}
\Gamma_{ht}(U_1 + U_2) \supset \Gamma_{ht}(U_1) + \Gamma_{ht}(U_2);
\end{equation}
hence
\begin{equation}\label{eq:shincexc2}
\dim\bigl( \Gamma_{ht}(U_1) \bigr) + \dim\bigl( \Gamma_{ht}(U_2) \bigr)
\le
\dim\bigl( \Gamma_{ht}(U_1\cap U_2) \bigr) +
\dim\bigl( \Gamma_{ht}(U_1 + U_2) \bigr) .
\end{equation}
Combining equations~(\ref{eq:shincexc}) and (\ref{eq:shincexc2}) yields
equation~(\ref{eq:shsupermodular}).
It follows that if $U_1$ and $U_2$ are maximizers of the excess function
of ${\cal F}$, then
so are $U_1\cap U_2$ and $U_1 + U_2$, and 
equations~(\ref{eq:shincexc2}) and hence 
(\ref{eq:shsupset}) must hold with equality.
\end{proof}

The supermodularity has a number of important consequences.  We list
two such theorem below.

\begin{theorem} 
\label{th:shme_scale}
Let $\phi\from G'\to G$ be a covering map of graphs, and
let ${\cal F}$ be a sheaf on $G$.  Then
\begin{equation}\label{eq:shmescale}
{\rm m.e.}(\phi^*{\cal F}) = \deg(\phi)\ {\rm m.e.}({\cal F}).
\end{equation}
Furthermore, if the maximum excess of ${\cal F}$ is achieved at
$U\subset{\cal F}(V_G)$, then the maximum excess of $\phi^*{\cal F}$
is achieved at $\phi^{-1}(U)$.
\end{theorem}
\begin{proof}  Our proof uses Theorem~\ref{th:shsupermodular}
and Galois theory.  
Let ${\cal F}'=\phi^*{\cal F}$.
If $T\subset{\cal F}(V)$ is compartmentalized, $T=\oplus_{v\in V_G} T_v$, let
$$
\phi^{-1}(T) = \bigoplus_{v'\in V_{G'}} T_{\phi(v')} 
\subset {\cal F}'(V_{G'}).
$$
Since $\phi$ is a covering map, the number of preimages of any element
of $V_G\amalg E_G$ is $\deg(\phi)$, and hence
\begin{equation}\label{eq:shexcess_lift}
{\rm excess}\bigl({\cal F}',\phi^{-1}(T)\bigr) = \deg(\phi)
\,{\rm excess}({\cal F},T).
\end{equation}
Taking $T$ to maximize the excess of ${\cal F}$ we get 
\begin{equation}\label{eq:sheasy}
{\rm m.e.}({\cal F}') \ge
\deg(\phi)\, {\rm m.e.}({\cal F}).
\end{equation}
It remains to prove the reverse inequality in order to establish
equation~(\ref{eq:shmescale});  note that if we do so, then the second
statement of the theorem follows from equation~(\ref{eq:shexcess_lift}).

First let us assume that $\phi$ is Galois, with Galois group
${\rm Gal}(\phi)$. 
Each $g\in{\rm Gal}(\phi)$ is a morphism $g\from K\to K$.
Let ${\cal F}'=\phi^*{\cal F}$.  There is a natural map
$\iota_g\from g^*{\cal F}'\to{\cal F}'$, since for every $P\in V_{G'}\amalg
E_{G'}$ we have ${\cal F}'(P)={\cal F}'(Pg)$ (note that this really
is equality of vector spaces; they both equal ${\cal F}(\phi(P))$, by
definition).
So $\iota_g$ gives automorphism on ${\cal F}'(E_{G'})$ and 
${\cal F}'(V_{G'})$.
For any $U\subset(\phi^*{\cal F})(V)$, 
any element of ${\rm Gal}(\phi)$ preserves $\dim(U)$ and 
$\dim(\Gamma_{ht}(U))$, and hence the excess.  It follows that
for all $g\in{\rm Gal}(\phi)$, $\iota_g$ takes
${\rm maximizers}(\phi^*{\cal F})$ to itself.
Hence
if $W$ is the unique maximal element of the maximizers, then
$W$ is invariant under $\iota_g$ for all $g\in{\rm Gal}(\phi)$; this
means that if $W=\oplus_{v'\in V(G')}W_{v'}$ and 
$$
\mt W = \bigoplus_{v\in V_G} \left( \sum_{v'\in\phi^{-1}(v)} W_{v'}\right),
$$
then ($W_{v'}=W_{v''}$ if $\phi(v')=\phi(v'')$ and)
$W=\phi^{-1}(\mt W)$.  Hence
\begin{eqnarray*}
{\rm m.e.}({\cal F}')&=&{\rm excess}(W)
\\
&=& \deg(\phi) \ {\rm excess}_{\cal F}(\mt W) 
\le \deg(\phi) \ {\rm m.e.}({\cal F}).
\end{eqnarray*}
In summary,
$$
{\rm m.e.}({\cal F}') \le \deg(\phi) \ {\rm m.e.}({\cal F}).
$$
From equation~(\ref{eq:sheasy}), it follows that
the above inequality holds with equality.

It remains to prove the equality when $\phi\from G'\to G$ is
not Galois. 
By the Normal Extension Theorem of Galois graph theory 
(i.e., Theorem~\ref{th:shnormal_extension_theorem}), there exists
a
$\nu\from L\to G'$ be such that $\phi\nu$ (and hence
$\nu$) is Galois.  Since $\phi\nu$ is Galois, we have
$$
{\rm m.e.}(\nu^*\phi^*{\cal F}) = \deg(\phi\nu)\,{\rm m.e.}({\cal F}),
$$
and since $\nu$ is Galois we have
$$
{\rm m.e.}(\nu^*(\phi^*{\cal F})) = \deg(\nu)\,{\rm m.e.}(\phi^*{\cal F}).
$$
It follows that
$$
{\rm m.e.}(\phi^*{\cal F}) =
\deg(\phi)\,{\rm m.e.}({\cal F}).
$$
\end{proof}

\subsection{Additional Remarks on the Maximum Excess}

Here we make some additional remarks on the maximum excess, either
for later use or to provide some more intuition about it.

We mention that ${\rm m.e.}({\cal F})+\chi({\cal F})$ can be viewed as
a generalization of
the ``number of acyclic components'' of a graph; for example,
for the sheaf $\underline\field$ on $G$ we have
$$
{\rm m.e.}(\underline{\field})+\chi(\underline{\field}) 
=
\rho(G) + |V_G|-|E_G|
= h_0^{\rm acyclic}(G)
$$
equals the number of ``acyclic components'' of $G$, i.e. the number
of connected components of $G$ that have no cycles, i.e.,
that are isolated vertices or trees.
A similar remark holds for $\underline\field$ replaced by
$\underline\field_K$ and $G$ replaced by $K$,
for any map $K\to G$.

We shall make use of the following
alternate interpretation of the maximum excess.

\begin{theorem}\label{th:me_as_subsheaf}
For any sheaf, ${\cal F}$, on a digraph, $G$, the
maximum excess of ${\cal F}$ is the same as
$$
\max_{{\cal F}'\subset{\cal F}} -\chi({\cal F}'),
$$
i.e., the maximum value of minus the Euler characteristic over all
subsheaves, ${\cal F}'$, of ${\cal F}$.
\end{theorem}
\begin{proof} Each compartmentalized $U\subset{\cal F}(V)$ along with
$\Gamma_{\rm ht}(U)$ determines a subsheaf ${\cal F}'$ whose
Euler characteristic is minus the excess of $U$.  Conversely, for
any subsheaf ${\cal F}'\subset{\cal F}$ we have $U={\cal F}'(V)$
satisfies
$$
\dim({\cal F}')=\dim(U), \qquad
{\cal F}'(E)\subset \Gamma_{\rm ht}(H).
$$
Hence the excess of $U$ is at least minus the Euler characteristic 
of ${\cal F}'$.
\end{proof}

The above theorem has a simple graph theoretic analogue, namely that
$$
\rho(G) = \max_{H\subset G} -\chi(H).
$$
One can easily prove this directly (with $\rho(G)=-\chi(H)$ when $H$
consists of all cyclic connected components of $G$) or use
Theorem~\ref{th:me_as_subsheaf}.

We remark that it is easy to give a direct proof 
that the maximum excess satisfies some of the properties of a first
quasi-Betti number.  For example, it is immediate that 
for sheaves ${\cal F}_1,{\cal F}_2$ on a graph, $G$, we have
$$
{\rm m.e.}({\cal F}_1\oplus{\cal F}_2) =
{\rm m.e.}({\cal F}_1) +
{\rm m.e.}({\cal F}_2) .
$$
As another example, if ${\cal F}_1\to{\cal F}_2$ is an injection, then
Theorem~\ref{th:me_as_subsheaf} shows that
$$
{\rm m.e.}({\cal F}_1) \le {\rm m.e.}({\cal F}_2).
$$
It is quite conceivable that all of the ``first quasi-Betti number''
properties of the maximum excess have simple, direct proofs that
avoid using Theorem~\ref{th:shmain}.
However, we find that Theorem~\ref{th:shmain}, that implies that the
maximum excess is a limiting twisted Betti number, is extremely useful
in providing intuition about the maximum excess.

\section{$h_1^\twist$ and the Universal Abelian Covering}
\label{se:shabel}
For a digraph, $G$, we will study its maximum Abelian covering,
$\pi\from G[\integers]\to G$, which is an infinite graph,
and show that for a sheaf ${\cal F}$, on $G$, we have
$H_1^\twist({\cal F})$ is non-zero iff there is a non-zero element of
$H_1(\pi^*{\cal F})$ that is of finite support.
This is crucial to our proof of Theorem~\ref{th:shmain}.
We shall illustrate these theorems on the unhappy $4$-bundle, which
gives great insight into our proof of Theorem~\ref{th:shmain} that we
give in Section~\ref{se:shequal}.

Let $\integers$ be the set of integers, and let
$\integers_{\ge 0}$ be the set of non-negative integers.
For a set, $S$, we use $\integers^S$ to denote the set of functions
from $S$ to $\integers$.
We define the {\em rank} of an $n\in\integers^S$ to be
$$
{\rm rank}(n) = \sum_{s\in S} n(s)
$$
(in this paper $S$ will always be
finite, so the summation makes sense).

Given a digraph, $G$, let $G[\integers]$ be the infinite digraph with
$$
V_{G[\integers]} = V_{G}\times \integers^{E_G},\quad
E_{G[\integers]} = E_{G}\times \integers^{E_G},
$$
with heads and tails maps given for each $e\in E_G$ and 
$n\in\integers^{E_G}$ by
$$
h_{G[\integers]}(e,n) = (h_G e,n), \quad
t_{G[\integers]}(e,n) = (t_G e,n+\delta_e),
$$
where $\delta_e\in\integers^{E_G}$ is $1$ at $e$ and $0$ elsewhere.
Projection onto the first component gives an infinite degree covering
map $\pi\from G[\integers]\to G$.
For a vertex, $(v,n)$, or an edge, $(e,n)$,
of $G[\integers]$, we define its {\em rank}
to be the rank of $n$.

\begin{definition} For a digraph, $G$, we define the {\em universal
Abelian covering} of $G$ to be $\pi\from G[\integers]\to G$ described
in the previous paragraph.
\end{definition}
It is not important to us, but easy to verify, that $\pi$ factors
uniquely through any connected
Abelian covering of $G$.  Abelian coverings have been studied in
numerous works, including
\cite{friedman_tillich_generalized,friedman_murty_tillich}.

We similarly define $G[\integers_{\ge 0}]$, with $\integers_{\ge 0}$
replacing $\integers$ everywhere; $G[\integers_{\ge 0}]$ can be viewed 
as a subgraph of $G[\integers]$.

Our approach to Theorem~\ref{th:shmain} involves the properties of the graphs
$G[\integers_{\ge 0}]$, so let us consider some examples.
If $B_d$ denotes the bouquet of $d$ self-loops, i.e., the digraph with
one vertex and $d$ edges, then $B_d[\integers_{\ge 0}]$ is just the usual
$d$-dimensional non-negative
integer lattice, depicted in Figures~\ref{fg:b2} and \ref{fg:b2diag}.
\begin{figure}[ht]
  \begin{center}
    \epsfxsize=2in\epsffile{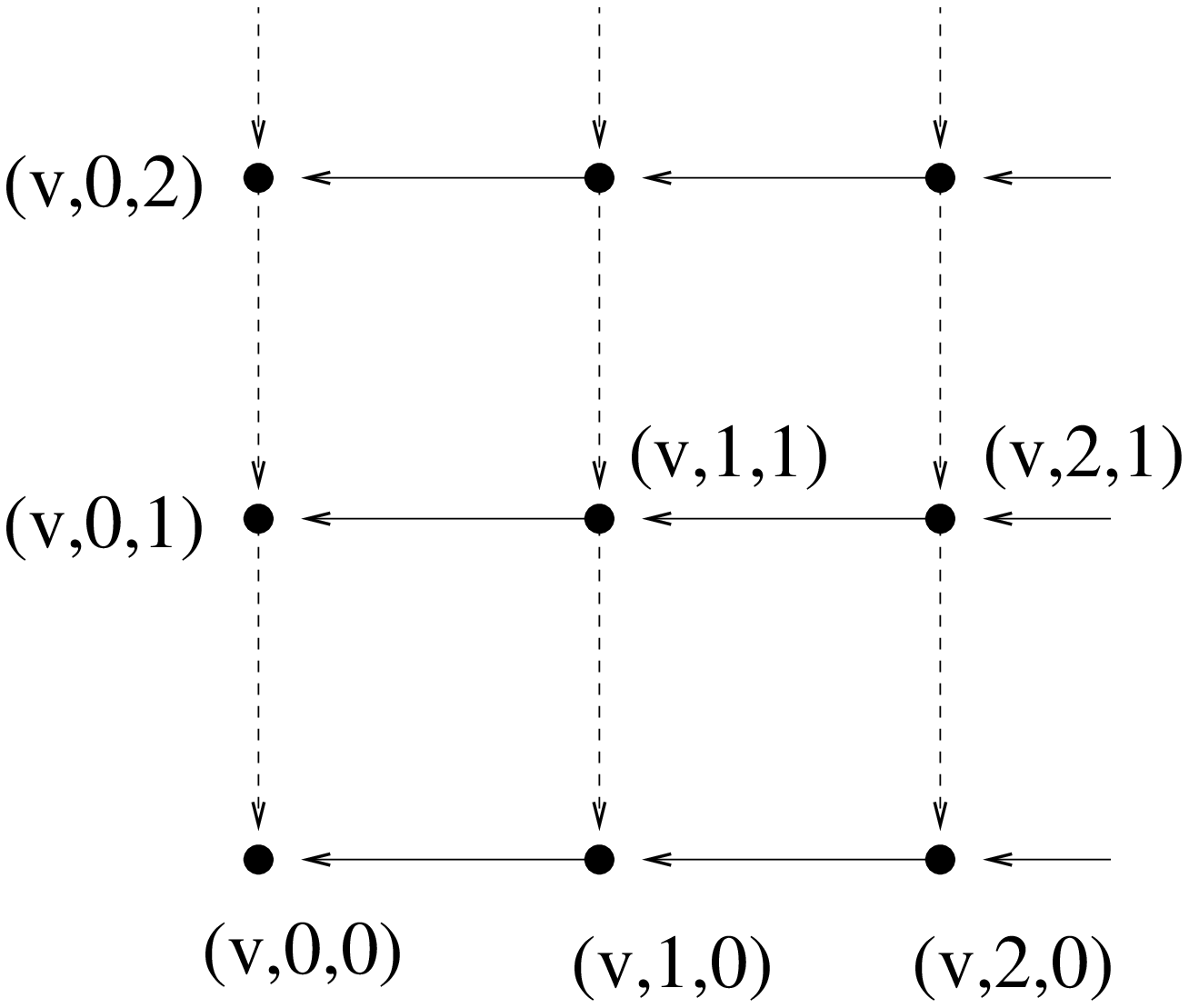}
    \caption{$B_2[\integers_{\ge 0}]$.}
    \label{fg:b2}
  \end{center}
\end{figure}
If $G'\to G$ is a covering map of degree $d$, then
$G'[\integers]\to G[\integers]$ and $G'[\integers_{\ge 0}]\to
G[\integers_{\ge 0}]$ are both covering maps.
\begin{figure}[ht]
\begin{center}
\begin{diagram}[nohug,height=3em,width=3em,tight]
&&&&(v,0,0) &&&& \\
&&&\ruTo && \luDashto &&& \\
&& (v,1,0) &&&& (v,1,0) && \\
&\ruTo && \luDashto && \ruTo && \luDashto & \\
(v,2,0) &&&& (v,1,1) &&&& (v,0,2) \\
\end{diagram}
\caption{First part of $B_2[\integers_{\ge 0}]$. Notice the cycle of
length four.}
\label{fg:b2diag}
\end{center}
\end{figure}
However, for $d>1$ and $|E_G|\ge 1$, we have $|E_{G'}|> |E_{G}|$,
and the covering will be of infinite degree.

Now consider $G'[\integers_{\ge 0}]$, where $\phi\from G'\to B_2$ 
is the degree two
cover of $B_2$ discussed with the unhappy $4$-bundle in 
Subsection~\ref{sb:unhappy} (just beneath equation~(\ref{eq:shme_scale})).
\begin{figure}[ht]
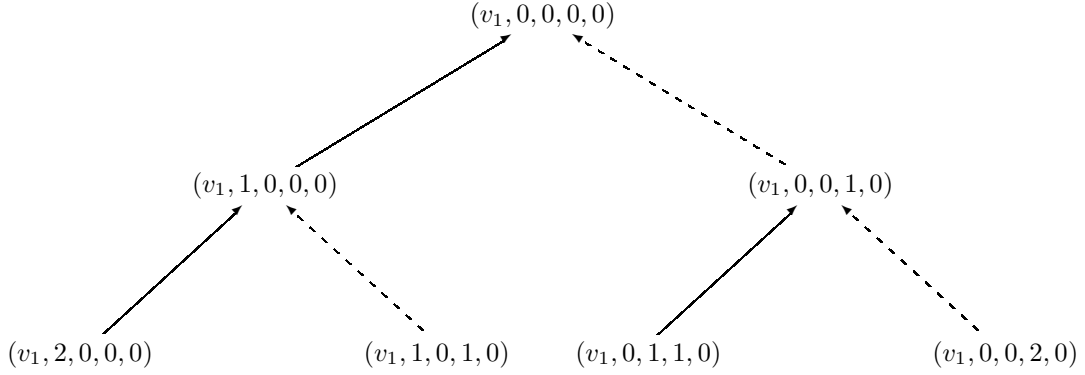

\begin{diagram}[nohug,height=3.2em,width=3.5em,tight]
&&&&&(v_1,0,0,0,0) &&&&& \\
&&&&\ruTo(3,2) && \luDashto(3,2) &&&& \\
&& (v_1,1,0,0,0) &&&&&& (v_1,0,0,1,0) && \\
&\ruTo && \luDashto &&&& \ruTo && \luDashto & \\
(v_1,2,0,0,0)&&&&(v_1,1,0,1,0)\quad &&\quad (v_1,0,1,1,0) &&&& (v_1,0,0,2,0) \\
\end{diagram}
\caption{First part of $G'[\integers_{\ge 0}]$ near $(v,\vec 0)$. 
No cycles of length four. 
The four $(\integers_{\ge 0})^{E_{G'}}$ coordinates are, in order,
$e_1^1,e_1^2,e_2^1,e_2^2$ where $e_i^j$ lies over $e_i\in E_{B_2}$ and
are described in the last equations
of Subsection~\ref{sb:unhappy} that give the $\nu_i^j$.}
\label{fg:u}
\end{figure}
As we see, and illustrated in Figure~\ref{fg:u}, 
$G'[\integers_{\ge 0}]$ has no cycle of length four.
As we shall see, the fact that $h_1^\twist({\cal U})=1$ is a
result, in a sense,
of the cycles
of length four in $B_2[\integers_{\ge 0}]$; the fact that these cycles
``open up'' to non-closed walks in $G'[\integers_{\ge 0}]$ is partly why
$h_1^\twist(\phi^*{\cal U})=0$.


Now we define homology groups on graphs of the form
$G[\integers]$ and $G[\integers_{\ge 0}]$, and, more generally, any
infinite graph.
If $K$ is a infinite graph that is locally finite (i.e., each vertex is
incident upon a finite number of edges), we can still define a sheaf
(of finite dimensional vector spaces over a field, $\field$) just as before.
Hence a sheaf,
${\cal F}$, on $K$ as a collection of a finite dimensional $\field$-vector
space, ${\cal F}(P)$ for each $P\in V_K\amalg E_K$, along with
restriction maps ${\cal F}(h,e)$ and ${\cal F}(t,e)$ for each $e\in E_K$.
We shall define
$$
{\cal F}^\oplus(V)=\bigoplus_{v\in V_G} {\cal F}(v),\quad\mbox{and}\quad
{\cal F}^\Pi(V)=\prod_{v\in V_G} {\cal F}(v),
$$
which generally differ, ${\cal F}^\oplus(V)$ being the subset of
${\cal F}^\Pi(V)$ of elements $\{f_v\}_{v\in V_G}$ that are supported
(i.e., nonzero)
on only finitely many $v$.  Similarly we define ${\cal F}^\oplus(E)$
and ${\cal F}^\Pi(E)$.
Then $d=d_h-d_t$ can be viewed as a map ${\cal F}^\Pi(E)\to{\cal F}^\Pi(V)$
or, respectively, ${\cal F}^\oplus(E)\to{\cal F}^\oplus(V)$,
and their cokernels and kernels are respectively denoted
$H_i^\Pi({\cal F})$ and $H_i^\oplus({\cal F})$ for $i=0,1$.

If ${\cal F}$ is a sheaf on $G$, and $\pi\from G[\integers]\to G$ the
universal Abelian covering, then $\pi^*{\cal F}$ is a sheaf on
$G[\integers]$.

The following simple but important observation explains our
interest in
the universal Abelian covering.

\begin{lemma}
\label{lm:habelian}
Let ${\cal F}$ be a sheaf on $G$, and 
$\pi\from G[\integers]\to G$ the
universal Abelian covering.
Then $H_1^\twist({\cal F})$ is non-trivial iff
$H_1^\oplus(\pi^*{\cal F})$ is non-trivial.
If so, there is a non-zero $w\in H_1^\oplus(\pi^*{\cal F})$
that is supported on $G[\integers_{\ge 0}]$.
\end{lemma}
\begin{proof} 
For each $e\in E_G$, let ${\cal F}(e)$ be of dimension $d_e$ and
have basis $f_{e,1},\ldots,f_{e,d_e}$.
Let 
$$
a_{e,i} = {\cal F}(h,e)f_{e,i} \in {\cal F}(he), \quad
b_{e,i} = {\cal F}(t,e)f_{e,i} \in {\cal F}(te).
$$
We have
$h_1^\twist(\mu^*{\cal F})\ge 1$ iff
the vectors
$$
a_{e,i} + \psi(e) b_{e,i} 
$$
are linear dependent over $\field(\psi)$, where $\psi$ is a collection
of indeterminates indexed on $E_{G}$.  This holds iff there
are rational functions $c_{e,i}\in \field(\psi)$ for each $e\in E_G$
and $i=1,\ldots,d_e$ such that
\begin{equation}\label{eq:shlinear_psi}
\sum_{e\in E_{G}} \ \sum_{i=1}^{d_{e}} c_{e,i}(\psi)
(a_{e,i} + \psi(e) b_{e,i} ) = 0,
\end{equation}
where not all $c_{e,i}$ are zero.
We may multiply the denominators
of the $c_{e,i}(\psi)$ to assume that they are polynomials, not all
zero.
We may write
$$
c_{e,i}(\psi) = \sum_{n\in(\integers_{\ge 0})^{E_G}} c_{e,i,n} \psi^n,
$$
where $c_{e,i,n}\in{\field}$ and
$$
\psi^n = \prod_{e\in E_G} \psi^{n(e)}(e).
$$

In summary, we see that $h_1^\twist({\cal F})\ne 0$ iff there 
exist $c_{e,i,n}\in\field$, with $c_{e,i,n}=0$ for all but finitely
many $n$, such that
\begin{equation}\label{eq:shbigsum}
\sum_{n\in (\integers_{\ge 0})^{E_G}}
\ \sum_{e,i} \psi^n c_{e,i,n} (a_{e,i} + \psi_{e} b_{e,i} ) = 0
\end{equation}
and not all the $c_{e,i,n}=0$.
But equation~(\ref{eq:shbigsum}) is equivalent to saying that
$$
w_{(e,n)} = \sum_{i=1}^{d_e} c_{e,i,n}f_{e,i}
$$ 
is a non-zero element of $H_1^\oplus(\pi^*{\cal F})$.
Hence $h_1^\twist({\cal F})\ne 0$ iff
$H_1^\oplus(\pi^*{\cal F})\ne 0$.

\end{proof}

The following is a simple graph theoretic definition that is crucial
to our proof of Lemma~\ref{lm:equalifzero}.

\begin{definition} The {\em Abelian girth} of a digraph graph, $G$, is
the girth of $G[\integers]$.
\end{definition}

Since $G[\integers]\to G$ is a covering map, 
the girth of $G[\integers]$, which is the Abelian girth of $G$,
is at least the girth of $G$.
Note also that $B_1$, the digraph with one vertex and one edge
(a self-loop), has girth one but infinite Abelian girth, i.e.,
$G[\integers]$ is a two-sided infinite path and has no cycles.
Similarly $B_2$, the digraph with one vertex and two edges,
has girth one but Abelian girth four.

\section{Proof of {Theorem~\ref{th:shmain}}}
\label{se:shequal}
We begin with the following lemma that is one of the (if not the)
technical core of this chapter.

\begin{lemma}
\label{lm:equalifzero}
Let ${\cal F}$ be a sheaf on a digraph, $G$.  Let
$\mu\from G'\to G$ be a covering map such that $G'$ is of 
Abelian girth greater than 
$$
2 \Bigl( \dim\bigl({\cal F}(V)\bigr) +
\dim\bigl({\cal F}(E)\bigr) \Bigr).
$$
Then $h_1^\twist(\mu^*{\cal F})>0$ implies that
${\rm m.e.}({\cal F})>0$.
\end{lemma}

In Subsection~\ref{sb:main_theorem}, the last
subsection of this section, we use this lemma to
prove Theorem~\ref{th:shmain}.
The rest of the subsections of this section will be devoted
to proving the lemma; our proof, whose basic idea is fairly simple,
requires a lot of new notation
and definitions.

\subsection{Outline of the Proof of Lemma~\ref{lm:equalifzero}}
\label{sb:outline}

Consider the hypotheses of Lemma~\ref{lm:equalifzero}.
Let $\pi\from G'[\integers]\to G'$ be the universal Abelian cover
of $G'$, and let ${\cal F}'=\mu^*{\cal F}$.  
We assume $h_1^\twist({\cal F}')\ge 1$, and we wish to prove that
${\rm m.e.}({\cal F})\ge 1$.
According to
Lemma~\ref{lm:habelian}, there exists a nonzero
$w\in H_1^\oplus(\pi^*{\cal F}')$ supported in $G'[\integers_{\ge 0}]$;
fix such a $w$.

Let us introduce some notation to explain the idea behind the proof.
For $e\in E_G$, we may identify ${\cal F}(e)$ 
with the subspace of ${\cal F}(E)$ supported in $e$, i.e.,
consisting of vectors
whose ${\cal F}(e')$ component vanishes for $e'\ne e$ (this subspace is
the image of ${\cal F}(e)$ under $u\mapsto {\rm extend}(u,e)$).
If $f\in E_{G'[\integers]}$, then we let $w_f$ be the $f$-component of
$w$ (as done in the proof of
Lemma~\ref{lm:habelian}), so $w_f\in
(\pi^*{\cal F}')(f)$; but $(\pi^*{\cal F}')(f)$ equals
${\cal F}(\mu\pi f)$, and can therefore be 
identified with the subset of ${\cal F}(E)$ supported
in $\mu\pi f$; 
let
$\overline{w_f}$ be the element of ${\cal F}(E)$ corresponding to $w_f$.
For $F\subset E_{G'[\integers]}$, set
$$
C(F) = {\rm span}\{ \overline{w_f}\ |\ f\in F\} \subset {\cal F}(E),
$$
$$
A(F) = {\rm span}\{ d_{{\cal F},h}
\overline{w_f}\ |\ f\in F\} 
= d_{{\cal F},h} C(F)
\subset {\cal F}(V),
$$
and
$$
B(F) = {\rm span}\{ d_{{\cal F},t}
\overline{w_f}\ |\ f\in F\} 
= d_{{\cal F},t} C(F)
\subset {\cal F}(V).
$$
Our idea is to
construct an increasing sequence of subgraphs, 
$U_1\subset\cdots\subset U_r= U$, of $G[\integers_{\ge 0}]$,
and set $F_i=E_{U_i}$, so that $F=F_r$
satisfies
\begin{equation}\label{eq:shABC}
\dim\bigl( A(F)+B(F) \bigr) \le \dim\bigl( C(F) \bigr) -1.
\end{equation}
At this point we have
$$
{\rm excess}\bigl({\cal F},A(F)+B(F)\bigr) \ge 1
$$
and the lemma is established.

The subgraphs $U_1,\ldots,U_r$ will be selected in ``phases.''
In the first phase we choose $U_1,\ldots,U_{k_1}$ for some integer 
$k_1\ge 1$.  We will show that
\begin{equation}\label{eq:shfirst_phase_summ}
\dim\bigl( A(F_{k_1}) \bigr) \le 
\dim\bigl( C(F_{k_1}) \bigr) - k_1.
\end{equation}
This inequality is worse than equation~(\ref{eq:shABC}) because it doesn't
involve $B(F_{k_1})$; however, it is possibly better, in that the 
right-hand-side has a $-k_1$ and we may have $k_1>1$.

The $i$-th phase will select $U_{k_{i-1}+1},U_{k_{i-1}+2},\ldots,U_{k_i}$
for some integer $k_i\ge k_{i-1}$.  (Hence we set $k_0=0$ for
consistency and convenience.)
The third, fifth, and all odd numbered
phases will be called C-phases, for a reason that will become clear
(see equations~(\ref{eq:shfirst_gathering}) and (\ref{eq:shith_C-phase}) 
and nearby discussion); the C-phases
select their $U_i$ in a similar way.
The second phase will be called a B-phase; in this phase we choose
$U_{k_1+1},\ldots,U_{k_2}$ to derive an equality akin to
equation~(\ref{eq:shfirst_phase_summ}) that involves $B(F_{k_1})$
(namely equation~(\ref{eq:shtail_finish}));
unfortunately, the inequality no longer involves $A(F_{k_1})$ and
$C(F_{k_2})$, rather it involves $A(F_{k_2})$ and $C(F_{k_2})$.
The fourth, sixth, and all even numbered phases will be called B-phases,
because of the way in which their $U_i$ are selected
(see equation~(\ref{eq:shith_B-phase})).

After the first two phases, i.e., the first C-phase and first B-phase,
each subsequent phase, alternating between C-phases and B-phases, allows us
to write an inequality akin to equation~(\ref{eq:shABC}) or
(\ref{eq:shfirst_phase_summ}).
The inequality after the $i$-th phase will involve the values of $A,B,C$
at $F_{k_i},F_{k_{i-1}},F_{k_{i-2}}$; roughly speaking, as $i$ gets larger,
the values of $A$, $B$, or $C$ on
$F_{k_i},F_{k_{i-1}},F_{k_{i-2}}$ must ``converge,'' since these are subspaces
of finite dimensional spaces ${\cal F}(V)$ and ${\cal F}(E)$.
At the point of ``convergence''
(more precisely, when either equation~(\ref{eq:shphases_endC}) 
or (\ref{eq:shphases_endB}) hold)
our phases end after completing the $i$-th phase, whereupon
taking $r=k_{i}$ we 
will have that $F=F_r$
satisfies equation~(\ref{eq:shABC}) and we are done.

Now we give the details.  The construction of the $U_i$ and the
inequalities we prove involve definitions of what we call
``stars'' and ``star union data,'' given in Subsection~\ref{sb:star_union}.
We shall describe the first and second phase, respectively, in detail in
Subsections~\ref{sb:first_phase} and \ref{sb:second_phase}, respectively.
In Subsection~\ref{sb:mosey} we state and prove a number of facts used in
Subsections~\ref{sb:first_phase} and \ref{sb:second_phase} in greater
generality; we hope that this greater generality will clarify the proofs.
In Subsection~\ref{sb:finish_lemma} we finish the proof of
Lemma~\ref{lm:equalifzero}.
As mentioned before, in Subsection~\ref{sb:main_theorem}, we use
Lemma~\ref{lm:equalifzero} to prove Theorem~\ref{th:shmain}.

\subsection{Star Union Data}
\label{sb:star_union}

We now fix some graph theoretic notions to describe the $U_i$, $F_i$, and
related concepts.
For a vertex, $u$, of $G'[\integers_{\ge 0}]$, let the {\em star at
$v$}, denoted ${\rm Star}(u)$,
be the subgraph of $G'[\integers_{\ge 0}]$ consisting of those
edges of $G'[\integers_{\ge 0}]$ whose head is $u$ and of those vertices
that are the endpoints of these edges
(the star at $u$ is easily seen to be a tree, since
$G'[\integers_{\ge 0}]$ has no self-loops or multiple edges).

\begin{definition}
\label{de:star_union}
For any sequence $v=(v_1,\ldots,v_j)$ of vertices of 
$G'[\integers_{\ge 0}]$, we define the {\em star union}
of $v$ to be the union of the stars at $v_1,\ldots,v_j$.  Furthermore, to
any such sequence $v=(v_1,\ldots,v_j)$ we associate the following
data,
$(U_i,F_i,I_i,X_i)_{i=1,\ldots,j}$,
that we call {\em star union data}: for positive integer $i\le j$
we associate 
\begin{enumerate}
\item the {\em $i$-th star union}, $U_i$, which is the star
union of $(v_1,\ldots,v_i)$;
\item the {\em $i$-th edge set}, $F_i=E_{U_i}$;
\item the {\em $i$-th interior edge set}, $I_i\subset F_i$, the set of
edges in $U_i$ whose tail is one of $v_1,\ldots,v_i$;
\item the {\em $i$-th interior vertex set}, $\{v_1,\ldots,v_i\}$;
and
\item the {\em $i$-th exterior vertex set}, $X_i=V_{U_i}\setminus 
\{v_1,\ldots,v_i\}$.
\end{enumerate}
\end{definition}

N.B.:
Throughout the rest of this section, the variables $U_i,F_i,I_i,X_i$
and terminology of Definition~\ref{de:star_union}, will refer to 
star union data with respect to the variable $v=(v_1,\ldots,v_j)$, 
where $j$ will change during the section.
Our goal is to construct $v=(v_1,\ldots,v_r)$ such that $F=F_r$ satisfies
equation~(\ref{eq:shABC}), but to do so will construct $v$ in phases,
and during any part of any phase the variables $U_i,F_i,I_i,X_i$ refer
to the portion of $v$ constructed so far (which limits $i$ to be at most
$j$ for the current value of $j$)

\subsection{The First C-Phase}
\label{sb:first_phase}

We remind the reader that, as explained at the end
of Subsection~\ref{sb:star_union},
$U_i,F_i,I_i,X_i$ are assumed to refer to star union data derived from
a sequence $v=(v_1,v_2,\ldots)$, at any stage of its construction.

Choose any edge, $e_1$,
of minimal rank with $\overline{w_{e_1}}\ne 0$ and let
$v_1=he_1$ and let $\rho={\rm rank}(v_1)$.
We claim 
$$
\dim(A(F_1))+1\le\dim(C(F_1)); 
$$
indeed, if $v_1$ is the tail of
an edge, $f$, then
$\overline{w_f}=0$, by the minimal rank of $e_1$.
Hence
\begin{equation}\label{eq:shnotails}
\sum_{e\ {\rm s.t.}\ he=v_1} d_h \overline{w_e} 
\ =\sum_{e\ {\rm s.t.}\ te=v_1} d_t \overline{w_e} 
\ = 0.
\end{equation}
Consider the set
$$
E^1 = \{ e\ |\ \mbox{$he=v_1$ and $\overline{w_e}\ne 0$} \}
\subset E_{G'[\integers]_{\ge 0}}.
$$
We claim that 
\begin{equation}\label{eq:shcount_cf1}
\dim\bigl( C(F_1) \bigr) = |E^1|;
\end{equation}
indeed
$$
{\cal F}(E)=\bigoplus_{e\in E_G} {\cal F}(e),
$$
and since $\mu\pi\from G'[\integers_{\ge 0}]\to G$ is a covering
map, for each $f\in E_G$ there is at most 
one $e\in E_{G'[\integers_{\ge 0}]}$ such
that $\mu e=f$ and $he=v_1$.  Hence each nonzero $w_e$ with $e\in F_1$
is taken to its own component of ${\cal F}(E)$.  So in the terminology of
Subsection~\ref{sb:compartment}, the nonzero $w_e$ are compartmentally
distinct, and hence independent, by Theorem~\ref{th:shindependent_nonzero}.
Hence equation~(\ref{eq:shcount_cf1}) holds.
By contrast, equation~(\ref{eq:shnotails}) shows that the 
$d_h\overline{w_e}$ with $e\in E^1$ sum to zero and are therefore
dependent; 
hence
$$
\dim\bigl(A(F_1)\bigr) \le |E^1|-1,
$$
and so
\begin{equation}\label{eq:shfirst_case}
\dim\bigl(A(F_1)\bigr) \le \dim\bigl(C(F_1)\bigr)-1.
\end{equation}

Assume that there is an $e_2\in E_{G'[\integers_{\ge 0}]}$
for which ${\rm rank}(e_2)=\rho$ and
$\overline{w_{e_2}}\notin C(F_1)$.  In this case the first phase
continues; we fix any such $e_2$,
set $v_2=he_2$. 
We claim that
\begin{equation}\label{eq:shsecond_step}
\dim\bigl( A(F_2)/A(F_1)\bigr) \le \dim\bigl( C(F_2)/C(F_1) \bigr)-1.
\end{equation}
Indeed, 
let $E^2$ be
the number set of $e$ such that $he=v_2$ and $\overline{w_e}
\notin C(F_1)$ (i.e., $\overline{w_e}$ is non-zero modulo $C(F_1)$).
Note that $C(F_1)$ is compartmentalized.
Also, the $\overline{w_e}$ with $e\in E^2$ are compartmentally
distinct (by the same argument as used for $E^1$, which is true when $e$
ranges over the edges of any star).
Hence,
by
Theorem~\ref{th:shindependent_nonzero},
the $\overline{w_e}$ with $e\in E^2$ are 
linearly independent in ${\cal F}(E)/C(F_1)$.
Hence
$$
\dim\bigl( C(F_2)/C(F_1) \bigr) = |E^2|.
$$
However, as with $E^1$ we have
$$
\sum_{e\in E^2} d_h \overline{w_e} = 0,
$$
since $v_2$ has rank $\rho$ (so $\overline{w_e}=0$ for all $e$ with
$te=v_2$).  But if $he=v_2$ and
$e\notin E^2$, then $\overline{w_e}\in C(F_1)$ and so
$A(\{e\})\in A(F_1)$.  Hence
$$
\sum_{e\in E^2} d_h \overline{w_e} \in A(F_1),
$$
It follows that 
$$
\dim\bigl( A(F_2)/A(F_1) \bigr) \le |E^2|-1.
$$
This establishes equation~(\ref{eq:shsecond_step}), and adding that equation
to equation~(\ref{eq:shfirst_case}) gives
$$
\dim\bigl( A(F_2) \bigr) \le \dim\bigl( C(F_2) \bigr) - 2.
$$

If there is an $e_3$ such that ${\rm rank}(e_3)=\rho$ and
$\overline{w_{e_3}}\notin C(F_2)$, then the first phase continues,
with $v_3=he_3$, and we have
$$
\dim\bigl( A(F_3) \bigr) \le \dim\bigl( C(F_3) \bigr) - 3.
$$
We similarly find $e_i$ and set $v_i=he_i$ 
for each positive integer
$i$ for which there is an $e_i$ of rank $\rho$
with $\overline{w_{e_i}}\notin C(F_{i-1})$; for any such $i$ we have
\begin{equation}\label{eq:shfirst_AC}
\dim\bigl( A(F_i) \bigr) \le \dim\bigl( C(F_i) \bigr) - i.
\end{equation}
But for any such $i$ we have 
\begin{equation}\label{eq:shk_1bound}
\dim\bigl( C(F_i) \bigr) \ge i;
\end{equation}
hence for any such $i$ we have $i\le \dim({\cal F}(E))$, and so
for some $k_1\le\dim({\cal F}(E))$ this process stops at $i=k_1$, i.e.,
we construct $e_1,\ldots,e_{k_1}$ of rank $\rho$ with 
$\overline{w_{e_i}}\notin C(F_{i-1})$ 
for $i=2,\ldots,k_1$, but $C(F_{k_1})$ contains all
$\overline{w_e}$ for ${\rm rank}(e)=\rho$.
This is the end of the first phase.

A concise way to describe the first phase is that 
we choose any minimal
$v_1,\ldots,v_{k_1}$ 
of rank $\rho$ such that
\begin{equation}\label{eq:shfirst_gathering}
\forall \mbox{$e\in E_{G[\integers_{\ge 0}]}$ of rank $\rho$}, \quad
\overline{w_e}\in C(F_{k_1}),
\end{equation}
where minimal means that if we discard any $v_i$ from
$v_1,\ldots,v_{k_1}$ then equation~(\ref{eq:shfirst_gathering}) does not
hold.
We call this a C-phase because the equation~(\ref{eq:shfirst_gathering})
involves a ``C,'' as will all odd numbered phases.
Notice that equation~(\ref{eq:shfirst_AC}) is somewhat similar to
our desired equation~(\ref{eq:shABC}); one big difference is that 
equation~(\ref{eq:shfirst_AC}) makes no mention of $B$, but only
of $A$ and $C$.

\subsection{Moseying Sequences}
\label{sb:mosey}

Before describing the second phase, i.e., the first B-phase,
we wish to organize the inequalities we will need into a number of
lemmas.  Furthermore, we will usually state these lemmas in a slightly
more general context; this will help illustrate exactly what assumptions
are being used.  

We consider the setup and notation of the first two paragraphs of
Subsection~\ref{sb:outline}, which fixes ${\cal F}$, $\mu\from G'\to G$,
$\pi\from G'[\integers_{\ge 0}]\to G'$,
$w\in H_1^\oplus(\pi^*\mu^*{\cal F})$, and defines 
$\overline{w_f}$ for any $f\in E_{G'[\integers_{\ge 0}]}$, and defines
$A(F),B(F),C(F)$ for any $F\subset E_{G'[\integers_{\ge 0}]}$.

We will work with a sequence of vertices,
$v=(v_1,\ldots,v_s)$, of $G'[\integers_{\ge 0}]$, but we will
not assume the $v_i$ are constructed by our phases.  Instead, we will
be careful to write down our assumptions on the $v_i$ in a way that will
make clear which of their properties is used when and how.
Our central definition in this general context will
be that of a ``moseying sequence.''  

\begin{definition} 
By a {\em moseying sequence of length
$s$ for $G'$}
we mean a sequence 
$v=(v_1,\ldots,v_s)$ of distinct vertices of $G'[\integers]$ 
for which
${\rm rank}(v_{i+1})-{\rm rank}(v_i)$ is $0$ or $1$ for each $i$; if this
difference is $1$ we say that {\em $v$ jumps at $i$}.
We define star union data, $U_i,F_i,I_i,X_i$
as in Subsection~\ref{sb:star_union}.
For ease of notation we define $U_0,F_0,I_0,X_0$ to be
empty (i.e., $U_0$ is the empty graph,
$F_0,I_0,X_0$ the empty set).
\end{definition}
Moseying sequences are our basic object of study.

\begin{definition}
A moseying sequence, $v$, of length $s$ is {\em of increasing dimension} if
the integers
$$
n_i = \dim\bigl(C(F_i)\bigr) + \dim\bigl( B(I_i) \bigr)
$$
satisfy
$$
0=n_0<n_1<n_2<\cdots<n_s.
$$
\end{definition}

\begin{lemma}\label{lm:no_cycles} 
Let $v$ be a moseying sequence of length $s$
of increasing dimension for a digraph,
$G'$.  
Then
$$
s\le \dim\bigl( {\cal F}(E) \bigr) + \dim\bigl( {\cal F}(V) \bigr).
$$
Furthermore,
for any $i\le s$, $U_i$ has no cycles provided
that the girth of $G'[\integers]$ is at least $2i+1$.
\end{lemma}
\begin{proof}
The first statement is clear.  For the second statement,
assume, to the contrary, that $U_i$ has a cycle.
$U_i$ is the union of stars, which are trees of diameter two.  If $c$ is
a cycle in $U_i$ of minimal length, then it traverses each vertex at most
once.  But every vertex of $c$ not appearing in $v$ must be a leaf 
(i.e., tail of an edge) of a
star, and hence followed by (and preceded by) a vertex in $v$.  Hence
the length of $c$ is at most twice $i$.  Hence $G'[\integers]$ has a cycle
of length at most $2i$, contradicting the hypotheses of the lemma.
\end{proof}

The inequality in
equation~(\ref{eq:shfirst_AC}), derived after the first C-phase, will
be built up along further phases to eventually give
equation~(\ref{eq:shABC}).  However, to express these later phase inequalities,
we shall need some graph theoretic notions, such as the ``overdegree''
and ``capacity'' that we now define.

\begin{definition} 
Let $v$ be a moseying sequence of length $s$ for $G'$.
For any $u\in V_{G'[\integers]}$ we define the 
{\rm stable outdegree
of $u$}, denoted ${\rm sod}(u)$, to be the outdegree of $u$ in $U_s$.
(If $v$ is not a vertex of $U_s$, we define its outdegree in $U_s$ to be
zero.)
\end{definition}
Note that the outdegree of $u$ in $U_{j-1}$, viewed as a function of $j$,
does not change as soon as
${\rm rank}(v_j)\ge {\rm rank}(u)$; indeed, the edges that affect the
outdegree of $u$ are the edges of rank equal to ${\rm rank}(u)-1$,
and such edges come from stars about vertices of ${\rm rank}(u)-1$.
Hence, for any $j$ with $1\le j\le s$, we have
\begin{equation}\label{eq:shsod}
{\rm rank}(v_j)\ge {\rm rank}(u) \implies
{\rm sod}(u) = {\rm outdeg}(U_{j-1},u),
\end{equation}
where ${\rm outdeg}(G,w)$ denotes the outdegree of $w$ in $G$.
In particular,
$$
{\rm sod}(v_j)={\rm outdeg}(U_{j-1},v_j)
$$
for all $j=1,\ldots,s$.

\begin{definition} 
Let $v$ be a moseying sequence of length $s$ for $G'$.
By the overdegree of $U_i$, for an integer, $i$ with $1\le i\le s$,
we mean
$$
{\rm Over}(U_i) = \sum_{v\in X_i}
\bigl( {\rm outdeg}(U_i,v) -1 \bigr),
$$
\end{definition}
Notice that for any $i$, the overdegree of $U_i$ is non-negative,
since each exterior vertex of $U_i$ is the tail of some edge in $U_i$,
and hence has outdegree at least one.

\begin{definition}
Let $v$ be a moseying sequence of length $s$ for $G'$.
For non-negative integer, $i\le s$, we define the {\em capacity of $U_i$}
to be
$$
{\rm Cap}(U_i) = h_0(U_i)+{\rm Over}(U_i).
$$
\end{definition}
Note that for $i\ge 1$, 
$h_0(U_i)\ge 1$, since $U_i$ is nonempty, and ${\rm Over}(U_i)\ge 0$;
hence for $i\ge 1$ we have
${\rm Cap}(U_i)\ge 1$.
Our
fundamental inequalities will use the capacity.

\begin{lemma}\label{lm:capacity}
Let $v$ be a moseying sequence of length $s$ for $G'$.
Assume that $U_j$ has no cycles for some $j\le s$.  Then for any
non-negative integers $i\le j$ we have
$$
{\rm Cap}(U_j) = {\rm Cap}(U_i) - \sum_{m=i+1}^j
\bigl( {\rm sod}(v_m) - 1 \bigr)
$$
\end{lemma}
\begin{proof}
It suffices
to prove the lemma for $j=i+1$, for then the general lemma follows
by induction on $j-i$.

So assume $j=i+1$, and set $\rho={\rm rank}(v_{i+1})$.  
Let $p_0$ and $p_1$, respectively,
be the number of vertices of rank $\rho$ and $\rho+1$, respectively,
in which the
star of $v_{i+1}$ intersects $U_i$; so $p_0$ is $1$ or $0$ according
to whether or not $v_{i+1}\in V_{U_i}$, and $p_1$ is the number of tails
of edges in ${\rm Star}(v_{i+1})$ that lie in $U_i$; let $p=p_0+p_1$. 
First, note that since $U_{i+1}=U_i\cup{\rm Star}(v_{i+1})$, we have
$$
\chi(U_{i+1}) = \chi(U_i)+\chi\bigl({\rm Star}(v_{i+1})\bigr) - 
\chi\bigl(U_i\cap
{\rm Star}(v_{i+1})\bigr);
$$
since $U_i$, $U_{i+1}$, and any star have $h_1=0$, in the above equation we
may replace each $\chi$ with $h_0$, and conclude that
$$
h_0(U_{i+1}) = h_0(U_i)+h_0\bigl({\rm Star}(v_{i+1})\bigr) - 
h_0\bigl(U_i\cap
{\rm Star}(v_{i+1})\bigr);
$$
since $U_i\cap {\rm Star}(v_{i+1})$ contains no edges, it has
$p$ connected components ($p$ isolated vertices), and hence
\begin{equation}\label{eq:shph0}
h_0(U_{i+1})=h_0(U_i)+1-p.
\end{equation}
Second, note that each of the $p_1$ tails of
edges of the star adds one to its degree in $U_{i+1}$ over that of
$U_i$; the remaining tails of star edges have degree one in $U_{i+1}$.
This means that $U_{i+1}$ gains $p_1$ over $U_i$
in the overdegree contribution from vertices of rank $\rho+1$.
Third, note that $p_0=1$ iff $v_{i+1}\in V_{U_i}$ iff $v_{i+1}$ contributes
$$
{\rm outdeg}(U_i,v_{i+1})-1={\rm sod}(v_{i+1})-1
$$
to the overdegree 
of $U_i$; if so, this contribution
is lost in $U_{i+1}$, since $v_{i+1}$ becomes an interior vertex.
Hence if $p_0=0$ we have
$$
{\rm Over}(U_{i+1})={\rm Over}(U_i)+p_1
$$
and if $p_0=1$ we have
$$
{\rm Over}(U_{i+1})={\rm Over}(U_i)+p_1-({\rm sod}(v_{i+1})-1);
$$
in both cases we may write
$$
{\rm Over}(U_{i+1})={\rm Over}(U_i)+p-{\rm sod}(v_{i+1}).
$$
Combining this with equation~(\ref{eq:shph0}) yields
$$
{\rm Cap}(U_{i+1})={\rm Cap}(U_i)+1-{\rm sod}(v_{i+1}),
$$
which proves the lemma for $j=i+1$ and therefore, as explained earlier,
for all $j>i$.
\end{proof}

\begin{lemma} 
\label{lm:C-phase_est}
Let $v$ be a moseying sequence of length $s$ for $G'$.
Assume that $v$ jumps at an integer $i<s$, but not at
$i+1,i+2,\ldots,k$ for some integer $k\le s$.
(We adopt the convention that $v$ jumps at $i$ if $i=0$.)
Assume that for each edge, $e$, of $G'[\integers]$ of rank at most
${\rm rank}(v_i)$ we have $\overline{w_e}\in C(F_i)$.
Then for any $j$ with $i+1\le j\le k$ we have
\begin{equation}\label{eq:shC-phase_ineq}
\dim\Bigl( A(F_k)/\bigl( A(F_j)+B(F_i) \bigr) \Bigr) 
\le
\dim\bigl( C(F_k)/C(F_j) \bigr) - (k-j).
\end{equation}
\end{lemma}
We remark that the assumptions of this lemma are highly restrictive;
to apply this to our phases, $i+1$ (or $v_{i+1}$)
will have to be the beginning
of a B-phase, and $k$ (or $v_k$) will lie either in that B-phase or
the C-phase immediately thereafter.  Also, if $v$ jumps somewhere between
$i+1$ and $k$, then we cannot expect equation~(\ref{eq:shC-phase_ineq}) to
hold unless $B(F_i)$ is replaced with $B(F_{i'})$ for an $i'>i$.
\begin{proof}
For $j=k$ the lemma is immediate.  Let us first establish
the case $k=j+1$; the general case will then easily follow by induction
on $k-j$.  Let $\rho={\rm rank}(v_i)$.

Consider that
$$
\sum_{te=v_{j+1}} d_t\overline{w_e}
= \sum_{he=v_{j+1}} d_h\overline{w_e} .
$$
We have $d_t\overline{w_e}\in B(F_{i})$ for all $e$ with 
$te=v_{j+1}$, and, more generally, for any $e$ of rank $\rho$,
since $\overline{w_e}\in C(F_{i})$.
Hence
\begin{equation}\label{eq:shinbfi}
\sum_{he=v_{j+1}} d_h\overline{w_e} \in B(F_i).
\end{equation}
Now, as before, let
$E'$ be those $e$ with $he=v_{j+1}$ and $\overline{w_e}\notin C(F_j)$,
and let $E''$ be the same but with $\overline{w_e}\in C(F_j)$.
We have
$$
\dim\bigl( C(F_{j+1})/ C(F_{j}) \bigr)  = |E'|,
$$
since $C(F_j)$ is 
a compartmentalized subspace of ${\cal F}(E)$; yet for
$e\in E''$ we have $d_h\overline{w_e}\in A(F_j)$ and hence
$$
\sum_{e\in E''}d_h\overline{w_e} \in A(F_j),
$$
which implies, along with equation~(\ref{eq:shinbfi}) that
$$
\sum_{e\in E'}d_h\overline{w_e} =
\sum_{he=v_{j+1}} d_h\overline{w_e} - 
\sum_{e\in E''}d_h\overline{w_e} \in B(F_i)+A(F_j).
$$
Hence the $d_h\overline{w_e}$ ranging over $e\in E'$ are linearly
depedent modulo $A(F_j)+B(F_i)$, and so
$$
\dim\Bigl( A(F_{j+1})\bigm/ \bigl(A(F_{j})+B(F_{i})\bigr) \Bigr) \le |E'|-1.
$$
Hence 
\begin{equation}\label{eq:shbase_case}
\dim\Bigl( A(F_{j+1})\bigm/ \bigl(A(F_{j})+B(F_{i})\bigr) \Bigr) \le 
\dim\bigl( C(F_{j+1})/ C(F_{j}) \bigr)  - 1.
\end{equation}
This establishes the case $k=j+1$ of the lemma.

The general case of the lemma now follows from the fact that
$F_j$ and hence $C(F_j)$ are increasing in $j$, and hence
$$
\dim\bigl( C(F_{k})/ C(F_{j}) \bigr) = 
\sum_{m=j}^{k-1}\dim\bigl( C(F_{m+1})/ C(F_{m}) \bigr) ;
$$
similarly the spaces $A(F_j)$ modulo $B(F_i)$, i.e., viewed as subspaces
of ${\cal F}(V)/B(F_i)$, are increasing in $j$, and hence
$$
\dim\Bigl( A(F_{k})\bigm/ \bigl(A(F_{j})+B(F_{i})\bigr) \Bigr) =
\sum_{m=j}^{k-1}\dim\Bigl( A(F_{m+1})\bigm/ 
\bigl(A(F_{m})+B(F_{i})\bigr) \Bigr) .
$$
Hence applying equation~(\ref{eq:shbase_case}) with $m$ replacing $j$
and $m$ over the range $j,j+1,\ldots,k-1$ yields the lemma.
\end{proof}

\begin{lemma}
\label{lm:B-phase_est}
Let $v$ be a moseying sequence of length $s$ for $G'$.
Then for non-negative integers $i\le j\le s$ we have
$$
\dim\bigl( B(I_j)/B(I_i) \bigr)
\le
\sum_{m=i+1}^j {\rm sod}(v_m).
$$
\end{lemma}
\begin{proof}
Clearly $B(I_j)/B(I_i)$ is at most the size of $I_j\setminus I_i$.
But an edge, $e$, of $G'[\integers]$, lies in
$I_j\setminus I_i$ (viewing $I_i\subset I_j$ as subsets
of $E_{G'[\integers]}$) precisely when $te=v_m$ for some $m$ between
$i+1$ and $j$; furthermore,
for each such $m$, the number of $e$ with $te=v_m$ in $U_j$
is ${\rm outdeg}(U_j,v_m)$.  Hence
$$
\dim\bigl( B(I_j)/B(I_i) \bigr)
\le
\sum_{m=i+1}^j {\rm outdeg}(U_j,v_m).
$$
But ${\rm outdeg}(U_j,v_m)
={\rm sod}(v_m)$, either by definition, if $j=s$ or, if $j<s$,
in view of equation~(\ref{eq:shsod}) and the fact
that ${\rm rank}(v_{j+1})\ge
{\rm rank}(v_m)$.
Hence the lemma follows.
\end{proof}

\subsection{The First B-Phase}
\label{sb:second_phase}

At this point we have finished the first C-phase, having constructed
$v_1,\ldots,v_{k_1}$.
If 
\begin{equation}\label{eq:shend_first_C}
B(F_{k_1})\subset A(F_{k_1}),
\end{equation}
then we are done, for then $F=F_{k_1}$
satisfies equation~(\ref{eq:shABC}), in view of
equation~(\ref{eq:shfirst_AC}) with $i=k_1$.  In this case we end our phases,
and Lemma~\ref{lm:equalifzero} is finished in
this case.
Otherwise $B(F_{k_1})$ is not entirely contained in $A(F_{k_1})$. 
At this point
we enter the second
phase; the rough idea is to generate an inequality similar to
equation~(\ref{eq:shfirst_AC}), but which involves
$B(F_{k_1})$;
this will come at the
expense of making the $A$ and $C$ terms involve $F_{k_2}$ as opposed to
$F_{k_1}$.

We will choose $v_{k_1+1},\ldots,v_{k_2}$ minimal with
\begin{equation}\label{eq:shB-phase}
B(F_{k_1})\subset A(F_{k_1})+B(I_{k_2}),
\end{equation}
which we do as follows: choose any $e\in F_{k_1}$ with
$d_t\overline{w_e}\notin A(F_{k_1})$, and set $v_{k_1+1}=te$;
then $d_t\overline{w_e}\in B(I_{k_1+1})$; 
then choose any $e'\in F_{k_1}$ with
$d_t\overline{w_{e'}}\notin A(I_{k_1})+B(I_{k_1+1})$ and take
$v_{k_1+2}=te'$ if such an $e'$ exists; continuing on in this fashion
we generate a new vertices $v_i$ until we reach a vertex $v_{k_2}$
such that
$$
\forall e\in F_{k_1},\quad d_t\overline{w_e}\in A(F_{k_1})+B(I_{k_2});
$$
such a point is reached, since we have
proper containments
\begin{equation}\label{eq:shcontainments}
A(F_{k_1}) \subset A(F_{k_1})+B(I_{k_1+1}) \subset
A(F_{k_1})+B(I_{k_1+2}) \subset \cdots
\end{equation}
which are subsets of the finite dimensional space ${\cal F}(V)$.
Hence this point is reached with
$$
k_2-k_1 \le \dim\bigl( {\cal F}(V) \bigr),
$$
and since $k_1\le \dim( {\cal F}(V) )$ (see equation~(\ref{eq:shk_1bound})
and the discussion below it), we have
\begin{equation}\label{eq:shk_2bound}
k_2 \le \dim\bigl( {\cal F}(V) \bigr) + \dim\bigl( {\cal F}(E) \bigr).
\end{equation}
The choice of $v_{{k_1}+1},\ldots,v_{k_2}$
comprises the second phase; we call this a (the first)
B-phase because of the prominence of the letter ``B'' in
equation~(\ref{eq:shB-phase}).
Now we combine a number of inequalities from Subsection~\ref{sb:mosey}
to prove a sequel to equation~(\ref{eq:shfirst_AC}).

First, Lemma~\ref{lm:capacity} with $j=v_{2k}$ and $i=0$ (for which the
lemma is still valid) shows that
\begin{equation}\label{eq:shfirst_tail}
{\rm Cap}(U_{k_2})= k_2 - \sum_{m=1}^{k_2} {\rm sod}(v_m)
\end{equation}
(note that $U_{k_2}$ has no cycles, using Lemma~\ref{lm:no_cycles}).
Second, Lemma~\ref{lm:C-phase_est} with 
$k=k_2$ and $i=j=k_1$ yields
\begin{equation}\label{eq:shsecond_tail}
\dim\bigl( A(F_{k_2})/ (A(F_{k_1})+B(F_{k_1})) \bigr) \le
\dim\bigl( C(F_{k_2})/ C(F_{k_1}) \bigr)  - (k_2-k_1).
\end{equation}
Third, 
we have $I_{k_1}=\emptyset$ since $v_1,\ldots,v_{k_1}$ are all of
rank $\rho$.  
Hence
Lemma~\ref{lm:B-phase_est} with $j=k_2$ and $i=k_1$ gives
\begin{equation}\label{eq:shthird_tail}
\dim\bigl( B(I_{k_2}) \bigr) =
\dim\bigl( B(I_{k_2})/B(I_{k_1}) \bigr) \le
\sum_{i=k_1+1}^{k_2} {\rm sod}(v_i).
\end{equation}

We have now established three inequalities in
equations~(\ref{eq:shfirst_tail}), 
(\ref{eq:shsecond_tail}), and
(\ref{eq:shthird_tail}).
We now establish a simple
inequality to describe the end of the first B-phase.

Equations~(\ref{eq:shthird_tail}) and (\ref{eq:shfirst_AC}) with $i=k_1$
imply that
$$
\dim\bigl( A(F_{k_1}) + B(I_{k_2}) \bigr) \le \dim\bigl( C(F_{k_1}) \bigr) 
-k_1
+\sum_{i=k_1+1}^{k_2} {\rm sod}(v_i),
$$
and in view of equation~(\ref{eq:shB-phase}) this implies that
$$
\dim\bigl( A(F_{k_1}) + B(F_{k_1}) \bigr) \le \dim\bigl( C(F_{k_1}) \bigr) 
-k_1
+\sum_{i=k_1+1}^{k_2} {\rm sod}(v_i),
$$

Equation~(\ref{eq:shsecond_tail}) added to this gives
$$
\dim\bigl( A(F_{k_2}) + B(F_{k_1}) \bigr) 
$$
$$
\le \dim\bigl( C(F_{k_2}) \bigr) 
-k_2+
\sum_{i=k_1+1}^{k_2} {\rm sod}(v_i)
$$
$$
= \dim\bigl( C(F_{k_2}) \bigr) 
-k_2+
\sum_{i=1}^{k_2} {\rm sod}(v_i)
$$
(since ${\rm sod}(v_i)=0$ for $i=1,\ldots,k_1$)
$$
= \dim\bigl( C(F_{k_2}) \bigr) - 
\sum_{i=1}^{k_2} \bigl({\rm sod}(v_i) -1 \bigr).
$$
Then using equation~(\ref{eq:shfirst_tail}) we get
\begin{equation}\label{eq:shtail_finish}
\dim\bigl( A(F_{k_2}) + B(F_{k_1}) \bigr) 
\le \dim\bigl( C(F_{k_2}) \bigr) 
-{\rm Cap}(U_{k_2}).
\end{equation}
This equation is all we need to know about the B-phase we have just
finished.

If 
\begin{equation}\label{eq:shend_first_B}
B(F_{k_2})\subset A(F_{k_2}) + B(F_{k_1}),
\end{equation}
then our phases are
over and we easily establish Lemma~\ref{lm:equalifzero}: indeed, we have
$$
\dim\bigl( A(F_{k_2}) + B(F_{k_2}) \bigr)
=\dim\bigl( A(F_{k_2}) + B(F_{k_1}) \bigr)
\le \dim\bigl( C(F_{k_2}) \bigr) -
1
$$
since ${\rm Cap}(U_{k_2})\ge 1$ 
(indeed, $h_0(U_{k_2})\ge 1$ and the overdegree is non-negative).
Hence we have established equation~(\ref{eq:shABC}) with
$F=F_{k_2}$ and we
are done.

Otherwise we undergo a second C-phase, possibly a second B-phase,
possibly a third C-phase, etc.
So for $i=2,3,\ldots$, the $(2i-1)$-th phase, or $i$-th C-phase,
adds vertices $v_{k_{2i-2}+1},\ldots,v_{k_{2i-1}}$ of rank $\rho+i-1$
so that
\begin{equation}\label{eq:shith_C-phase}
\forall \mbox{$e\in E_{G[\integers_{\ge 0}]}$ of rank $\rho+i-1$}, \quad
\overline{w_e}\in C(F_{k_{2i-1}})
\end{equation}
(for $j \ge k_{2i-1}+1$ we
successively add a vertex $v_j$ which is the head of an edge, $e$, of
rank $\rho+i-1$ for
which $\overline{w_e}\notin C(F_j)$, augmenting $j$ until no such edges exist);
the $(2i)$-th phase, or the $i$-th B-phase, adds
$v_{k_{2i-1}+1,\ldots,k_{2i}}$ so that
\begin{equation}\label{eq:shith_B-phase}
B(F_{k_{2i-1}}) \subset A(F_{k_{2i-1}})+B(I_{k_{2i}});
\end{equation}
as in the first B-phase, the $i$-th B-phase selects its vertices
by choosing an $e\in F_{k_{2i-1}}$ for which
$$
d_t\overline{w_e}\notin A(F_{k_{2i-1}})+B(I_{k_{2i-1}}),
$$
setting $v_{k_{2i-1}+1}=te$; then choosing an $e'\in F_{k_{2i-1}}$ for which
$$
d_t\overline{w_{e'}}\notin A(F_{k_{2i-1}})+B(I_{k_{2i-1}+1}),
$$
setting $v_{k_{2i-1}+2}=te'$; then repeating this procedure until
reaching $v_{k_{2i}}$ such that for all $e\in F_{k_{2i-1}}$ we have
$$
d_t\overline{w_e}\in A(F_{k_{2i-1}})+B(I_{k_{2i}}),
$$
whereupon equation~(\ref{eq:shith_B-phase}) holds (minimally, i.e., it would
fail to hold if we omitted any vertex, $v_m$, added during this phase).

The phases end either at the end of a C-phase or B-phase as follows:
the phases end at the $j$-th C-phase for $j\ge 1$ when 
\begin{equation}\label{eq:shphases_endC}
B(F_{k_{2j-1}}) \subset A(F_{k_{2j-1}})+B(F_{k_{2j-3}})
\end{equation}
(with $k_{-1}=0$ and so $F_{k_{-1}}=\emptyset$ for the case $j=1$),
which restricts to equation~(\ref{eq:shend_first_C}) for $j=1$;
the phases end at the $j$-th B-phase for $j\ge 1$ when 
\begin{equation}\label{eq:shphases_endB}
B(F_{k_{2j}}) \subset A(F_{k_{2j}})+B(F_{k_{2j-1}}),
\end{equation}
which restricts to equation~(\ref{eq:shend_first_B}) for $j=1$.
In the next subsection
show that one of these two conditions
eventually holds for some finite $j$, and that
$F=F_r$ with $r={k_{2j}}$ satisfies equation~(\ref{eq:shABC}).
We already have all the main inequalities needed to prove this,
and just need to apply them to the phases beyond the second phase.

\subsection{End of the Proof of Lemma~\ref{lm:equalifzero}}
\label{sb:finish_lemma}

\begin{proof}[Proof of Lemma~\ref{lm:equalifzero}]
Now we 
claim that, for all $i\ge 1$, at the end of the $i$-th C-phase we have
\begin{equation}\label{eq:shtail_induction_odd}
\dim\bigl( A(F_{k_{2i-1}}) + B(F_{k_{2i-3}}) \bigr) 
\le \dim\bigl( C(F_{k_{2i-1}}) \bigr) - 
{\rm Cap}(U_{k_{2i-2}})-(k_{2i-1}-k_{2i-2})
\end{equation}
(for $i=1$ we understand that
$k_{-1}=k_0=0$ and $F_{0}=\emptyset$),
and that, for all $i\ge 1$, at the end of the $i$-th B-phase we have
\begin{equation}\label{eq:shtail_induction_even}
\dim\bigl( A(F_{k_{2i}}) + B(F_{k_{2i-1}}) \bigr) 
\le \dim\bigl( C(F_{k_{2i}}) \bigr) - 
{\rm Cap}(U_{k_{2i}}).
\end{equation}
We shall prove these by induction.
To do so, first note that after $i$ phases we produce a sequence
$v=(v_1,\ldots,v_{k_{i}})$ that is of increasing dimension, since each
$v_m$ of a C-phase increases $\dim(C(F_m))$ by at least one, and
each $v_m$ of a B-phase increases $\dim(B(F_m))$ by at least one.
Hence, according to Lemma~\ref{lm:no_cycles}, 
\begin{equation}\label{eq:shk2i_bound}
k_{i} \le \dim\bigl( {\cal F}(V) \bigr) +
\dim\bigl( {\cal F}(E) \bigr) ,
\end{equation}
and
$U_{k_{i}}$ contains
no cycles, using the hypotheses of Lemma~\ref{lm:equalifzero}.

Let us also note that the phases eventually end.  Indeed, if
$k_{2j}=k_{2j-1}$, then according to equation~(\ref{eq:shphases_endB})
we finish.  Hence, we are not done by the $j$-th B-phase we have
$$
k_{2j}> k_{2j-1}\ge k_{2j-2}> k_{2j-3} \ge \cdots\ge k_2>k_1\ge 1,
$$
so $k_{2j}\ge j+1$; in view of equation~(\ref{eq:shk2i_bound}), the total
number of phases is less than
$$
2\Bigl( \dim\bigl( {\cal F}(V) \bigr) +
\dim\bigl( {\cal F}(E) \bigr) \Bigr).
$$

Equation~(\ref{eq:shtail_induction_even}) has been established for
$i=1$ in equation~(\ref{eq:shtail_finish}).  So let us first
show that
equation~(\ref{eq:shtail_induction_even}) implies
equation~(\ref{eq:shtail_induction_odd}) with $i$ replaced by $i+1$.

So assume equation~(\ref{eq:shtail_induction_even}) for some $i\ge 1$.  
By Lemma~\ref{lm:C-phase_est}, since $v$ jumps at $k_{2i-1}$ but does
not jump thereafter until $k_{2i+1}$,  we have
$$
\dim\bigl( A(F_{k_{2i+1}})/ (A(F_{k_{2i}})+B(F_{k_{2i-1}})) \bigr) 
$$
$$
\le
\dim\bigl( C(F_{k_{2i+1}})/ C(F_{k_{2i}}) \bigr)  - (k_{2i+1}-k_{2i}).
$$
Adding this to equation~(\ref{eq:shtail_induction_even}) yields
$$
\dim\bigl( A(F_{k_{2i+1}}) + B(F_{k_{2i-1}}) \bigr) 
\le \dim\bigl( C(F_{k_{2i+1}}) \bigr) - 
{\rm Cap}(U_{k_{2i}})-(k_{2i+1}-k_{2i}).
$$
This is equation~(\ref{eq:shtail_induction_odd}), with $i$ replaced by
$i+1$.

Finally assume equation~(\ref{eq:shtail_induction_odd}) for some value of
$i\ge 1$; we shall
conclude that equation~(\ref{eq:shtail_induction_even}) holds for the same
value of $i$.
By Lemma~\ref{lm:B-phase_est} we have
$$
\dim\bigl( B(I_{k_{2i}})/B(I_{k_{2i-2}}) \bigr)
\le
\sum_{m=k_{2i-2}+1}^{k_{2i}} {\rm sod}(v_m).
$$
This implies that
\begin{equation}\label{eq:shB_quotient}
\dim\Bigl( \bigl( A(F_{k_{2i-1}})+B(I_{k_{2i}}) \bigr) \bigm/
\bigl( A(F_{k_{2i-1}})+B(I_{k_{2i-2}}) \bigr) \Bigr) \le
\sum_{m=k_{2i-2}+1}^{k_{2i}} {\rm sod}(v_m).
\end{equation}

In view of equation~(\ref{eq:shith_B-phase}), and since
$I_{k_{2i}}\subset F_{k_{2i-1}}$, we have
\begin{equation}\label{eq:shfirstAB}
A(F_{k_{2i-1}})+B(F_{k_{2i-1}}) = A(F_{k_{2i-1}})+B(I_{k_{2i}});
\end{equation}
similarly we have
$$
A(F_{k_{2i-3}})+B(F_{k_{2i-3}}) = A(F_{k_{2i-3}})+B(I_{k_{2i-2}})
$$
and therefore
\begin{equation}\label{eq:shsecondAB}
A(F_{k_{2i-1}})+B(F_{k_{2i-3}}) = A(F_{k_{2i-1}})+B(I_{k_{2i-2}})
\end{equation}
Given equations~(\ref{eq:shfirstAB}) and (\ref{eq:shsecondAB}),
equation~(\ref{eq:shB_quotient}) can be rewritten as
\begin{equation}\label{eq:shB_quotient2}
\dim\Bigl( \bigl( A(F_{k_{2i-1}})+B(F_{k_{2i-1}}) \bigr) \bigm/
\bigl( A(F_{k_{2i-1}})+B(F_{k_{2i-3}}) \bigr) \Bigr) \le
\sum_{m=k_{2i-2}+1}^{k_{2i}} {\rm sod}(v_m).
\end{equation}
Adding this to equation~(\ref{eq:shtail_induction_odd}) gives
$$
\dim\bigl( A(F_{k_{2i-1}})+B(F_{k_{2i-1}}) \bigr) 
$$
$$
\le \dim\bigl( C(F_{k_{2i-1}}) \bigr) - 
{\rm Cap}(U_{k_{2i-2}})-(k_{2i-1}-k_{2i-2})
+
\sum_{m=k_{2i-2}+1}^{k_{2i}} {\rm sod}(v_m)
$$
$$
= \dim\bigl( C(F_{k_{2i-1}}) \bigr) -
{\rm Cap}(U_{k_{2i}}) + (k_{2i}-k_{2i-1})
$$
in view of Lemma~\ref{lm:capacity} with $i,j$ respectively set to
$k_{2i-2},k_{2i}$.
Adding this to Lemma~\ref{lm:C-phase_est} with $i,j,k$ respectively replaced
with $k_{2i-2},k_{2i-1},k_{2i}$ yields
$$
\dim\bigl( A(F_{k_{2i}})+B(F_{k_{2i-1}}) \bigr) \le
\dim\bigl( C(F_{k_{2i}}) \bigr) -
{\rm Cap}(U_{k_{2i}}).
$$
This proves equation~(\ref{eq:shtail_induction_even}).

At this point we have established
equations~(\ref{eq:shtail_induction_odd}) and
(\ref{eq:shtail_induction_even}), and the fact that the phases
eventually end.  Now we claim that Lemma~\ref{lm:equalifzero}
easily follows.  Indeed, if our phases end at the $j$-th B-phase, then
$$
B(F_{k_{2j}})\subset A(F_{k_{2j}})+B(F_{k_{2j-1}}),
$$
and so equation~(\ref{eq:shtail_induction_even}) gives
$$
\dim\bigl( A(F_{k_{2j}}) + B(F_{k_{2j}}) \bigr)
\le \dim\bigl( C(F_{k_{2j}}) \bigr) -
{\rm Cap}(U_{k_{2j}}).
$$
Since $U_{k_{2j}}$ is non-empty, its capacity is at least one, and
hence $F=F_r$ with $r=k_{2j}$ satisfies equation~(\ref{eq:shABC}).
Similarly, if our phases end at the $j$-th C-phase, then
$$
B(F_{k_{2j-1}})\subset A(F_{k_{2j-1}})+B(F_{k_{2j-3}}),
$$
and so equation~(\ref{eq:shtail_induction_odd}) gives
$$
\dim\bigl( A(F_{k_{2j-1}}) + B(F_{k_{2j-1}}) \bigr)
\le \dim\bigl( C(F_{k_{2j-1}}) \bigr) - 1,
$$
since 
$$
{\rm Cap}(U_{k_{2j-2}})+(k_{2j-1}-k_{2j-2}) \ge 1
$$
(for $j=1$ this follows since $k_1>0$, and for $j\ge 2$ this follows
since $U_{k_{2j-2}}$ is nonempty).
Hence, similarly,
$F=F_r$ with $r=k_{2j-1}$ satisfies equation~(\ref{eq:shABC}).
\end{proof}

\subsection{Proof of Theorem~\ref{th:shmain}}
\label{sb:main_theorem}

\begin{proof}[Proof (of Theorem~\ref{th:shmain})]
First we will verify Theorem~\ref{th:shmain} in some special cases.

Lemma~\ref{lm:equalifzero} establishes
Theorem~\ref{th:shmain} in the case where
${\rm m.e.}({\cal F})=0$.

\begin{definition} A sheaf, ${\cal E}$, on a digraph, $G$, is
{\em edge supported} if ${\cal E}(V)=0$.
\end{definition}
For an edge supported sheaf, ${\cal E}$, it is immediate that
for any covering map $\phi\from G'\to G$ we have
$$
h_1^\twist(\phi^*{\cal E})={\rm m.e.}(\phi^*{\cal E})
=\deg(\phi)\dim({\cal E}(E)).
$$
This establishes Theorem~\ref{th:shmain} in the case where ${\cal F}$ is
edge supported and $\phi$ is any covering map.

Next we introduce a type of sheaf which will be an important tool.
\begin{definition} A sheaf, ${\cal F}$, on a graph $G$, is said to be
{\em tight} if the maximum excess of ${\cal F}$ occurs at and only
at ${\cal F}(V)$.
\end{definition}
\begin{lemma}
\label{lm:tight1} 
For any sheaf, ${\cal F}$, on a digraph, $G$, there is a tight sheaf,
${\cal F}'$, that is a subsheaf of ${\cal F}$, such that
${\rm m.e.}({\cal F}')={\rm m.e.}({\cal F})$.
Furthermore, let ${\cal F}'\subset{\cal F}$ be sheaves on a graph, $G$,
with $-\chi({\cal F}')={\rm m.e.}({\cal F})$ (which includes the 
situation in the previous sentence); then we have
${\rm m.e.}({\cal F}/{\cal F}')=0$.
\end{lemma}
\begin{proof}
Let ${\cal F}$ be a sheaf on $G$, and let $U\subset{\cal F}(V)$ be the
minimum subspace of ${\cal F}(V)$  on which the maximum excess occurs.
Let ${\cal F}'$ be the subsheaf of ${\cal F}$ such
that ${\cal F}'(V)=U$ and ${\cal F}'(E)=\Gamma_{\rm ht}(U)$.  
We have that ${\rm m.e.}({\cal F}')={\rm m.e.}({\cal F})$ and the maximum
excess of ${\cal F}'$ occurs at and only at ${\cal F}'(V)$ (by the minimality
of $U$).  This establishes the first sentence in the lemma.
In particular 
$$
{\rm m.e.}({\cal F}) ={\rm m.e.}({\cal F}') = -\chi({\cal F}') .
$$

For the second sentence of the lemma, we claim that
${\cal F}/{\cal F}'$ has maximum excess zero, for if not then we have
compartmentalized
$$
U \subset {\cal F}(V)/{\cal F}'(V), \quad
W \subset {\cal F}(E)/{\cal F}'(E)
$$
with $d_h W,d_t W\subset U$ and $\dim(U)<\dim(W)$.  So let
$U'$ be the inverse image of $U$ in ${\cal F}(V)$ (under the map
${\cal F}(V)\to {\cal F}(V)/{\cal F}'(V)$), and $W'$ that of $W$
in ${\cal F}(E)$.  
We have that $U'$ and $W'$ are compartmentalized.
If $w'\in W'$, we claim that 
$d_{h,{\cal F}}w'$ must lie in
$U'$; indeed, $[w']$, the class of $w'$ in
${\cal F}(V)/{\cal F}'(V)$, is taken to $U$ via $d_{h,{\cal F}/{\cal F}'}$,
and we have a commutative diagram
\begin{diagram}[nohug,height=2em,width=3.5em]
{\cal F}(E) & \rTo & {\cal F}(E)/{\cal F}'(E) \\
\dTo && \dTo \\
{\cal F}(V) & \rTo & {\cal F}(V)/{\cal F}'(V) \\
\end{diagram}
and particular elements
\begin{diagram}[nohug,height=2em,width=5em]
w'\in W' & \rMapsto & [w']\in W \\
\dMapsto && \dMapsto \\
d_{h,{\cal F}}w' &\rMapsto 
& [d_{h,{\cal F}}w']=d_{h,{\cal F}/{\cal F}'}[w']\in U  \\
\end{diagram}
Hence $[d_{h,{\cal F}}w']$, the class of $d_{h,{\cal F}}w'$ in
${\cal F}(V)/{\cal F}'(V)$, lies in $U$ and hence
$d_{h,{\cal F}}w'$ lies in $U'$.
Similarly $d_{t,{\cal F}}w'$ lies in $U'$, and hence
$W'\subset \Gamma_{\rm ht}(U')$.
Since $U',W'$ are compartmentalized, it follows that
\begin{eqnarray*}
{\rm excess}({\cal F},U') & \ge &\dim(W')-\dim(U') \\
&=&\dim(W)+\dim({\cal F}'(E))
-\dim(U)-\dim({\cal F}'(V)).
\end{eqnarray*}
Since $\dim({\cal F}'(E))-\dim({\cal F}'(V))=-\chi({\cal F}')
={\rm m.e.}({\cal F}')={\rm m.e.}({\cal F})$, the above displayed equation
implies that
$$
{\rm excess}({\cal F},U') \ge 
\dim(W)-\dim(U)+{\rm m.e.}({\cal F})
\ge 1+{\rm m.e.}({\cal F}) 
$$
which is a contradiction.
\end{proof}

Returning to the proof of Theorem~\ref{th:shmain}, we claim that it suffices
to establish it for tight sheaves; indeed, consider an arbitrary sheaf,
${\cal F}$, and apply Lemma~\ref{lm:tight1} to obtain a sheaf
tight sheaf, ${\cal F}'$, as described in the lemma.
For any map $\phi\from G'\to G$, we have an exact sequence
$$
0\to \phi^*{\cal F}'\to \phi^*{\cal F}\to 
\phi^*({\cal F}/{\cal F}') \to 0.
$$
We have that ${\cal F}/{\cal F}'$ has maximum excess zero, and hence
so does $\phi^*({\cal F}/{\cal F}')$; by Lemma~\ref{lm:equalifzero},
$$
h_1^\twist\bigl(\phi^*({\cal F}/{\cal F}')\bigr) = 0
$$
provided that $\phi$ is a covering map with the Abelian girth of $G'$ at 
least
\begin{eqnarray*}
&& 2 \Bigl( \dim\bigl(({\cal F}/{\cal F}')(V)\bigr) +
\dim\bigl(({\cal F}/{\cal F}')(E)\bigr) \Bigr) + 1
\\
& \le& 2 \Bigl( \dim\bigl({\cal F}(V)\bigr) +
\dim\bigl({\cal F}(E)\bigr) \Bigr) + 1.
\end{eqnarray*}
In this case 
we get in the long exact sequence beginning
$$
0\to H_1^\twist( \phi^*{\cal F}' ) \to H_1^\twist( \phi^*{\cal F}) 
\to H_1^\twist(\phi^*({\cal F}/{\cal F}'))\to \cdots
$$
amounts to
$$
0\to H_1^\twist( \phi^*{\cal F}' ) \to H_1^\twist( \phi^*{\cal F}) \to 0,
$$
or
$$
H_1^\twist( \phi^*{\cal F}' ) \isom H_1^\twist( \phi^*{\cal F}).
$$
Hence to prove Theorem~\ref{th:shmain} for all ${\cal F}$
of a given maximum excess, it suffices to prove it for those of the
${\cal F}$ that are tight.

We finish the proof by induction on ${\rm m.e.}({\cal F})$
via a second exact sequence.  

\begin{lemma}
\label{lm:tight2}
Let ${\cal F}$ be a tight sheaf on a graph, $G$, of
maximum excess at least one.  Then there exists a subsheaf, ${\cal F}''$,
of ${\cal F}$, such that 
$$
{\rm m.e.}({\cal F}'')=-\chi({\cal F}'')={\rm m.e.}({\cal F})-1,
$$
and such that ${\cal F}/{\cal F}''$ is edge supported and
$\dim(({\cal F}/{\cal F}'')(E))=1$.
\end{lemma}
\begin{proof}
Let ${\cal F}''$ be any subsheaf such that
${\cal F}''(V)={\cal F}(V)$ and ${\cal F}''(E)$ is a codimension one
subspace of ${\cal F}(E)$.  Then ${\cal F}/{\cal F}''$ is edge supported
with the dimension of $({\cal F}/{\cal F}'')(E)$ equal one.
We claim that, furthermore,
the maximum excess of ${\cal F}''$ is ${\rm m.e.}({\cal F})-1$; indeed
this excess is achieved by ${\cal F}''(V)={\cal F}(V)$; furthermore, for
any $U$ properly contained in ${\cal F}''(V)={\cal F}(V)$
we have
$$
{\rm excess}({\cal F}'',U)\le {\rm excess}({\cal F},U)\le {\rm m.e.}({\cal F})
-1.
$$
\end{proof}

We now prove Theorem~\ref{th:shmain} by induction upon ${\rm m.e.}({\cal F})$.
The base case, ${\rm m.e.}({\cal F})=0$, was established in
Lemma~\ref{lm:equalifzero}.
Assume that we have established 
that Theorem~\ref{th:shmain} holds whenever ${\rm m.e.}({\cal F})\le k$
for some integer $k\ge 0$.  
We wish to prove Theorem~\ref{th:shmain} for all ${\cal F}$ of maximum
excess $k+1$, and we know it suffices to do so when ${\cal F}$ is tight.
So let ${\cal F}$ be a tight sheaf of maximum excess of $k+1$, and
let ${\cal F}''$ be any subsheaf as in Lemma~\ref{lm:tight2}.
Then Theorem~\ref{th:shmain} holds for ${\cal F}''$, since ${\cal F}''$
has maximum excess $k$; so for
$\phi\from G'\to G$ of girth greater than
\begin{eqnarray*}
&&2 \Bigl( \dim\bigl({\cal F}''(V)\bigr) +
\dim\bigl({\cal F}''(E)\bigr) \Bigr)
\\
&\le &2 \Bigl( \dim\bigl({\cal F}(V)\bigr) +
\dim\bigl({\cal F}(E)\bigr) \Bigr)
\end{eqnarray*}
we have
\begin{equation}\label{eq:shtwist_F''}
h_1^\twist(\phi^*{\cal F}'')={\rm m.e.}(\phi^*{\cal F}'') 
= \deg(\phi)k .
\end{equation}
Since, by the construction of ${\cal F}''$ in Lemma~\ref{lm:tight2}, we have
$$
\chi({\cal F}'') = \chi({\cal F})+1;
$$
by tightness of ${\cal F}$ we have $\chi({\cal F})=-k-1$ and hence
$$
-\chi({\cal F}'') = k = {\rm m.e.}({\cal F}'');
$$
hence
$$
h_0^\twist(\phi^*{\cal F}'') = \chi(\phi^*{\cal F}'') +
h_1^\twist(\phi^*{\cal F}'') = \deg(\phi)(-k) + {\rm m.e.}(\phi^*{\cal F}'')
$$
$$
= \deg(\phi)(-k) + \deg(\phi)(k) = 0.
$$
We have a short exact sequence
$$
0\to \phi^*{\cal F}'' \to \phi^*{\cal F} \to \phi^*({\cal F}/{\cal F}'')\to 0,
$$
which yields the long exact sequence
$$
0\to H_1^\twist( \phi^*{\cal F}'' ) \to H_1^\twist( \phi^*{\cal F}) 
\to H_1^\twist(\phi^*({\cal F}/{\cal F}''))\to 0,
$$
since $h_0^\twist(\phi^*{\cal F}'')=0$.
Hence
\begin{equation}\label{eq:shtwist_sum}
h_1^\twist( \phi^*{\cal F}) = 
h_1^\twist( \phi^*{\cal F}'' ) +
h_1^\twist(\phi^*({\cal F}/{\cal F}'')).
\end{equation}
But according to Lemma~\ref{lm:tight2}, ${\cal F}/{\cal F}''$
is edge supported, and we therefore know that Theorem~\ref{th:shmain}
holds for ${\cal F}/{\cal F}''$ for any covering map, $\phi$, and hence
$$
h_1^\twist(\phi^*({\cal F}/{\cal F}'')) = \deg(\phi)
{\rm m.e.}({\cal F}/{\cal F}'') = \deg(\phi).
$$
Therefore 
equations~(\ref{eq:shtwist_F''}) and (\ref{eq:shtwist_sum}) shows that 
$$
h_1^\twist( \phi^*{\cal F}) = \deg(\phi)(k+1) = {\rm m.e.}(\phi^*{\cal F})
$$
This establishes Theorem~\ref{th:shmain} for all tight ${\cal F}$ with
${\rm m.e.}({\cal F})=k+1$.

Hence, by induction on the maximum excess of ${\cal F}$, Theorem~\ref{th:shmain}
holds for all sheaves, ${\cal F}$, on $G$.

\end{proof}

\section{Concluding Remarks}
\label{se:shconclude}

In this section we conclude with 
a few remarks about the results in this chapter and
ideas for further research.

We would like to know how much we can prove about the maximum
excess without appealing to homology theory.  Our main application
of homology theory to the maximum excess was
Theorem~\ref{th:shmain}, which implies that the maximum excess
is a first quasi-Betti number.
But part of the proof of Theorem~\ref{th:shmain}, namely
Subsection~\ref{sb:main_theorem}, involved a lot of direct reasoning
about the maximum excess and short exact sequences.
While we believe that the interaction between twisted homology and
maximum excess is interesting, we also think that a treatment of maximum
excess without homology might give some new insights
into the maximum excess.

The maximum excess gives an interpretation of the limit of
$$
h_i^\twist(\phi^*{\cal F})/\deg(\phi)
$$ 
over covering maps
$\phi\from G'\to G$ for a sheaf, ${\cal F}$, of $\field$-vector spaces
on a digraph $G$.  It would be interesting
to have an interpretation of
$$
\lim_\phi \frac{\dim({\rm Ext}^i(\phi^*{\cal F},\phi^*{\cal G}))}{\deg(\phi)}
$$
for any sheaves ${\cal F},{\cal G}$; the maximum excess gives the
interpretation in the special case where ${\cal G}$ is the structure
sheaf, $\underline{\field}$, in which case the Ext groups reduce to
(duals of) homology groups.
We would also be interesting in generalizations of this to a wider
class of settings, such as an arbitrary finite category, or an interesting
subclass such as semitopological categories (defined as categories where
any morphism of an object
to itself must be the identity morphism; see
\cite{friedman_cohomology}).

We would also be interested in knowing if there is a good algorithm for
computing the maximum excess of a sheaf exactly, or even
just giving interesting upper and lower
bounds on it.  This would also be interesting for certain types of 
sheaves.
For example, it would be interesting to know classes of sheaves
for which the first twisted Betti number equals the maximum excess,
in addition to edge simple sheaves of Theorem~\ref{th:shedge_simple}.

Notice that if $G$ is an undirected graph, all the discussion in this 
chapter
goes through.
Either one can orient each edge and use the notation in this chapter,
or just rewrite the notation in this chapter without reference to heads or
tails.  We see that the distinction between heads and tails is never
essential.  For example, rather than having twists at the tails of edges,
we can have them at the heads and tails of edges.  Rather than define a
canonical $d=d_{\cal F}$ to define homology, we simply define homology as
$$
{\rm Ext}^i({\cal F},\underline{\field})^\vee,
$$
which, by the injective resolution of $\underline{\field}$, becomes the
homology groups of
$$
\cdots\to 0\to \oplus_e {\cal F}(e) \to \oplus_v {\cal F}(v)\to 0,
$$
where each ${\cal F}(e)$ is really
$$
({\cal F}(e))^2 / \Delta_e
$$
where $\Delta_e$ is the diagonal in $({\cal F}(e))^2$
(see the discussion regarding equation~(\ref{eq:shbetter_than_standard})
that appears just below
equation~(\ref{eq:shbetter_than_standard2})).
Choosing an
identification of $({\cal F}(e))^2/\Delta_e$ with ${\cal F}(e)$ via
$(a,b)\mapsto a-b$ or $(a,b)\mapsto b-a$ amounts to choosing an
orientation for $e$.  The price of giving a ``canonical'' treatment of
the undirected case, i.e., avoiding edge orientations, is that one
has to work with $({\cal F}(e))^2/\Delta_e$ instead of ${\cal F}(e)$.

\chapter{The Hanna Neumann Conjecture}

\section{Introduction}

Howson, in \cite{howson54}, showed that if $\cH,\cK$ are 
nontrivial, finitely 
generated subgroups of a 
free group, $\cF$, then $\cH\cap \cK$ is finitely
generated, and moreover that
\begin{equation}\label{eq:howson}
{\rm rank}(\cH\cap \cK)-1 \le 2\; {\rm rank}(\cH)\;{\rm rank}(\cK) -
{\rm rank}(\cH) - {\rm rank}(\cK).
\end{equation}
Hanna Neumann, in \cite{hanna,hanna_add} improved this bound to 
what is now called the Hanna Neumann Bound,
\begin{equation}\label{eq:hanna_bound}
{\rm rank}(\cH\cap \cK)-1 \le 2\; \bigl( {\rm rank}(\cH) - 1 \bigr)
\;\bigl( {\rm rank}(\cK) - 1 \bigr);
\end{equation}
furthermore, she conjectured 
that one can remove the factor of $2$ in this bound, i.e., that
\begin{equation}\label{eq:hanna}
{\rm rank}(\cH\cap \cK)-1 \le \bigl( {\rm rank}(\cH) - 1 \bigr)
\;\bigl( {\rm rank}(\cK) - 1 \bigr);
\end{equation}
this conjecture is now known as the Hanna Neumann Conjecture
(or HNC).  One goal of this chapter is to prove the HNC.  Moreover, we
shall prove a strengthened form of the conjecture, first studied by
Walter Neumann in \cite{walter90}, known as the Strengthened Hanna
Neumann Conjecture (or SHNC);
we will state the strengthened conjecture in the next section.

\begin{theorem} 
\label{th:shnc}
The Hanna Neumann Conjecture and the Strengthened
Hanna Neumann Conjecture hold.
\end{theorem}

These conjectures have received 
considerable attention (see
\cite{burns71,imrich76,imrich77,
servatius83,gersten83,stallings83,walter90,tardos92,dicks94,
tardos96,ivanov99,ar00,dicks01,ivanov01,khan02,meakin02,
jitsukawa03,walter07,everitt08,mineyev10}).  However, our proof
uses very different methods from the previous papers.

The main new idea in our approach to the SHNC is to reduce it to
the vanishing maximum excess of a type of sheaf we call a
{\em $\rho$-kernel}.
Although this was described in the introduction to this paper, we
can be a bit more precise here in view of the developments in
Chapter~1.
The SHNC has a well-known reformulation in terms of an inequality
involving the reduced
cyclicity of graphs; 
we shall reformulate this in terms of two graphs,
the inequality now saying that the reduced cyclicity
of one graph is less than that of another.
This would follow if we can (1) realize both graphs as sheaves over some
base graph, $G$, (2) find a surjection of the first onto the second, and
(3) show that the kernel (a sheaf) has vanishing maximum excess, in view
of the fact that the maximum excess is a first quasi-Betti number than
reduces to the reduced cyclicity on a sheaves associated to graphs.
We shall use Galois graph theory to carry this out; then the
base graph, $G$, will be a Cayley graph;
the resulting kernels (for the SHNC) will be called $\rho$-kernels.
It is interesting to note the surjections we use for the SHNC
don't generally exist
as surjections of graphs, rather only as surjections of sheaves; hence
in representing graphs as sheaves, the ``additional morphisms''
we get are crucial to the construction of $\rho$-kernels and
hence to our proof of the SHNC.

It turns out that some $\rho$-kernels have nonvanishing maximum excess
(at least if one defines $\rho$-kernels in a broad sense).
However, any graph, $L$, of interest to us in the SHNC, will have a
family of associated $\rho$-kernels, and we will prove that the 
generic maximum excess in this family is zero, for each $L$.
To do this we shall use Galois theory and symmetry to argue that
if this generic maximum excess does not vanish then it is large,
i.e., a multiple
of the order of the associated Galois group.
Then we will give an inductive argument, showing that if the
generic maximum excess of $\rho$-kernels for $L$ is positive, then the
same is true when we remove some edge from $L$, provided that $L$ has
positive reduced cyclicity.
The base case of the induction, when $L$ has vanishing reduced cyclicity,
is trivial to establish.
Hence each $L$ of interest in the SHNC has a $\rho$-kernel of
vanishing maximum excess, and this establishes in SHNC.

We emphasize that our proof of the SHNC also uses
Theorem~\ref{th:shmain},
that implies that the ``maximum
excess'' is first quasi-Betti number.
However, 
it is quite possible that one can prove the inequalities we need
for the SHNC for without Theorem~\ref{th:shmain}.
Furthermore, in the conclusion to this chapter we will give a slight
variant of our proof of the SHNC that will avoid any
reference to Theorem~\ref{th:shmain} or any homology theories
(although this would require
the lengthy combinatorial argument of Appendix~\ref{ap:elem}).
However, regardless of how we present the proof, we shall explain in
the conclusion that the
homology theory can provide valuable insight into the maximum excess.

This paper shows that the SHNC is not merely
an attempt to improve an inequality by a factor of two; 
our study of the SHNC has lead to new ideas in sheaves on graphs that
can be applied to graph theory.
This came as a surprise to us at first, 
although it is perhaps less surprising in retrospect, for a number
of reasons.
First, the HNC and SHNC seem to describe a fairly fundamental question in
group theory (of how rank behaves under intersection).
Second, the SHNC can be viewed as a graph theory question involving
the reduced cyclicity,
which is an interesting graph invariant (e.g. it scales under
covering maps).
Third, the SHNC, viewed in terms of the Galois graph theory,
has a simple homological explanation, namely the vanishing of a limit
homology group.  The vanishing of (co)homology groups has a vast literature
and importance; the SHNC is an interesting and seemingly
difficult result in the
family of homology group vanishing theorems.
Fourth, Lior Silberman has pointed out to us that the reduced cyclicity
is the discrete analogue of
$L^2$ Betti numbers; the $L^2$ Betti number was defined first by
Atiyah (\cite{atiyah}), and has been the subject of much surrounding the
``Atiyah conjecture'' (see \cite{luck02}).  Mineyev's article, 
\cite{mineyev10},
also makes a connection between the SHNC and $L^2$ Betti numbers.

At this point we can give more motivation for the use 
sheaf theory in this chapter, i.e., why we do not just use graphs and
their homology.
Our reformulation of the SHNC begins by 
searching for a morphism involving 
the graphs of interest to the SHNC.  
In order for this morphism to exist, to be surjective, and to have
a kernel, we must work with more general objects than graphs.
In many topological situations, the topological spaces are sufficiently
``robust'' that one does not have to generalize the objects.
However, in non-Hausdorff spaces, such as graphs
or those in algebraic geometry, many geometric notions, such as
``connect two points with a path,'' ``form a cone,'' etc.,
don't make sense or are very awkward to implement.
So for graphs we use sheaf theory, which is 
a simple (co)homology theory that
is adapted to our spaces, but general and expressive enough for appropriate
surjections and kernels to exist.
Of course, it is possible that there are other reasonable frameworks that
one could use instead of sheaves.

The rest of this chapter is organized as follows.
In Section~\ref{se:meshnc} we describe the SHNC and previous work on the
HNC and SHNC, including some resolved special cases of the SHNC.
In Section~\ref{se:megraph} we give a common graph theoretic reformulation
of the SHNC.
In Section~\ref{se:megalois} we describe applications of 
``graph Galois theory'' to simplifying the SHNC in a way that leads
to the construction of $\rho$-kernels; this builds on some of the Galois
graph theory described ealier, in Section~\ref{se:shgalois}.
In Section~\ref{se:menewrhokernel} we construct $\rho$-kernels and prove that
if their maximum excesses vanish then the SHNC holds; we also describe
what we call ``$k$-th power kernels,'' which generalize $\rho$-kernels
and which will be necessary to prove our main theorems about
the generic maximum excess of $\rho$-kernels.
In Section~\ref{se:mesymmetry} we use symmetry to prove that the generic
maximum excess of a certain
type of $k$-th power kernel is divisible by the order of
a group associated to the class of kernels.
In Section~\ref{se:mevar} we prove some comparison theorems about how
the maximum excesses of different classes of $k$-th power kernels compare;
the main theorem that we prove shows that if $\rho$-kernels associated
to a graph, $L$, have positive generic maximum excess, then the same
is true of a graph $L'$ that consist of $L$ with one edge discarded,
provided that $\rho(L')=\rho(L)-1$.
In Section~\ref{se:meproof} we briefly combine a number of theorems of previous
sections to argue that if the SHNC
does not hold, then for some $L$ we have the class of $\rho$-kernels
associated to $L$ have positive generic maximum excess, which by the
results of Section~\ref{se:mevar} means that the same is true for some
$L$ with $\rho(L)=0$, which we easily show is impossible.  This establishes
the SHNC.
In Section~\ref{se:meconclude} we make some concluding remarks, including
a variant of our proof that avoids Theorem~\ref{th:shmain} and any reference
to homology theory.  Such a proof will require
Appendix~\ref{ap:elem}, where we show that 
vanishing maximum excess of enough $\rho$-kernels implies the SHNC, but
we do so just using elementary graph theory; this shows that
a lot of the sheaf theory can be translated into direct graph theoretic terms;
it also shows that a simple sheaf theoretic
calculation may translate into a much longer graph theoretic
calculation.

\section{The Strengthened Hanna Neumann Conjecture}
\label{se:meshnc}

In this section we state the SHNC and comment on previous work on the
HNC and SHNC, including
some established special cases of these conjectures.

Walter Neumann, in
\cite{walter90}, 
showed that the Hanna Neumann Bound, i.e., equation~(\ref{eq:hanna_bound}),
could be strengthened to
$$
\sigma(\cH,\cK) \le 2\ {\rm rk}_{-1}(\cH)\ {\rm rk}_{-1}(\cK),
$$
where ${\rm rk}_{n}(\cG)$ denotes $\max({\rm rank}(\cG)+n,0)$, and
where
$$
\sigma(\cH,\cK) =
\sum_{\cH x\cF\in \cH\backslash  \cF/\cK}\ {\rm rk}_{-1}(\cH\cap x^{-1}\cK x),
$$
the summation being over the
double coset, $\cH\backslash \cF/\cK$, representatives, 
$x$; taking $x$ be the identity in the summation shows that 
$$
{\rm rk}_{-1}(\cH \cap \cK) \le \sigma(\cH,\cK),
$$
so that Walter Neumann's above bound strengthens the Hanna Neumann Bound.
Walter Neumann further formulated the conjecture that
\begin{equation}\label{eq:strengthened}
\sigma(\cH,\cK) \le {\rm rk}_{-1}(\cH)\ {\rm rk}_{-1}(\cK),
\end{equation}
now known as the Strengthened Hanna Neumann Conjecture (or SHNC).
For the rest of this section we review previous work on the HNC and SHNC.

One collection of results on the problem involves general bounds on
$\sigma(\cH,\cK)$ or ${\rm rk}_1(\cH\cap\cK)$.  It turns out that
all general bounds we know for the HNC, i.e., on ${\rm rk}_1(\cH\cap\cK)$,
also are known to hold for $\sigma(\cH,\cK)$.  Also, all
bounds we know are of the form
$$
\sigma(\cH,\cK)\le
2\ {\rm rank}(\cH)\ {\rm rank}(\cK) + c_1\; {\rm rank}(\cH) +
c_2\; {\rm rank}(\cK) +c_3
$$
for ranks $\cH,\cK$ sufficiently large, where $c_1,c_2,c_3$ are constants
depending on the bound; thus all improvements of Howson's original bound
are in the lower order terms, i.e., in the $c_i$'s.
The improved bounds on $\sigma(\cH,\cK)$ after
\cite{howson54,hanna,hanna_add} include the bound
$$
2\ {\rm rk}_{-1}(\cH)\ {\rm rk}_{-1}(\cK) - \min({\rm rk}_{-1}(\cH),
{\rm rk}_{-1}(\cK))
$$
of Burns in \cite{burns71}\footnote{
	Bounds appearing before
	\cite{walter90} are stated as bounds on 
	${\rm rk}_{-1}(\cH\cap\cK)$, but
	actually give bounds on $\sigma(\cH,\cK)$ as well.
}, the bound
$$
{\rm rk}_{-1}(\cH)\ {\rm rk}_{-1}(\cK) + \max( 
{\rm rk}_{-2}(\cH)\ {\rm rk}_{-2}(\cK) -1 ,0)
$$
of Tardos \cite{tardos92,tardos96}, and, what is the best bound prior
to ours,
\begin{equation}\label{eq:df}
{\rm rk}_{-1}(\cH)\ {\rm rk}_{-1}(\cK) +
{\rm rk}_{-3}(\cH)\ {\rm rk}_{-3}(\cK)
\end{equation}
of Dicks and Formanek in \cite{dicks01}.  

Another collection of results concerns special cases of the HNC and SHNC
that are resolved.  To be precise,
say that the ``HNC holds
for $(\cH,\cK)$'' if equation~(\ref{eq:hanna}) holds, and say that
$\cH$ is {\em universal for the HNC} if for any $\cK$,
the HNC holds for $(\cH,\cK)$.  
Similarly for the SHNC and equation~(\ref{eq:strengthened}).
Similar to before, all results we know that resolve special cases of the
HNC also resolve those cases of the SHNC.
Note that any finitely generated free group,
$\cF$,
is a subgroup of $\cF_2$, the free group on two generators, so we are
free to assume that $\cF=\cF_2$ in the HNC and SHNC.
Here are some results on special cases
of the SHNC that are easy to describe in group
theoretic terms:
\begin{enumerate}
\item $\cH$ is universal for the SHNC if it is of rank at most 
three (\cite{dicks01}), in view of
equation~(\ref{eq:df}), with rank two settled earlier by
Tardos (\cite{tardos92});
\item $\cH$ is universal for the SHNC if it is 
positively generated
(see \cite{khan02,meakin02,walter07});
\item $\cH$ is universal for the SHNC for ``most'' $\cH$
(see \cite{walter90,jitsukawa03});
\item the SHNC holds
either for $(\cH,\cK)$ or for $(\cH,\cK')$ 
for any $\cH,\cK$ that are
subgroups of $\cF_2$,
where $\cK'$
is obtained from $\cK$ by the map taking each generator
of $\cF_2$ to its inverse
(see \cite{jitsukawa03}).
\end{enumerate}
The result of item~(3) 
on ``most'' groups, of Walter Neumann (\cite{walter90}), and some
additional results on the SHNC, such as
Corollary~3.2 of \cite{meakin02}, are easier to describe using a
graph theoretic formulation of the SHNC that we give in the next section.
It is also known that the SHNC is related to the coherence problem in
one-relator groups (\cite{wise05}).

\section{Graph Theoretic Formulation of the SHNC}
\label{se:megraph}

The goal of this section is to describe an equivalent 
formulation
of the HNC and SHNC in graph theoretic terms involving fibre products;
this formulation is implicit in
\cite{howson54}, but more explicit in
\cite{imrich76,imrich77,gersten83,stallings83,walter90} and 
other references
in \cite{dicks94}.
There is another equivalent reformulation of the SHNC by
Dicks in \cite{dicks94}, known as the ``amalgamated graph conjecture,'' which
we do not discuss here.

By a {\em bicoloured digraph}, or simply a {\em bigraph}, we mean a directed
graph,
$G$, such that each edge is coloured (or labelled) either ``1'' or ``2.''
It is also equivalent to giving a directed graph homomorphism
$\nu\from G\to B_2$, where $B_2$ is the graph with one vertex and two
self-loops, one coloured ``1'' and the other ``2.''
If, moreover, $\nu$ is \'etale, we call $\nu$ or (somewhat abusively)
$G$ an \'etale bigraph, which means that $G$ is a bigraph such that
no vertex has
two incident edges, both incoming or both outgoing, of the same colour.

Given a digraph, $G$, 
recall the definition of the reduced cyclicity from
equation~(\ref{eq:shrho}), where ${\rm conn}(G)$ 
denotes the connected components of $G$; 
set
$$
\rho'(G) = \max_{X \in {\rm conn}(G)} \Bigl( \max\bigl(0,h_1(X)-1\bigr)\Bigr).
$$

The HNC is 
equivalent to
\begin{equation}\label{eq:graphHNC}
\rho'(K\times_{B_2} L)\le \rho(K)\rho(L)
\end{equation}
for all \'etale bigraphs $K$ and $L$; the SHNC
is equivalent to
\begin{equation}\label{eq:graphSHNC}
\rho(K\times_{B_2} L)\le \rho(K)\rho(L)
\end{equation}
for all \'etale bigraphs $K$ and $L$
(see \cite{howson54,imrich76,imrich77,gersten83,stallings83,walter90,dicks94}).
We shall work with this form of the SHNC.
Again, we say the HNC or SHNC, respectively,
holds for a pair of \'etale bigraphs,
$(K,L)$, if equation~(\ref{eq:graphHNC}) or (\ref{eq:graphSHNC}),
respectively,
holds; and again,
we say that $K$ is {\em universal} for
the HNC or SHNC, respectively, if for any $L$ the same conjecture
holds for $(K,L)$.

Let us briefly explain the connection between the group theoretic
formulations of the HNC and SHNC and the graph theoretic formulations.
Given generators, $g_1,g_2$, for the free group, $\cF_2$, for each
subgroup $\cH\subset\cF_2$, there is a canonically
associated \'etale bigraph,
$K$; $K$ is given by constructing the Schreier coset graph, 
${\rm Sch}(\cF_2,\cH,
\{g_1,g_2\})$, and letting $K$ be the ``core'' of 
${\rm Sch}(\cF_2,\cH,\{g_1,g_2\})$, i.e., its smallest
subgraph containing all reduced loops based at the
vertex $\cH$ (see \cite{meakin02} or the references in the 
previous paragraph);
${\rm Sch}(\cF_2,\cH,\{g_1,g_2\})$ with directed edges 
labelled either $g_1$ or $g_2$
is a (typically infinite degree)
covering of $B_2$, and $K$, a finite subgraph of
${\rm Sch}(\cF_2,\cH,\{g_1,g_2\})$, is therefore an \'etale bigraph.
If $\cH,\cK$ are subgroups of $\cF_2$, and 
$K,L$ the corresponding \'etale bigraphs, then
each component of $K\times_{B_2}L$ corresponds to the graph associated
to 
$\cH\cap x^{-1} \cK x$ ranging over double coset representatives, $x$.

Theorem~\ref{th:shnc} will be proven by the following equivalent
theorem.

\begin{theorem} 
\label{th:main}
The Strengthened Hanna Neumann Conjecture holds.  That is,
if $K\to {B_2}$ and $\kone\to {B_2}$ are two 
\'etale bigraphs over ${B_2}$, then
\begin{equation}\label{eq:shnc}
\rho(K\times_{B_2} \kone)\le \rho(K)\rho(\kone).
\end{equation}
\end{theorem}
Equation~(\ref{eq:shnc}) is tight in that
if either $K$ or $L$ is a 
covering of $B_2$ (i.e., has all vertices of degree four), then the
inequality is satisfied with equality.

\section{Galois and Covering Theory in the SHNC}
\label{se:megalois}

In this section develop more aspects of Galois theory, in addition to
those given in Section~\ref{se:shgalois}.
This will later lead us to
sheaves we call $\rho$-kernels.
Let us now give definitions and state the main theorems to be
developed in this section.

\begin{definition} By the 
{\em Cayley bigraph on a group, $\cg$, with generators
$g_1$ and $g_2$}, denoted
$G={\rm Cayley}({\cal G};g_1,g_2)$, we mean the \'etale bigraph, $G$, where
$V_G=\cg$ and $E_G=\cg\times\{1,2\}$ (as sets), 
such that for each $g\in\cg$ and $i=1,2$, the edge $(g,i)$ has colour $i$,
tail $g$, and
head $g_ig$.
\end{definition}

We reduce the SHNC to the special case of subgraphs of a Cayley graph, as
follows.

\begin{theorem}\label{th:sub_bi_enough}
To prove Theorem~\ref{th:main}, the SHNC, it suffices verify the SHNC
on all
pairs, $(L,L')$, such that $L,L'$ are subgraphs of the same Cayley bigraph.
In particular, to prove the SHNC it suffices 
to show that
any subgraph of a Cayley bigraph is universal for
the SHNC.
\end{theorem}

The following simplifications of the SHNC on subgraphs
of Cayley graphs will help solidify the connection between the
SHNC and $\rho$-kernels of the next section.

\begin{theorem}\label{th:sub_bi_shnc}
Let $L$ be a subgraph of a Cayley bigraph, $G$, on a group, $\cg$.
Then
\begin{enumerate}
\item $L$ is universal for the SHNC if for any \'etale $L'\to G$ we
have $(L,L')$ satisfies the SHNC (with $L'$ inheriting the edge colouring
from $G$, i.e., from
the composition $L'\to G$ followed by $G\to B_2$);
\item for any \'etale $L'\to G$ we have
$$
L\times_{B_2} L' \isom (L\cg)\times_G L',
$$
where
$$
L\cg = \coprod_{g\in\cg} Lg.
$$
\end{enumerate}
\end{theorem}

Before giving Galois theory we quickly describe the remarkable
reason for the strong
connection between the SHNC and covering and Galois theory.
Since its proof is so short, we give it here as well.

\begin{theorem} 
\label{th:remarkable_rho}
For any covering map $\pi\from K\to G$ of degree $d$,
we have $\chi(K)=d\chi(G)$ and $\rho(K)=d\rho(G)$.
\end{theorem}
\begin{proof} The claim on $\chi$ follows since $d=|V_K|/|V_G|=|E_K|/|E_G|$.
To show the claim on $\rho$, it suffices to consider the case of $G$
connected, the general case obtained by summing over connected components; 
but similarly it suffices to consider the case of $K$ connected.
In this case
$$
\rho(G)=h_1(G)-1 = -\chi(G) = -d\chi(K) = d\bigl(h_1(K)-1\bigr)=d\rho(K).
$$
\end{proof}

From this theorem it follows that if $\mt K\to K$ and
$\mt{L}\to L$ are covering maps of \'etale bigraphs, then
$$
\rho(\mt K\times_{B_2}\mt L)-\rho(\mt K)\rho(\mt L) =
[\mt K \colon K]\ [\mt L\colon L]
\bigl( \rho(K\times_{B_2}L)-\rho(K)\rho(L) \bigr);
$$
hence $(K,L)$ satisfy the SHNC iff $(\mt K,\mt L)$ do.
This means that to study the SHNC, one can always pass to covers of the
bigraphs of interest.

For the rest of this section we describe a number of aspects of what we
call Galois graph
theory and use it for prove Theorems~\ref{th:sub_bi_enough}
and \ref{th:sub_bi_shnc}.

\subsection{Galois Theory of Graphs}

Here we further develop the Galois theory of graphs discussed
in Section~\ref{se:shgalois}.
We remind the reader that in this 
article Galois group actions, when written
multiplicatively (i.e., not viewed as functions or morphisms)
will be written
on the right, since our Cayley
graphs are written with its generators acting on the left.

The following fact is an analogue of a 
standard and surprisingly useful fact
in descent theory (as in \cite{sga4.5}); it is also surprisingly useful
for the SHNC, despite the fact that it 
is trivial.

\begin{theorem}\label{th:decompose}
Let $\pi\from K\to G$ be Galois.  Then
$$
K\times_G K = \coprod_{\sigma\in{\rm Aut}(K/G)} K_\sigma
$$ 
where $K_\sigma$ is 
the subgraph of $K\times_G K$
given via
$$
V_{K_\sigma} = \{(v,vg) \;|\; v\in V_K,\;g\in{\rm Aut}(K/G) \}, 
$$
$$
E_{K_\sigma} = \{(e,eg) \;|\; e\in E_K,\;g\in{\rm Aut}(K/G) \}.
$$
Each $K_\sigma$ is isomorphic to $K$.
\end{theorem}
(See \cite{friedman_geometric_aspects}, and 
compare with the identical formula for fields in \cite{sga4.5}, Section~I.5.1).

\begin{corollary}\label{cr:twisted_product}
In Theorem~\ref{th:decompose}, let us further assume
that we have morphisms $K_1\to K$ and $K_2\to K$.  Then
$$
K_1\times_G K_2 \isom \coprod_{\sigma\in{\rm Aut}(K/G)} K_1\times_K(K_2\sigma).
$$
\end{corollary}
\begin{proof} 
$$
K_1\times_G K_2 \isom (K_1\times_K K)\times_G (K\times_K K_2)
\isom K_1\times_K (K\times_G K)\times_K K_2
$$
$$
\isom \coprod_{\sigma} \bigl( K_1\times_K K_\sigma \times_K K_2 \bigr)
\isom \coprod_{\sigma} \bigl( K_1\times_K (K_2\sigma) \bigr).
$$
\end{proof}

There are many extensions to this basic theory.
We mention one interesting example.

Assume, for simplicity, that $G$ is connected.
If $K\to G$ is Galois
and factors as $K\to K'\to G$, then $K\to K'$ is Galois, with Galois
group being the subgroup of ${\rm Aut}(K/G)$ fixing
any vertex or edge fiber of $K\to K'$; hence ${\rm Aut}(K/K')$ is a 
subgroup of ${\rm Aut}(K/G)$.  Conversely, a subgroup of
${\rm Aut}(K/G)$ divides the vertices and edges of $K$ into orbits,
giving a graph $K'$ (whose vertices and edges are these orbits) and
a factorization $K\to K'\to G$.
Furthermore, for an intermediate cover $K\to K'\to G$, $K\to K'$ is
always Galois (since ${\rm Aut}(K/K')$ has the right cardinality), 
and $K'\to G$ is Galois iff the subgroup of ${\rm Aut}(K/G)$ fixing
$K\to K'$ fibers
is a normal subgroup of ${\rm Aut}(K/G)$.
See \cite{st1} for details.

If $K\to G$ is Galois and factors as $K\to K'\to G$, 
\begin{equation}\label{eq:fancy}
K\times_G K' = \coprod_{g\in{\rm Aut}(K/G)/{\rm Aut}(K/K')} K_g,
\end{equation}
where
$$
V_{K_g} = \{(v,[v]g) \ |\  v\in V_K\}, \quad
E_{K_g} = \{(e,[e]g) \ |\  e\in E_K\},
$$
where $[v],[e]$ respectively
denote the images of $v,e$, respectively, in $K'$; each $K_g$ is
isomorphic to $K$.  Special cases of this statement include
the trivial case $K'=G$
and the case $K'=K$ stated earlier.

\subsection{Base Change}
\label{sb:base_change}

There are a number of easy ``stability under
base change'' results; these say that
in a digram arising from arbitrary digraph morphisms $L\to B$ and
$M\to B$,
\begin{diagram}[nohug,height=2.0em,width=2.5em,tight]
L\times_{B}M & \rTo & M \\
\dTo & & \dTo \\
L & \rTo & B  \\
\end{diagram}
if $L\to B$ has a certain property, then so does $L\times_B M\to M$.
Just from the construction of the fibre product, we easily see that
the following classes of morphisms
are stable under base change: 
\'etale morphisms, covering morphism, and Galois morphisms (and many others
that we won't need, such as open inclusions, morphisms that are $d$--to--$1$
for some fixed $d$, etc.).

\subsection{Etale Factorization}

In this subsection
we shall prove that any \'etale map factorizes as an open inclusion
followed by a covering map.  This will easily establish
Theorem~\ref{th:sub_bi_enough}.

We define an {\em open inclusion} to be any inclusion $H\to G$
of a subgraph, $H$, in a graph, $G$.
We say the inclusion is {\em dense} 
if $V_H=V_G$; this agrees with the topological notion, 
i.e., the closure of $G$ in $H$
is $H$, under the topological view of $G$ in \cite{friedman_sheaves}.

\begin{lemma} Let $\pi\from G\to B$ be an \'etale map.  Then
$\pi$ factors as an open inclusion, $\iota\from G\to G'$, 
followed by a covering map, $\pi'\from G'\to B$.
If the vertex fibres of $\pi$ (i.e., $\pi^{-1}(v)$ over all $v\in V_B$)
are all of the same size, i.e., 
$\pi_V\from V_G\to V_B$ is $d$-to-$1$ for
some $d$, then we may assume $\iota$ is dense;
if in addition $G$ is connected, then we may assume $G'$ is connected.
\end{lemma}
A variant of the first sentence of this
theorem is called Marshall Hall's theorem
in \cite{stallings83}.
\begin{proof} By adding isolated vertices to $G$ we may assume $\pi_V$ is
$d$-to-$1$ for some $d$.
Extend $G$ to $G'$ and $\pi$ to $\pi'\from G'
\to B$ by completing each $\pi^{-1}(e)$ to a perfect bipartite
matching of the vertices
over the tail of $e$ and those over the head of $e$ (for a self-loop we
view these two sets of vertices as disjoint).
Clearly $\pi'\from G'\to B$ is a covering map.
If $\pi_V$ was originally $d$-to-$1$ for some $d$, then
$G'$ is obtained by adding only edges, so $G$ is dense in $G'$;
if furthermore $G$ is connected, then the $G'$ obtained by adding only
edges is, of course, still connected.
\end{proof}

Here is an easy, but vital, observation.

\begin{lemma} If $S\to B_2$ is a Galois map with Galois group, $\cg$, 
then $S$ is isomorphic
to a Cayley
bigraph on the group $\cg={\rm Gal}(S/B_2)$.
\end{lemma}

\begin{proof} 
Choose any $v_0\in V_S$ to be the ``origin.''  The association
$g\mapsto v_0g$ sets up an identification of $\cg$ with $V$, by definition
of a Galois covering map, since there is a unique vertex of $B_2$ and
hence a singe vertex fibre in $S$.
Since $S\to B_2$ is a covering map, the vertex $v_0$ is the tail of a
unique colour $1$ edge, $e$, whose head is $v_0g_1$ for a unique $g_1\in\cg$.
For any $g$ we have $eg$ has tail $v_0g$ and head $v_0g_1g$.
It follows that identifying $V$ with $\cg$ means that
there is an edge $(g,g_1g)$ (i.e., whose tail is $g$ and head is $g_1g$)
of colour $1$ 
for each $g\in\cg$.  Similarly for edges of colour $2$, and this sets
up an isomorphism between $S$ and ${\rm Cayley}(\cg;g_1,g_2)$.
\end{proof}

\begin{proof}[Proof (of Theorem~\ref{th:sub_bi_enough})]
Let $G\to B_2$ and $K\to B_2$ be
\'etale maps.  Let these \'etale maps factor as open inclusions
followed by covering maps as $G\to\widetilde{G}\to B$ and
$K\to\widetilde{K}\to B$. 
Let $S$ be a Galois cover of $\widetilde{G}\times_B\widetilde{K}$.
Consider
$G'=G\times_{\widetilde G}S$, which admits a natural map to $G$
(namely projection onto the first component),
and similarly $K'=K\times_{\widetilde K}S$.
We claim that $G'\to G$ is a covering map; indeed, by stability
under base change (see Subsection~\ref{sb:base_change}), since $\widetilde K
\to B_2$ is a covering map, so is $\widetilde{G}\times_B\widetilde{K}\to
\widetilde{G}$; since $S\to \widetilde{G}\times_B\widetilde{K}$ is 
a covering map, so is $S\to\widetilde{G}$; hence, by base change
so is $G'\to G$.  Similarly $K'\to K$ is a covering map.
According to Theorem~\ref{th:remarkable_rho} and the discussion below,
the SHNC is satisfied at
$(G,K)$ iff it is satisfied at $(G',K')$.
But $G',K'$ are subgraphs of $S$, and $S$ is a Galois cover of $B_2$,
and therefore a Cayley bigraph.
\end{proof}

Although we shall not need it, we mention that the idea in this last proof
can be extended from $G\to B_2$ and $K\to B_2$ to an arbitrary number
of \'etale maps, $L_i\to B_2$, and gives the following interesting fact.

\begin{theorem}\label{th:common_cover}
For any \'etale bigraphs $L_1,\ldots,L_k$, there are covering
maps $L'_i\to L_i$ and a Cayley bigraph, $S$,
such that each $L'_i$ is a subgraph of $S$ that is dense (i.e.,
$L'_i$ has the same vertex set as $S$).
\end{theorem}

\subsection{The Proof of Theorem~\ref{th:sub_bi_shnc}}

We finish this section with the proof of Theorem~\ref{th:sub_bi_shnc}.

\begin{proof}[Proof of Theorem~\ref{th:sub_bi_shnc}]
Claim~(1) of the theorem is a simple base change argument:
if $L$ is a subgraph of a Cayley bigraph,
$G$, and $L'\to B_2$ is any \'etale bigraph, let $L''=L'\times_{B_2}G$.
Then, by base change (see Subsection~\ref{sb:base_change}),
$L''\to G$ is \'etale and $L''\to L'$ is a covering map.  
Then $(L,L')$ satisfies the SHNC iff $(L,L'')$ does.
Hence
$L$ is universal for the SHNC iff $(L,L'')$ satisfies the SHNC for all
\'etale bigraphs, $L''$, whose colouring map
factor as $L''\to G\to B_2$.

Claim~(2) is an immediate consequence of
Corollary~\ref{cr:twisted_product}, with
$G,K,K_1,K_2$ respectively replaced by 
$B_2,G,L,L'$, noting that ${\rm Aut}(G/B_2)=\cg$.

\end{proof}

\section{$\rho$-kernels}
\label{se:menewrhokernel}

In this section we introduce a collection of sheaves that 
are central to our proof
of the SHNC.  They are called $\rho$-kernels.  
Before defining them, we motivate their construction by showing
how their study is connected to the SHNC.
First we need to set some notation on Cayley graphs.

\subsection{Sheaves on Cayley graphs}

Let $G={\rm Cayley}(\cg;g_1,g_2)$ be a 
Cayley bigraph on a group, $\cg$.  Recall that
since our generators act on the left, e.g., the colour $1$ edges
are of the form $(g,g_1g)$, the Galois group of $G$ is $\cg$ acting
on the right.
Now we define a right action of $\cg$ on sheaves on $G$.
We shall state this in slightly more general terms.
This is 
completely
straightforward and mildly tedious, but convenient in this section
and vital to Section~\ref{se:mesymmetry}.

\begin{definition} We say that a group, $\cg$, {\em acts} 
on a digraph, $G$,
{\em on the right}, if associated to each $g\in\cg$ is an isomorphism
$\pi_g$, of $G$ such that $\pi_{g_1g_2}=\pi_{g_2}\pi_{g_1}$ for all
$g_1,g_2\in G$.  
We will identify $g$ with $\pi_g$ if no confusion can arise.
If $L$ is a subgraph of $G$, we write $Lg$ for the image
of $L$ under $g$ (i.e., under $\pi_g$); similarly
if $P\in V_G\amalg E_G$, $Pg$ denotes the image of $P$ under $g$.
\end{definition}
Of course, if $G$ is a Cayley bigraph on a group, $\cg$, then $\cg$
acts on $G$ on the right.

\begin{theorem}
\label{th:acts}
Let a group, $\cg$, act on a digraph, $G$,
on the right.  Then each element of
$\cg$ acts naturally as a functor on sheaves, via
the association
$g\mapsto \pi_{g^{-1}}^*$, such that
\begin{enumerate}
\item $\cg$ acts on the right, i.e., if we write ${\cal F}g$ for
$\pi_{g^{-1}}^*{\cal F}$ for any sheaf, ${\cal F}$, on $G$, then
for any
$g_1,g_2\in\cg$ we have ${\cal F}g_1g_2=({\cal F}g_1)g_2$, and
similarly with the sheaf ${\cal F}$ replaced by a morphism of sheaves;
\item for each sheaf, ${\cal F}$, on $G$, any $g\in\cg$, and any
$P\in V_G\amalg E_G$ we have
$$
({\cal F}g)(P)={\cal F}(Pg^{-1}); 
$$
and
\item for each subgraph $L\subset G$ and field, $\field$, we have
$$
\underline\field_L\,g = \underline\field_{Lg}.
$$
\end{enumerate}
\end{theorem}
In Section~\ref{se:mesymmetry} it will be important to use the fact that
for each $g\in\cg$, $\pi_{g^{-1}}^*$ is a {\em functor}, i.e., it
acts (compatibly) on morphisms of sheaves as well as on sheaves.
\begin{proof}
For item~(1), we recall that for any $u\from G'\to G$, $u^*$ is a functor
on sheaves, and for composable morphisms of digraphs, $u_1,u_2$, we have
$(u_1u_2)^*=u_2^*u_1^*$;
hence, since $\cg$ acts on the right, for any $g_1,g_2\in\cg$ we have
$$
\pi_{g_1^{-1}}^* \pi_{g_2^{-1}}^* =
(\pi_{g_2^{-1}}\pi_{g_1^{-1}} )^* 
=
(\pi_{g_1^{-1}g_2^{-1}} )^* =
\pi_{(g_2g_1)^{-1}}^*
$$
and so $g\mapsto \pi_{g^{-1}}^*$ is defines an action on sheaves and morphisms
of sheaves that acts on the right.

Item~(2) follows immediately from the definition of the pullback.
Item~(3) follows since for all $P\in V_G\amalg E_G$ and $L\subset G$
and $g\in \cg$ we have
$$
(\underline\field_L\,g)(P)=\underline\field_L(Pg^{-1});
$$
but $Pg^{-1}\in L$ iff $P\in Lg$, so
$$
\underline\field_{Lg}(P)=\underline\field_L(Pg^{-1})
=(\underline\field_L\,g)(P).
$$
Hence $\underline\field_{Lg}=\underline\field_L\,g$.
\end{proof}

Given a sheaf, ${\cal F}$, on $G$ we define
$$
{\cal F}\cg = \bigoplus_{g\in \cg} {\cal F}g.
$$
In particular, for $L\subset G$,
if we set
$$
L\cg = \coprod_{g\in\cg} Lg
$$
(akin to the notation in Theorem~\ref{th:sub_bi_shnc})
then
$$
\underline\field_L \,\cg \,\isom\, \underline\field_{L\cg}.
$$

\subsection{Kernels and the SHNC}

The following theorem summarizes our approach to the SHNC.

\begin{theorem}\label{th:motivate}
Let $L$ be a subgraph of a Cayley bigraph, $G$.  Assume there is an
exact sequence
\begin{equation}\label{eq:motivate}
0\to \ck \to \underline\field_L\cg\to \underline\field^{\rho(L)} \to 0
\end{equation}
such that ${\rm m.e.}(\ck)=0$.  Then
$L$ is universal for the SHNC.
\end{theorem}
\begin{proof} 
According to Theorem~\ref{th:sub_bi_shnc}, it suffices to show that
for each \'etale $u\from L''\to G$ we have that $(L,L'')$ satisfies the SHNC.
Tensoring equation~(\ref{eq:motivate}) with
$\field_{L''}=u_!\underline\field$ gives
\begin{equation}\label{eq:tensored}
0\to \ck\otimes \underline\field_{L''} \to 
\underline\field_L \cg \otimes\underline\field_{L''}
\to \underline\field^{\rho(L)}_{L''}\to 0.
\end{equation}
Since ${\rm m.e.}(\ck)=0$ and $u\from L''\to G$ is \'etale, 
Theorem~\ref{th:etale_contagious} and the discussion before it implies that
$$
{\rm m.e.}(\ck\otimes \underline\field_{L''} ) = 0 .
$$
Since the maximum excess is a first quasi-Betti number,
this and equation~(\ref{eq:tensored}) implies that
$$
{\rm m.e.}(\underline\field_L \cg \otimes\underline\field_{L''} )
\le
{\rm m.e.}(\underline\field^{\rho(L)}_{L''}) = \rho(L)\rho(L'').
$$
But
$$
\underline\field_L \cg \otimes\underline\field_{L''} 
\isom
\underline\field_{(L\cg)\times_G L''}
\isom
\underline\field_{L\times_{B_2}L''}
$$
(using equation~(\ref{eq:tensor_product})
and Corollary~\ref{cr:twisted_product}), and so we have
$$
\rho(L\times_{B_2}L'')={\rm m.e.}(\underline\field_{L\times_{B_2}L''})
=
{\rm m.e.}(\underline\field_L \cg \otimes\underline\field_{L''})
\le
\rho(L)\rho(L'').
$$
\end{proof}

\subsection{Definition and Existence of $\rho$-Kernels and $k$-th Power
Kernels}
\label{sb:define_rho}

We begin with some notation to describe the kernels we introduce here
and study throughout the rest of this paper.

Let $G$ be a Cayley bigraph on a group, $\cg$.
For any integer $k\ge 0$, let $\field^{k\times\cg}$ be the set of
$k\times |\cg|$ matrices with entries $m_{ig}\in\field$ indexed
over $i=1,\ldots,k$ and $g\in \cg$.
If $M\in\field^{k\times\cg}$, then we can view $M$ as a map
from $\field^\cg$ to $\field^k$.  Then $M$ gives rise to a morphism
of constant sheaves
$$
\underline M\from\underline\field^\cg\to\underline\field^k.
$$

For any $L\subset G$ and $g\in\cg$, we have an inclusion 
$\underline\field_{Lg}\to\underline\field$, which gives us an inclusion
$$
\underline\field_L\cg \to \underline\field\cg\isom \underline\field^\cg.
$$
Thus we get a monomorphism
$$
\iota_{L\cg}\from \underline\field_L\cg\to \underline\field^\cg,
$$
and, for any $M\in\field^{k\times \cg}$, a composite morphism
$$
\underline M\iota_{L\cg}\from \underline\field_L\cg \to
\underline\field^k.
$$
We shall often write $\iota$ instead of $\iota_{L\cg}$, since the subscript
$L\cg$ can be inferred from the source (even if two different $\iota$'s are 
involved).

\begin{definition} Let $L$ be a subgraph of a Cayley bigraph, $G$, on a
group, $\cg$, and let $\field$ be a field.
For any integer $k\ge 0$, we say that $M\in\field^{k\times\cg}$ is
{\em $L$-surjective} if the map, $\underline M \iota_{L\cg}$ is surjective.
If so, we that its kernel, $\ck=\ck_M(L,G,\cg)$, is a {$k$-th power
kernel for $(L,G,\cg)$}; 
if, in addition, $k=\rho(L)$,
we also say that $\ck$ is a {\em $\rho$-kernel for $(L,G,\cg)$}.
\end{definition}

Note that when kernels are defined in category 
theory, i.e., for a category with
a zero morphism, then a kernel is defined only up to (unique) isomorphism.
However, for sheaves on a graph, we can define the kernel of a morphism
${\cal F}_1\to{\cal F}_2$ uniquely, as the subsheaf of ${\cal F}_1$ that
is the kernel.  Hence we will speak of {\em the} kernel of a morphism, or
{\em its} kernel, for convenience; when we say 
``a kernel'' we shall mean the category theory
notion, i.e., any morphism ${\cal K}\to{\cal F}_1$ that is the equalizer
of ${\cal F}_1\to{\cal F}_2$ and the zero morphism.

Note that we could also define $k$-th power kernels when 
$\underline M\iota_{L\cg}$ is not surjective, as the element of the
derived category (see \cite{gelfand}) as a single shift of the
mapping cone of $\underline M\iota_{L\cg}$; we shall not pursue this
here.

The important point to notice is that if $k\le\rho(L)$, ``most''
matrices $M\in\field^{k\times\cg}$ are $L$-surjective.
We now demonstrate this, in a rather explicit fashion.


%
%

\begin{definition} We say that $M\in\field^{k\times\cg}$ is 
{\em totally linearly independent} (or just {\em totally independent})
if every subset $\cg'$ of $\cg$ of size $k$ we have
$\{m^g\}_{g\in\cg'}$ is linearly
independent, where $m^g$ denotes the column of $M$ corresponding to 
$g\in\cg$.
\end{definition}

\begin{lemma}\label{lm:more_than_rho} 
Let $L$ be a subgraph of a Cayley bigraph, $G$, on a group $\cg$.
Then the number of vertices of $L$ and the number of edges of either
colour in $L$ are all at least $\rho(L)$.
\end{lemma}
\begin{proof}
Adding vertices and edges to a graph does not decrease its reduced
cyclicity (i.e., its $\rho$).  So
if $P$ is an edge of colour $2$, let $L'$ be $L$ union all vertices of
$G$ and all edges of colour $1$.  Then 
$\rho(L')\ge\rho(L)$ and $L'$ has the same number of edges of colour
$2$ as $L$.  But if we discard the edges of colour $2$ from $L'$ we 
are left with a union of cycles, for which $\rho=0$, and discarding one edge
decreases $\rho$ by at most one (given equation~(\ref{eq:shrho})).  
Hence the number of edges of colour
$2$ in $L'$ is at least $\rho(L)$, and so the same is true of the number
of colour $2$ edges in $L$.

Similarly $L$ must have at least $\rho(L)$ edges of colour $1$.
Finally, since each vertex of $L$ is the head of at most one
edge of colour $1$, the number of vertices is also at least $\rho(L)$.
\end{proof}

Now we wish to describe $\rho$-kernels, both as a kernel of a sheaf
morphism and, alternatively, by explicitly giving their values and
restrictions.

\begin{definition}\label{de:free} 
Fix a subgraph, $L$, of a Cayley bigraph, $G$, on a group, $\cg$.
Fix an $M\in\field^{k\times\cg}$ for an integer $k\ge 0$.
For a subset $T\subset \cg$,
the {\em $T$-free subspace of $\ker(M)$} we mean the set
$$
\free_{T}=\free_T(M)=
\{ \vec a\in\ker(M) \ |\ \forall g\notin T,\ a_g=0 \}.
$$
A {\em free subspace of $\ker(M)$} is a subspace that is $T$-free
for some $T\subset \cg$.
For $P\in V_G\amalg E_G$, we set 
$$
\cg_L(P)=\{ g \in \cg \ |\  P\in Lg \} .
$$
\end{definition}
In the above definition, if $M\in\field^{k\times\cg}$ is totally 
independent, then for all 
$T\subset\cg$ we have
\begin{equation}\label{eq:free_dimension}
\dim(\free_T) = \max(0,|T|-k).
\end{equation}

\begin{lemma}\label{lm:construct_rho} 
Let $L$ be a subgraph of a Cayley bigraph, $G$, on a graph $\cg$.
Let $M\in\field^{k\times\cg}$ be totally independent, for some 
$k\le\rho(L)$.
Then
$$
\underline M \iota\from 
\underline\field_{L}\cg \to (\underline\field_G)^k
$$
is surjective.  Furthermore, if $\ck_M$
denotes its kernel, then for each $P\in V_G\amalg E_G$ we have
$$
\ck_M(P) = \free_{\cg_L(P)}(M),
$$
in the notation of Definition~\ref{de:free},
and the restriction maps for $\ck_M$ are the inclusions.  In particular,
$$
\dim\bigl( \ck_M(P) \bigr) = n_P-\rho(L),
$$
where $n_P=|\cg_L(P)|$.
(We shall sometimes write $\ck_M$ as $\ck_M(L)$ or $\ck_M(L,G,\cg)$
to emphasize $\ck_M$'s dependence upon $L$, $G$, and $\cg$.)
\end{lemma}
\begin{proof} 
For each $g\in\cg$ we have $\underline\field_L g =
\underline\field_{Lg}$.  Hence
for each $P\in V_G\amalg E_G$, we have
$$
(\underline\field_L g)(P) 
= (\underline\field_{Lg})(P) 
= \left\{ \begin{array}{ll}
\field & \mbox{if $P\in Lg$}  \\
0 & \mbox{if $P\notin Lg$} 
\end{array}\right.
$$
Hence
$$
(\underline\field_L \cg)(P) \isom \bigoplus_{g\in \cg_L(P)} \field,
$$
and
the image of $\underline M \iota$ in
$(\field_G)^k$ at $P$ is the span of the subcollection of the $n_P$
columns of $M$ corresponding to the elements of $\cg_L(P)\subset \cg$.
Since $G$ is a Cayley bigraph, $n_P$ is either the number of vertices,
edges of colour $1$, or edges of colour $2$ in $L$.
By Lemma~\ref{lm:more_than_rho} we have $n_P\ge\rho(L)$,
and hence this subcollection of $n_P$ vectors in $M$
spans $\field^k$.  Hence
$\underline M\iota$ is surjective at $P$, and its kernel,
$\free_{\cg_L(P)}$, is of dimension
$n_P-k$.  The restriction maps on $\underline\field_L\cg$
are, component by component,
those of the individual $\underline\field_L g$ over all $g\in\cg$,
and those are just inclusions; since $\ck$ is a subsheaf of 
$\underline\field_L \cg$, we have that $\ck$ inherits those restriction
maps.
\end{proof}


Note that it is easy to see, even with $\cg=\integers/3\integers$ and $L$
consisting of five edges, that
there need not be any graph theoretic
surjections $L\cg\to G^{\rho(L)}$, where $G^{\rho(L)}$ is $\rho(L)$ disjoint
copies of $G$; so in passing from the graphs
$L\cg$ and $G^{\rho(L)}$ to the sheaves $\underline\field_L \cg$
and $\field^{\rho(L)}$,
there exists a surjection of sheaves that does not arise from any surjection
of graphs.  So an added benefit of working with sheaves (aside from using
them to form
kernels useful in studying the SHNC) is that sheaves give
``additional surjections'' that don't exist in graph theory.

\section{Symmetry and Algebra of the Excess}
\label{se:mesymmetry}

In this section we make some general observations
about the maximum excess of $k$-th power kernels.  The main observation is
that given $(L,G,\cg)$ as usual, the maximum excess of $\ck_M(L)$ for
generic $M\in\field^{k\times\cg}$ 
is divisible by $|\cg|$, where by ``generic'' we mean for $M$ in
some subset of
$\field^{k\times\cg}$ that contains a nonempty, Zariski open subset
of $\field^{k\times\cg}$.
Let us outline this argument.

First, in Subsection~\ref{sb:symmetry},
we will show that for any $M\in\field^{k\times\cg}$ and $g\in\cg$
we have $\ck_M(L)g\isom\ck_{Mg}(L)$ where $Mg$ is obtained by an appropriate
action of $g$ on the columns of $M$.
This means that if ${\cal F}$ is the maximal (or minimal) excess maximizer
for $\ck_M(L)$, then ${\cal F}g$ is isomorphic to
the maximal (or, respectively, minimal)
excess maximizer for $\ck_{Mg}(L)$.
It may be helpful, albeit somewhat fanciful, 
to understand this symmetry via
two ``observers'' looking at the exact sequence
$$
0\to \ck_M \to \underline\field_L\cg\xrightarrow{\ \underline M\iota\ }
\underline\field^k \to 0,
$$
one who examines this at $P\in V_G\amalg E_G$, and the other at $Pg$,
for $g$ fixed and $P$ varying; for example, $\underline\field_L\cg$
``looks'' the same to both observers, except that its summands appear
permuted from one observer to the other.

In Subsection~\ref{sb:generic} we discuss the generic maximum excess of
$\ck_M=\ck_M(L)$ with
$L$ fixed and
$M\in\field^{k\times\cg}$ variable (and $G,\cg,k$ fixed).  
The key to this discussion is considering
what we call ``dimension profiles,'' which we now define.
\begin{definition}\label{de:dimension_profile}
By a {\em dimension profile} on a bigraph, $G$, we mean a function
$$
n\from V_G\amalg E_G \to \integers_{\ge 0}.
$$
For any such $n$, we set
$$
\chi(n) = \sum_{v\in V_G} n(v) - \sum_{e\in E_G} n(e);\qquad
|n| = \sum_{P\in V_G\amalg E_G} n(P).
$$
Any sheaf, ${\cal F}$, on $G$ determines a dimension profile, $\dim({\cal F})$,
as the function $P\mapsto \dim({\cal F}(P))$.
For any dimension profile, $n$, of a Cayley bigraph, $G$, on a group, $\cg$,
any subgraph, $L\subset G$, 
any field, $\field$, and any $k\ge 0$,
let 
$$
\cm(n)=\cm(n,L,G,\cg,\field,k)
$$
be the set
of $M\in\field^{k\times \cg}$ for which 
$\ck_M=\ck_M(L,G,\cg)$ exists (i.e., $M$ is
$L$-surjective) and has 
a subsheaf, ${\cal F}$, with
$\dim({\cal F})=n$.
\end{definition}
We easily see that for all $n$, $\cm(n)$ is a constructible subset of
$\field^{k\times\cg}$.  
Let ${\cal N}$ be the set of $n$ for which $\cm(n)$ is generic,
i.e., contains a Zariski open subset of $\field^{k\times\cg}$;
since $\cm(n)$ is constructible, it is equivalent to say that
its Zariski closure is $\field^{k\times\cg}$.
The generic (in $M$) value of the maximum excess of $\ck_M$ is the
largest value of
$-\chi(n)$ among those $n\in {\cal N}$; let ${\cal N}'$ be the subset
of $n\in {\cal N}$ which attain this largest $-\chi(n)$ 
value.
Since $\cg$ is finite, for any $n\in{\cal N}$ there
is a generic set of $M$ such that
$Mg\in\cm(n)$ for all $g\in\cg$.
Hence, if $n\in {\cal N}'$ is chosen with 
$|n|$ at a maximum value (or minimum value), then by the uniqueness
of the maximum (or minimum) maximizer of the excess, the
symmetry $\ck_{M}g\isom \ck_{Mg}$ implies that $n(P)=n(Pg^{-1})$
for all $P\in V_G\amalg E_G$ and $g\in \cg$.  Hence the generic
maximum excess of $\ck_M$,
which equals
$-\chi(n)$ for any $n\in{\cal N}'$, is divisible by $\cg$.

We wish to remark that the generic maximum excess is not generally
attained by all
$M\in\field^{k\times\cg}$.  For example, our approach to the SHNC
is based on the fact that the generic maximum excess of a 
$\rho$-kernel is zero, i.e., for $k=\rho(L)$ 
(see Theorem~\ref{th:vanishing_rho}).
However, if $M$ is zero in one column, but totally independent in the
others, then $M$ will still be $L$-surjective provided that $L$ has
at least $\rho(L)+1$ edges of each colour.
(A simple example of such an $L$ can be obtained by deleting one
edge of each colour
from ${\rm Cayley}(\integers/m\integers;1,1)$ with $m\ge 2$.)
In such a situation, the fact that $M$ has a column of zeros implies
that $\ck_M$ has
$\field_L g$ as a subsheaf (more precisely, a direct summand)
for some $g\in\cg$, and hence 
the maximum excess of $\ck_M$ will be at least $\rho(L)$.
Hence any $L$ that has at least $\rho(L)+1$ edges
of each colour, and for which $\rho(L)>0$,
has a $\rho$-kernel of positive maximum excess.
Hence it is essential to study the maximum excess of $\ck_M$ with some
restrictions on $M$, i.e., requiring some special properties of $M$; in
our case, these properties restrict $M$ to some generic subset of
$\field^{\rho(L)\times \cg}$.

\subsection{Symmetry of $k$-th Power Kernels}
\label{sb:symmetry}

The point of this subsection is to establish the following symmetry
of $k$-th power kernels.
\begin{theorem}
\label{th:symmetry_isomorphism}
Let $L$ be a subgraph of a Cayley bigraph, $G$, 
on a group, $\cg$,
and let $\field$ be an arbitrary field.
Let $k$ be an arbitrary non-negative integer and $M\in\field^{k\times \cg}$.
Let $Mg$ be the matrix (described earlier)
whose $g'$ column, for $g'\in\cg$, is the $g'g^{-1}$ column of $M$.
Then $M$ is $L$-surjective iff $Mg$ is $L$-surjective, and if so
then $\ck_M(L)g\isom\ck_{Mg}(L)$.
\end{theorem}

\begin{proof}
We begin our discussion of symmetry with a somewhat pedantic, but
important, point.  If ${\cal A}$ is a category in which finite direct
sums exists, such as an additive category, and $\{ A_s\}_{s\in S}$ is a 
family of objects in the category indexed upon a finite set, $S$, then
their direct sum comes with projections
$$
f_r\from \bigoplus_{s\in S} A_s \to A_r
$$
for each $r\in S$.  If $\pi\from S\to S$ is a permutation, then we
have a ``component permuting map,'' $P=P(\pi)$, given by
$$
P(\pi)\left( \bigoplus_{s\in S} A_s \right) =
\bigoplus_{s\in S} A_{\pi(s)},
$$
The two direct sums in this last equation are isomorphic, but not equal
(e.g., the direct sum on the right-hand-side has the projection
$f'_r$ whose target is $A_{\pi(r)}$, not to $A_r$, for each $r\in \cg$).
We shall need to keep the seemingly unimportant operator
$P=P(\pi)$ in mind in order to make things
precise for this subsection.
If $A_\bullet$ is any direct sum indexed on $\cg$, then we easily see
$P(\pi_2)(P(\pi_1) A_\bullet) = P(\pi_2\circ\pi_1) A_\bullet$.

Again, let $\field$ be a field, $k\ge 0$ an integer,
$L$ a subgraph of a Cayley bigraph, $G$, on a group, $\cg$, and
$M\in\field^{k\times \cg}$ that is $L$-surjective.
We have an exact sequence.
\begin{equation}
\label{eq:multiply_by_g0}
0 \longrightarrow \ck \longrightarrow 
\bigoplus_{g'\in\cg} \underline\field_L\,g' 
\xrightarrow{\ \underline M\iota\ } 
\underline\field^k \longrightarrow 0 .
\end{equation}
For a $g\in\cg$, applying
$\pi^*_{g^{-1}}$, of Theorem~\ref{th:acts}, to this sequence gives an
exact sequence:
\begin{equation}\label{eq:multiply_by_g}
0 \longrightarrow \ck g \longrightarrow 
\bigoplus_{g'\in\cg} \underline\field_L\,g'g  
\xrightarrow{\ \pi^*_{g^{-1}}(\underline M\iota)\ }
\underline\field^k g \longrightarrow  0.
\end{equation}
We have $\underline\field g=\underline\field$ since $\underline\field$
is a constant sheaf (note the we mean that the two are equal, not merely
isomorphic).
Note that $\pi^*_{g^{-1}}$ acts on sheaves by renaming the vertices of $G$,
so it acts on $\underline M\iota$ only by permuting 
sheaf inclusions $\underline\field_{Lg''}\to\underline\field$ for various
values of $g''$; in other words,
$$
\pi^*_{g^{-1}}(\underline M\iota_{\underline\field_L\cg}) =
\underline M\iota',
$$
where $\iota'$ is $\iota$ with the source $\underline\field_L\cg g$.
Hence we may write equation~(\ref{eq:multiply_by_g}) as
\begin{equation}
\label{eq:multiply_by_g2}
0 \longrightarrow \ck g \xrightarrow{\ \ j\ \ }
\bigoplus_{g'\in\cg} \underline\field_L\,g'g
\xrightarrow{\ \underline M\iota'\ }
\underline\field^k \longrightarrow  0 ,
\end{equation}
where $j$ is an inclusion.

Also, we have
$$
P(\pi_{g^{-1}}) \left( 
\bigoplus_{g'\in\cg} \underline\field_L\,g'g 
\right) 
=
\bigoplus_{g'\in\cg} \underline\field_L\,\pi_{g^{-1}}(g')g
=
\bigoplus_{g'\in\cg} \underline\field_L\,g' = \underline\field_L\cg.
$$
Hence from equation~(\ref{eq:multiply_by_g2}) we get a sequence
$$
0 \longrightarrow \ck g \xrightarrow{\ j\ }
\bigoplus_{g'\in\cg} \underline\field_L\,g'g
\xrightarrow{\,P(\pi_{g^{-1}})\,}
\underline\field_L\cg
\xrightarrow{\,P(\pi_{g})\,}
\bigoplus_{g'\in\cg} \underline\field_L\,g'g
\xrightarrow{\,\underline M\iota'\,}
\underline\field^k \longrightarrow  0 ,
$$
and hence an exact sequence (since $P(\pi_g)$ and $P(\pi_{g^{-1}})$ are
isomorphisms)
\begin{equation}
\label{eq:multiply_by_g3}
0 \longrightarrow \ck g \xrightarrow{\ P(\pi_{g^{-1}})\circ j\ } 
\underline\field_L\cg
\xrightarrow{\ \underline M\iota'P(\pi_g)\ }
\underline\field^k \longrightarrow 0 ,
\end{equation}
with $j$ being the inclusion in equation~(\ref{eq:multiply_by_g2}).
But clearly
$$
\underline M\iota' P(\pi_g) 
= \underline M P(\pi_g) \,\iota_{\underline\field_L\cg}
= \underline{M \pi_g} \,\iota_{\underline\field_L\cg},
$$
where $\pi_g$ is viewed as operating vectors in $\field^\cg$
sending $\alpha\in\field^\cg$, viewed as a function $\alpha\from\cg\to
\field$ to $\pi_g\alpha$ given by
$$
g'\mapsto \alpha(\pi_g(g'))=\alpha(g'g).
$$

Hence
setting $Mg=M\pi_g$, 
we get a short exact sequence
\begin{equation}
\label{eq:multiply_by_g4}
0 \longrightarrow \ck g \longrightarrow  
\underline\field_L\cg
\xrightarrow{\ \underline{Mg}\,\iota\ }
\underline\field^k \xrightarrow  0 .
\end{equation}
Hence $\ck_{Mg}(L)$ is, up to isomorphism, 
just $\ck g$.

To complete the proof of the theorem, it remains to find that permutation
that brings the columns of $M$ to those of $Mg$.
If $M=\{m_{i,g'}\}$ give $M$'s entries,
then
for any $w\in\field^\cg$, the $i$-th component of $(Mg)w=M(\pi_g w)$
is
$$
(M(\pi_g w))_i = \sum_{g'\in\cg} m_{i,g'}(\pi_g w)_{g'}
= \sum_{g'\in\cg} m_{i,g'}w_{g'g} 
= \sum_{g''\in\cg} m_{i,g''g^{-1}} w_{g''}.
$$
Hence the $i,g''$ entry of $Mg$ is $m_{i,g''g^{-1}}$, so the
$g''$ column of $Mg$ is the $g''g^{-1}$ of $M$.
\end{proof}

We wish to make a few comments on equation~(\ref{eq:multiply_by_g4})
and how we derived it.  First, kernels, in category
theory, are defined only up to isomorphism;
this is why we can ``forget about'' $P(\pi_{g^{-1}})j$ in
equation~(\ref{eq:multiply_by_g3}); it is only important to know that
this arrow gives an exact sequence there and in
equation~(\ref{eq:multiply_by_g4}).

Note that the two actions of $g\in\cg$ in
equation~(\ref{eq:multiply_by_g4}) are right $\cg$ actions on the exact
sequence.  To see this, first note that
$$
(Mg)g' = (M\pi_g)\pi_{g'} = M\pi_{gg'}=Mgg'.
$$
Then note that if we take the procedure for going from
equation~(\ref{eq:multiply_by_g0}) to
equation~(\ref{eq:multiply_by_g4}) and then do the same procedure with
$g$ replaced by $g'$, then we easily see (paying close attention to the
order of the
$P$'s, the $\pi$'s, and the $\iota$'s) that we get the same equation
as
equation~(\ref{eq:multiply_by_g0}) with $\ck$ replaced by $\ck gg'$ and
$M$ replaced by $Mgg'$.

We wish to comment on something that seems a bit contradictory.
The map $g\mapsto P(\pi_{g^{-1}})$ is a left action, and so it may seem
strange that
our forgotten monomorphism $P(\pi_{g^{-1}})j$
in equation~(\ref{eq:multiply_by_g3}) involves a left action.
But note that if we apply $g'$ to equation~(\ref{eq:multiply_by_g4}) we
get
$$
0 \longrightarrow \ck gg' 
\xrightarrow{\ \pi^*_{{g'}^{-1}}(P(\pi_{g^{-1}})\circ j)\ } 
\underline\field_L\cg g'
\xrightarrow{\ \pi^*_{{g'}^{-1}}(\underline M\iota'P(\pi_g))\ }
\underline\field^k \longrightarrow 0 ,
$$
and hence an exact sequence
$$
0 \longrightarrow \ck gg' 
\xrightarrow{\ \alpha\ }
\underline\field_L\cg
\xrightarrow{\ \beta\ }
\underline\field^k \longrightarrow 0 ,
$$
where
\begin{eqnarray*}
\alpha &=& P(\pi_{{g'}^{-1}})\circ  
\pi^*_{{g'}^{-1}}(P(\pi_{g^{-1}})\circ j) = P(\pi_{{g'}^{-1}})
P(\pi_{g^{-1}})j',\\
\beta &=& 
\pi^*_{{g'}^{-1}}(\underline M\iota'P(\pi_g))
P(\pi_{g'}), \\
\end{eqnarray*}
for an inclusion $j'$.  Examining $\alpha$ we see that
$P(\pi_{{g'}^{-1}})$ is applied to the left of 
$P(\pi_{g^{-1}})$, whose product equals $P(\pi_{{(gg')}^{-1}})$, so
that $g'$ appears to the right of $g$.

A similar remark applies for the column permuting rule taking
$M$ to $Mg$:
$g\mapsto \pi_{g^{-1}}$ is a left action, not a right action.
However, 
if $f\from\cg\to T$ is any function from $\cg$ to a set, $T$, then
defining a function $fg$ via $(fg)(g')=f(g'g^{-1})$ defines a {\em right}
action of $\cg$ on functions from $\cg$ to $T$; indeed, for $f\from\cg\to T$
and $g,g_1,g_2\in\cg$ we have
$$
\bigl((fg_1)g_2)\bigr)(g) = (fg_1)(gg_2^{-1}) = f(gg_2^{-1}g_1^{-1}) =
f\bigl( g(g_1g_2)^{-1} \bigr) = \bigl( f(g_1g_2)\bigr)(g).
$$
So the left action $g\mapsto \pi_{g^{-1}}$ turns into a right action
when it acts on the argument of a function.

We finish this subsection with a corollary 
of Theorem~\ref{th:symmetry_isomorphism}
that is our sole application of the theorem.

\begin{corollary}\label{cr:group_action}
Let $n$ be a dimension profile for Cayley bigraph, $G$, on a group,
$\cg$.
Let $L$ be a subgraph of $G$, let $\field$ be a field, and 
let $k\ge 0$ be an integer.
Then 
for any $g\in\cg$, we have 
$$
\cm(ng,L,G,\cg,k)=\cm(n,L,G,\cg,k)g,
$$
where $ng$ is given by
$$
(ng)(P)=n(Pg^{-1})
$$
for all $P\in V_G\amalg E_G$.
\end{corollary}
(We easily check that the action $g\to ng$ in this corollary is a
right action, similar to the above discussion of the action on functions
from $\cg$ to a set, $T$.)
\begin{proof}
Let $g\in\cg$ and
$M\in\cm(n)$.  Then there exists an ${\cal F}\subset\ck_M$
such that $\dim({\cal F})=n$.  Then we have
${\cal F}g\subset\ck_M g$ and we have $\dim({\cal F}g)=\dim({\cal F})g$,
since
$$
\dim\bigl(({\cal F}g)(P)\bigr) = \dim\bigl({\cal F}(Pg^{-1})\bigr)
$$
for all $P\in V_G\amalg E_G$.
But we have an isomorphism $\iota_g\from\ck_{M} g\to \ck_{Mg}$ of sheaves
on $G$; so on the one hand 
we have $\iota_g {\cal F}\subset \ck_{Mg}$, and on the other hand, since
isomorphisms preserve the dimension profile, we have $Mg\in\cm(n')$ where
$$
n'=\dim(\iota_g{\cal F}g)=\dim({\cal F}g) = ng.
$$
Hence $M\in\cm(n)$ implies that $Mg\in\cm(ng)$.  Applying this observation
to $M$ replaced with $Mg$ and $g$ replaced with $g^{-1}$ (or simply
reversing the argument in this proof) shows the converse.
Hence $\cm(n)g=\cm(ng)$.
\end{proof}

\subsection{Generic Maximum Excess}
\label{sb:generic}

If $\field$ is a field and $r\ge 1$ an integer, then by a {\em generic}
subset of $\field^r$ we mean a subset that contains 
$$
\{(x_1,\ldots,x_r)\in\field^r\ |\ p(x_1,\ldots,x_r)\ne 0 \}
$$
for some nonzero polynomial, $p$.
Algebraic geometry and generic subsets are most commonly discussed
(at least on the most basic level) under the assumption that
$\field$ is algebraically closed.  Under this situation, all generic
sets are nonempty; this remains true if $\field$ is infinite, or
if the polynomial, $p$, above is fixed and $\field$ is finite but
sufficiently large.

In order to have a sensible definition of generic and to conform to the
algebraic geometric literature, we will freely assume that $\field$ is
algebraically closed. 
However, the theorems we obtain in this section and the next will
be valid for any infinite field or ``sufficiently large'' finite
field, $\field$, by applying these theorems to the algebraic closure
of $\field$, finding the associated polynomials, $p$, to the generic
sets of interest, and determining how large $\field$ needs to be so
that the generic sets are nonempty.
The reader may find it amusing to note that in all our discussion of
generic sets and generic conditions, all that we ultimately
care about
is that certain of these generic sets are nonemtpy
(e.g., that there is at least one $\rho$-kernel for $(L,G,\cg)$ with
vanishing maximum excess).

Let us review some notation in algebraic geometry; see
\cite{hartshorne}, Chapter 1, Section 1.
Let us assume that $\field$ is algebraically closed.
Let $\affine^N=\affine^N(\field)$, where $N$ is an integer or a set or
a product thereof, denote affine $N$ space over $\field$, i.e., the
set $\field^N$, with its usual Zariski topology.
(When we speak of topological notions on $\field^N$ we mean those
of $\affine^N(\field)$; in the literature $\affine^N(\field)$ connotes
$\field^N$ viewed as a topological space, or scheme, etc.)
Recall that a 
{\em locally closed} set is the intersection of an open and closed set
(i.e., a subset of $\affine^N$ determined as the zeros of some
polynomials and complement of the zeros of some other polynomials),
and a
{\em constructible} set 
on $\affine^N$ amounts to a finite disjoint union of locally closed sets
(see \cite{hartshorne}, Exercise II.3.18).

\begin{lemma}
\label{lm:construct}
Let $\field$ be an algebraically closed field, $k\ge 0$ an integer, and $L$ a subgraph of a
Cayley bigraph, $G$, on a group, $\cg$.
For each $n\from V_G\amalg E_G\to\integers_{\ge 0}$,
${\cal M}(n)={\cal M}(n,L,G,\cg,k)$ is a constructible set.
\end{lemma}
\begin{proof}
We introduce $|n|\,|\cg|$ indeterminates as follows:
for each $P\in V_P\amalg E_P$, and $i=1,\ldots,n(P)$,
let $x_{P,i}$ be a vector of indeterminates indexed on $\cg$
(there are $|n|$ vector variables $x_{P,i}$, for a total of
$|n|\,|\cg|$ indeterminates).
We note that $M\in{\cal M}(n)$ precisely
when one can find a solution for $M$ and $x_{P,i}$ to the conditions
\begin{enumerate}
\item $M$ is $L$-surjective; i.e., for each $P\in V_G\amalg E_G$,
$\field^k$ is spanned by
the columns of $M$ corresponding to the elements 
of $\cg_P(L)$;
\item for all $P$ and $i$ we have that
$x_{P,i}$ has zero components outside of $\cg_L(P)$;
\item for all $P$ and $i$, $Mx_{P,i}=0$;
\item for all $P$, $x_{P,1},\ldots,x_{P,n_P}$ are linearly independent;
\item for all $e\in E_G$ and all $i$ we have that
$x_{e,i},x_{te,1},x_{te,2},\ldots,x_{te,n_{te}}$ are linearly dependent,
and similarly with $he$ replacing $te$.
\end{enumerate}
The dependence or independence or spanning
of vectors reduces to the vanishing or nonvanishing of determinants
of the vectors' coordinates.  Hence all the above equations give us
a collection of polynomials $f_i\in \field[M,x]$ (polynomials in the entries
of $M$ and the $x_{P,i}$'s) and $\mt f_j\in \field[M,x]$ such
that $M\in{\cal M}(n)$
iff for some $x$ we have $(M,x)\in C$, where $C$ is the set of
$(M,x)$ for which $f_i(M,x)=0$ for all relevant $i$ and
$\mt f_j(M,x)\ne 0$ for all relevant $j$; hence $C$ is constructible.
But $M\in\cm(n)$ iff $(M,x)\in C$ for some $x$;
hence ${\cal M}(n)$ is the image of $C$ under the
projection
$$
\affine^{k\times\cg}\times\affine^{|n|\times \cg} \to
\affine^{k\times\cg}.
$$
But any projection from an affine space to another by omitting some of
the coordinates has the property that it takes constructible sets to
constructible sets (see Exercise~II.3.19 of \cite{hartshorne} or
Theorem~3.16 of \cite{harris}, noting that such a projection is
both regular and of finite type).
Hence ${\cal M}(n)$, the image of $C$,
is constructible.
\end{proof}

We recall that a generic subset, $S$,
of some affine space, $\field^T$, is a subset that contains
a nonemtpy Zariski open subset of the space; if $S$ is constructible,
then $S$ is generic iff its Zariski closure is the affine space.

Next we claim that $\cm(n)$ is generic 
in $\field^{k\times\cg}$ for at least one $n$, provided that
$k\le\rho(L)$, and that
$\cm(n)=\emptyset$ for all but finitely many $n$.
Indeed,
for any totally independent $M\in\field^{k\times \cg}$ we have that
$M$ is $L$-surjective (for $(L,G,\cg,\field,k)$), and the zero sheaf,
${\cal Z}$, has
$\dim({\cal Z})=0$.
Hence the
Zariski closure of ${\cal M}(0)$ is $\field^{k\times \cg}$.
Furthermore, $\ck_M(P)$, for any $P\in V_G\amalg E_G$, is of
dimension at most $|\cg|-k$; hence ${\cal M}(n)=\emptyset$ unless
$|n(P)|\le |\cg|-k$ for all $P\in V_G\amalg E_G$, and there are only
finitely many such $n$.

\begin{definition} 
Let $L$ be a subgraph of a Cayley bigraph, $G$, on a group, $\cg$, and
let $\field$ be an algebraically closed field.
Let $k\le\rho(L)$ be a non-negative integer.
We say that $n\from V_G\amalg E_G\to\integers_{\ge 0}$ is {\em generic
for $(L,G,\cg,\field,k)$}
if the Zariski closure of ${\cal M}(n)$ is $\affine^{k\times\cg}(\field)$.
We define the {\em generic maximum excess of
$(L,G,\cg,\field,k)$}
to be the largest value of $-\chi(n)$ for which $n$ is generic.
We define $n$ to be a {\em maximal profile} (respectively,
{\em minimal profile}) {\em of $(L,G,\cg,\field,k)$}
if $n$ is generic, $-\chi(n)$ equals the
generic maximum excess, and there is no $n'\ne n$ which is generic with
$-\chi(n')=-\chi(n)$ and $n'(P)\ge n(P)$ (respectively $n'(P)\le n(P)$)
for all $P\in V_G\amalg E_G$.
\end{definition}

\begin{theorem}\label{th:generic_divisible}
Let $L$ be a subgraph of a Cayley bigraph, $G$, on a group, $\cg$.
Let $\field$ be an algebraically closed field, and $k\le\rho(L)$ an integer.
There is a unique maximal profile and a unique minimal profile for
$(L,G,\cg,\field,k)$.  
Furthermore, if $n$ is either the maximal or minimal profile, 
and $P\in V_G\amalg E_G$,
then $ng=n$ for all $g\in\cg$ (in the notion 
of Corollary~\ref{cr:group_action}).
In particular $-\chi(n)$ is divisible by $|\cg|$.
\end{theorem}
Actually, the proof below shows that
the theorem is still true when $k>\rho(L)$, provided
that $L$ has at least $k$ edges of each colour, so that a totally
independent $M\in\field^{k\times\cg}$ is $L$-surjective.
\begin{proof}
Let $n_1,n_2$ be two maximal profiles for $(L,G,\cg,\field,k)$.  
Let us show
that $n_1=n_2$.  
Consider the subset, $S$, of $\field^{k\times\cg}$, $M$, such that
$M\in{\cal M}(n_i)$ for $i=1,2$ and ${\rm m.e.}(\ck_M)=-\chi(n_1)=-\chi(n_2)$.
Clearly
$$
S = {\cal M}(n_1)\cap{\cal M}(n_2) \cap \bigcap_{n\ {\rm s.t.}\ 
-\chi(n)> -\chi(n_1)} \overline{{\cal M}(n)},
$$
where $\overline{{\cal M}(n)}$ denotes the complement of ${\cal M}(n)$.
But if $-\chi(n)>-\chi(n_1)$ then, by assumption, $n$ is not generic, and
hence $S$ is the intersection of a finite number of generic subsets of
$\field^{k\times\cg}$; hence $S$ is a generic subset of
$\field^{k\times\cg}$, as well.  But any element, $M\in S$, has
subsheaves ${\cal F}_1,{\cal F}_2$, of $\ck_M$ which obtain the maximum
excess of $\ck_M$ and with $\dim({\cal F}_i)=n_i$ for $i=1,2$.
But then ${\cal F}={\cal F}_1+{\cal F}_2$ also achieves the maximum excess
and has $\deg({\cal F})\ge n_i$ for $i=1,2$.  Hence
$$
S \subset \bigcup_{n\ {\rm s.t.}\ -\chi(n)=-\chi(n_1),\ n\ge N} {\cal M}(n),
$$
where $N=\max(n_1,n_2)$.
Since the union on the right-hand-side is a finite union of constructible sets,
the closure of one of these sets is $\field^{k\times\cg}$.
Hence there is an $n$ with $-\chi(n)=-\chi(n_1)$ and $n\ge N=\max(n_1,n_2)$
such that $n$ is generic; 
but if $n_1\ne n_2$, then $n$ does not equal either of them and
is at least as big as either, which contradicts the maximality of the $n_i$,
$i=1,2$.  Hence $n_1=n_2$, and the maximal profile is unique.

We argue similarly for the minimal profile, replacing
${\cal F}_1+{\cal F}_2$ with ${\cal F}_1\cap{\cal F}_2$.

Let $n$ be the maximal profile for $(L,G,\cg,\field,k)$ (now known to
be unique).
Since $\cm(n)$ is a generic subset of $\field^{k\times\cg}$, so is
$\cm(ng)=\cm(n)g$ for any $g\in\cg$.  But then $ng$ is also a maximal
profile, since clearly $\chi(n)=\chi(ng)$ and $|n|=|ng|$.
Hence $n=ng$ for all $g\in \cg$.  The same is true of the minimal profile.

It follows that the maximal (or minimal) profile, $n$, is
invariant under $\cg$, and hence has the
same value on all the vertices, on all the edges of colour $1$, and
on all the edges of colour $2$.
Hence $-\chi(n)$ is divisible by $|\cg|$ for the maximal (or minimal) profile,
and hence the generic maximum excess of $(L,G,\cg,\field,k)$ is divisible by
$|\cg|$.
\end{proof}

\section{Variability of $k$-th Power Kernels}
\label{se:mevar}

The main goal of this section is to prove the following theorem.

\begin{theorem}\label{th:increase}
Let $L$ be a subgraph
of a Cayley bigraph, $G$, on a group $\cg$.  Let $\field$ be an algebraically closed field,
and let $k\le\rho(L)$ be a positive integer.
Then the generic maximum excess of $(L,G,\cg,\field,k)$ is at most
that of $(L,G,\cg,\field,k-1)$, and we have equality iff
both excesses are zero.
\end{theorem}
As a consequence we get the following theorem.
\begin{theorem}\label{th:L'}
Let $L$ be a subgraph
of a Cayley bigraph, $G$, on a group $\cg$.  Let $\field$ be an algebraically closed field,
and let $k\le\rho(L)$ be a positive integer.
Let $L'$ be obtained from $L$ by removing a single edge.
Then the generic maximum excess
of $(L',G,\cg,\field,k-1)$ is at least that
of $(L,G,\cg,\field,k)$.
\end{theorem}
(As before, this theorem is also true if $k>\rho(L)$, provided that
$L$ has at least $k$ edges of each colour, so that a totally independent
$M\in\field^{k\times\cg}$ is $L$-surjective.)

A second goal of this section is to establish some general relations
between kernels $\ck=\ck_M(L)$ as $M$ and $L$ vary.  We shall derive
two interesting, short exact sequences.  First we 
establish a short exact sequence
\begin{equation}\label{eq:first_induct_exact}
0\to \ck_{M}(L)\to\ck_{M'}(L)\to\underline\field\to 0,
\end{equation}
for any $M\in\field^{k\times\cg}$ and $M'$ obtained from $M$ by deleting
the last row.
Second we establish 
a short exact sequence
\begin{equation}\label{eq:second_induct_exact}
0\to \ck_{M'}(L')\to\ck_{M'}(L)\to{\cal E}\to 0,
\end{equation}
with $L,L'$ as in Theorem~\ref{th:L'}, $M'\in\field^{(k-1)\times \cg}$
such that $\ck_{M'}(L')$ exists (i.e., $M'$ is $L'$-surjective), and
${\cal E}$ a sheaf with ${\cal E}(V)=0$ and ${\cal E}(E)$
of dimension $|\cg|$.

Equation~(\ref{eq:second_induct_exact}) will be used along with
Theorem~\ref{th:generic_divisible} to show that
Theorem~\ref{th:increase} implies Theorem~\ref{th:L'}.

Theorem~\ref{th:increase} will not be proven with short exact sequences,
but rather with a careful analysis of the unique minimal maximizer
of the excess of $\ck_M(L)$ and of that of $\ck_{M'}(L)$.
The sequence in equation~(\ref{eq:first_induct_exact}) gives a relationship
between $\ck_M(L)$ and $\ck_{M'}(L)$, but we don't know how to directly
use this to conclude anything interesting about the two sheaves, 
such as the result
of Theorem~\ref{th:increase}.

At this point we will divide our discussion into subsections.
In Subsection~\ref{sb:exact}, we will discuss the exact sequences
related to our proof.
In Subsection~\ref{sb:observe} we give the main observation
of how the maximum excess changes in passing to subsheaves, and give
an intuitive reason why the generic maximum excess of $\ck_{M'}(L)$,
as above, should be strictly greater than that of $\ck_M(L)$ provided that
these numbers don't
both vanish.
In Subsection~\ref{sb:rigour}
we mimic the notation of Section~\ref{se:mesymmetry} to include
a discussion of $\ck_{M'}(L)$ as above and make our arguments precise,
finishing the proof of Theorem~\ref{th:increase}; this will easily
yield Theorem~\ref{th:L'}.

\subsection{Variability as Exact Sequences}
\label{sb:exact}

Let $L$ be a subgraph of a Cayley bigraph, $G$, on a group, $\cg$.
For any non-negative integer, $k\le \rho(L)$, we have that a generic
$M\in\field^{k\times\cg}$ gives rise to a short exact sequence
\begin{equation}\label{eq:first_exact_k}
0\to \ck_M(L)\to \underline\field_L \cg \to \underline\field^k \to 0.
\end{equation}
First we considering the variance of this equation in $M$; in other words,
fix an $M\in\field^{k\times\cg}$ such that 
$$
\underline M\iota\from\underline\field_L\cg\to\underline\field^k
$$
is surjective.  Then we have an exact sequence given in
equation~(\ref{eq:first_exact_k}); if $M'\in\field^{(k-1)\times\cg}$ is
$M$ with its last row deleted, we have a similar exact sequence
\begin{equation}\label{eq:second_exact_k}
0\to \ck_{M'}(L)\to \underline\field_L \cg \to \underline\field^{k-1} \to 0.
\end{equation}
Notice that this discussion and everything below will remain essentially
the same if, more generally, $M'$ is taken
to be $M$ followed by any surjective map
$\field^k\to\field^{k-1}$.
In any event, we get a digram:
\begin{diagram}[nohug,height=2.5em,width=3em,tight]
 0&\rTo &\ck_M(L) &\rTo&\underline\field_L\cg &\rTo^{\underline M\iota} 
&\underline\field^{k}&\rTo&0
 \\
 &&\dDashto&&\dTo_{\isom}&&\dTo&& \\
 0&\rTo &\ck_{M'}(L) &\rTo&\underline\field_L\cg &\rTo^{\underline{M'}\iota} 
&\underline\field^{k-1}&\rTo&0
 \\
\end{diagram}
where the dotted arrow from $\ck_{M'}(L)$ to $\ck_M(L)$ is inferred from
the solid arrows; furthermore, given that the solid horizontal
arrows consist of an 
isomorphism and epimorphism, we infer that the dotted arrow is a
monomorphism.  We then complete the diagram to obtain a diagram
\begin{diagram}[nohug,height=1.8em,width=3.0em,tight]
 &&&&&&0&& \\
 &&&&&&\dTo&& \\
 &&0&&0&&\underline\field&& \\
 &&\dTo&&\dTo&&\dTo&& \\
 0&\rTo &\ck_M(L) &\rTo&\underline\field_L\cg &\rTo^{\underline M\iota} 
&\underline\field^{k}&\rTo&0
 \\
 &&\dTo&&\dTo&&\dTo&& \\
 0&\rTo &\ck_{M'}(L) &\rTo&\underline\field_L\cg &\rTo^{\underline{M'}\iota} 
&\underline\field^{k-1}&\rTo&0
 \\
 &&\dTo&&\dTo&&\dTo&& \\
 &&\ck_{M'}(L)/\ck_M(L)&&0&&0&& \\
 &&\dTo&&&&&& \\
 && 0      &&&&&&
\end{diagram}
A simple diagram chase shows that the nonzero upper right sheaf,
$\underline\field$, and the nonzero lower left sheaf,
$\ck_{M'}(L)/\ck_M(L)$, are isomorphic.  Hence we obtain the short
exact sequence in equation~(\ref{eq:first_induct_exact}).

An analogous exact sequence can be obtained by varying $L$ in
equation~(\ref{eq:first_exact_k}).  Let $L'\subset L$ be a subgraph
of $L$.  Fix an $M\in\field^{k\times\cg}$ that induces a surjection
$\underline\field_{L'}\cg\to\underline\field^k$.  Then we get a diagram:
\begin{diagram}[nohug,height=2.5em,width=3em,tight]
 0&\rTo &\ck_M(L') &\rTo&\underline\field_{L'}\cg &\rTo
&\underline\field^{k}&\rTo&0
 \\
 &&\dDashto&&\dTo&&\dTo_{\isom}&& \\
 0&\rTo &\ck_{M}(L) &\rTo&\underline\field_L\cg &\rTo
&\underline\field^{k}&\rTo&0
 \\
\end{diagram}
Since $\underline\field_{L'}\cg\to\underline\field_L\cg$ is an injection,
and the last vertical arrow is an isomorphism, the inferred dotted arrow
is an injection.  We therefore add a bottom row to the diagram and 
infer from the $3\times 3$ Lemma that
$$
\ck_M(L)/\ck_M(L') \isom (\underline\field_L/\underline\field_{L'})\cg.
$$
In particular, if $L'$ is obtained from $L$ by removing $m$ edges,
then we infer equation~(\ref{eq:second_induct_exact}) (with $M$ here
replaced by
$M'$, and $k$ implicit here replaced by $k-1$), where 
$\ce$ is a sheaf with $\ce(V)=0$ and $\ce(E)$ being of dimension
$m|\cg|$.

\subsection{Maximum Excess and Subsheaves}
\label{sb:observe}

The goal of this section is to explain the main idea
we will use to
prove Theorem~\ref{th:increase}; the formal proof will be given
in the subsection after this one.

The following theorem gives a number of ways of demonstrating whether or
not
a sheaf and one of its subsheaves have the same maximum excess.

\begin{theorem}\label{th:we_know_this} 
Let ${\cal F}'\subset
{\cal F}$ be sheaves on a digraph, $G$.  Let $U\subset{\cal F}(V)$
be the minimal maximizer of the excess of ${\cal F}$,
and let $U'\subset{\cal F}'(V)$
be the minimal maximizer of the excess of ${\cal F}'$.
Then
\begin{equation}\label{eq:we_know_this}
{\rm m.e.}({\cal F}') 
\le {\rm m.e.}({\cal F}),
\end{equation}
with equality iff $U=U'$ and
$$
\Gamma_{\rm ht}({\cal F},U') = \Gamma_{\rm ht}({\cal F}',U').
$$
\end{theorem}
We already know equation~(\ref{eq:we_know_this}) is true, 
since the maximum excess is a quasi-Betti number; the novelty 
of this theorem is that
we have a simple condition to characterize when equality holds.
\begin{proof}
Since $U'\subset{\cal F}(V)$ and 
$$
\Gamma_{\rm ht}({\cal F}',U') \subset \Gamma_{\rm ht}({\cal F},U'),
$$
we have that 
$$
{\rm m.e.}({\cal F}') = {\rm excess}({\cal F}',U')
\le {\rm excess}({\cal F},U') \le {\rm m.e.}({\cal F});
$$
hence equality holds in
equation~(\ref{eq:we_know_this}) iff
$$
{\rm excess}({\cal F}',U') = {\rm excess}({\cal F},U')={\rm m.e.}({\cal F}).
$$
The first equality holds iff
$$
\Gamma_{\rm ht}({\cal F},U') = \Gamma_{\rm ht}({\cal F}',U').
$$
The second equality holds iff $U'$ is also a maximizer for ${\cal F}$.
But since $U$ is the minimal maximizer for ${\cal F}$, this implies that
$U\subset U'$; but this means that $U\subset U'\subset{\cal F}'(V)$,
so $U$ is a maximizer for ${\cal F}'$, and hence $U'\subset U$ (since
$U'$ is the minimal maximizer for ${\cal F}'$).  Hence $U=U'$.
\end{proof}

Theorem~\ref{th:we_know_this} gives us a number of ways to
conclude that equation~(\ref{eq:we_know_this}) holds with strict inequality
in certain situations.
For example, consider a subgraph, $L$, of a Cayley bigraph, $G$, on a
group, $\cg$, and consider a value, $k$ for which
$$
{\rm m.e.}\bigl(\ck_M(L)\bigr) > 0
$$
for a generic $M\in\field^{k\times\cg}$.  Let $M'$ be obtained from
$M$ by removing its bottom row, and consider the minimal maximizer, 
$U=U(M')\subset \ck_{M'}(L)(V)$ of the excess of $\ck_{M'}(L)$.
We have that $\ck_M\subset\ck_{M'}$, and hence
$$
{\rm m.e.}\bigl(\ck_M(L)\bigr) = {\rm m.e.}\bigl( \ck_{M'}(L) \bigr)
$$
implies that $U(M')$, which is generically nonzero, lies entirely in
$\ck_M(L)(V)$.  But if $w\in\ck_{M'}(L)(V)$ is any nonzero vector
supported on $v\in V_G$, then we may identify $w$ with the corresponding
element of $\ck_{M'}(L)(v)$, and so
$$
w \in (\underline\field_L\cg)(v) \isom \bigoplus_{g\in \cg_L(v)}\field_g,
$$
where $\field_g$ denotes a copy of $\field$.  In other words,
$w$ is given by its $\cg$ components, which are (zero outside of
$\cg_L(v)$ and are) elements of $\field$.  Hence, if we add a generic
extra row to $M'$ on the bottom, to form $M$, the row, 
$\vec m =(m_g)_{g\in\cg}$ 
will (generically in $\vec m$) satisfy
\begin{equation}\label{eq:redundancy}
\sum_{g\in\cg} w_g m_{g} \ne 0.
\end{equation}
Hence $w\notin \ck_M(L)(V)$ generically, 
and therefore the minimal maximizers for
$\ck_M(L)$ and $\ck_{M'}(L)$ will generically
be different.  Hence, by Theorem~\ref{th:we_know_this}, we have
$$
{\rm m.e.}\bigl( \ck_{M'}(L) \bigr) \ge 1 +
{\rm m.e.}\bigl( \ck_M(L) \bigr) 
$$
for generic $M$ (and $M'$ obtained from $M$ by deleting its bottom row).
This argument will establish Theorem~\ref{th:increase};
all we need to do is to make this rigourous.

\subsection{Proof of Theorems~\ref{th:increase} and \ref{th:L'}}
\label{sb:rigour}

In this subsection we precisely state the idea in the last subsection
as Theorem~\ref{th:redundancy} and use it to prove Theorem~\ref{th:increase}.
We then easily conclude Theorem~\ref{th:L'}.

Let $\field$ be a field, $\cg$ a group, and $k\ge 1$ an integer.
If $M'\in\field^{(k-1)\times\cg}$ and $\vec m\in\field^\cg$, we define
${\rm merge}(M',\vec m)$ to be the element of $\field^{k\times \cg}$ whose
first $k-1$ rows consist of $M'$ and whose $k$-th row consists of
$\vec m$.

\begin{definition}
Let $L$ be a subgraph of a Cayley bigraph, $G$, on a group, $\cg$.  
Let $\field$ be an algebraically closed field.
Let $M'\in\field^{(k-1)\times\cg}$ be a matrix, for some integer $k\ge 1$,
that is $L$-surjective
and whose kernel, $\ck_{M'}(L)$, has nonzero maximum excess.
Define the {\em redundancy} of $M'$, denoted ${\rm redund}(M')$, to be
the subset of $\field^\cg$ consisting of $\vec m$ such that
$M={\rm merge}(M',\vec m)$ 
is $L$-surjective, and such that
$$
{\rm m.e.}\bigl( \ck_{M'}(L) \bigr) = {\rm m.e.}
\bigl( \ck_M(L) \bigr).
$$
\end{definition}

\begin{theorem}\label{th:redundancy}
Let $L$ be a subgraph of a Cayley bigraph, $G$, on a group, $\cg$.  
Let $\field$ be an algebraically closed field and $k$ a positive integer.
Let $M'\in\field^{(k-1)\times\cg}$ be a matrix
that 
is $L$-surjective, and 
whose kernel, $\ck_{M'}(L)$, has nonzero maximum excess.
Then the redundancy of $M'$ lies in a proper subspace of
$\field^\cg$.
\end{theorem}
\begin{proof}
This follows the argument of the last subsection.  If $U$ is the minimal
maximizer of $\ck_{M'}(L)$, then $U$ is nonzero because the maximum
excess is nonzero.  Hence there exists a $w\in U$ supported at $v\in V_G$
with $w\ne 0$.  So if ${\rm merge}(M',\vec m)$ 
is $L$-surjective, we have
$w\notin\ck_M(L)$ if
equation~(\ref{eq:redundancy}) holds.  Since $w\ne 0$,
equation~(\ref{eq:redundancy}) holds for all $\vec m$ outside of
a proper subspace of $\field^\cg$.
\end{proof}


\begin{proof}[Proof of Theorem~\ref{th:increase}]
If the generic maximum excesses were equal, then for a nonempty
Zariski
open subset, $U$, of $\field^{k\times \cg}$, we would have
for all $M\in U$ 
the maximum excess of $\ck_M(L)$ is the same as that of
$\ck_{M'}(L)$, where
$M'$ is obtained from $M$ by discarding its bottom
row.  Since $U$ is nonempty and Zariski open,
we have a polynomial, $p=p(M',\vec m)$ such that
$$
p(M',\vec m) \ne 0
$$
implies that $(M',\vec m)\in U$.  Write
$$
p(M',\vec m) = \sum_{n\in\integers_{\ge 0}^\cg} q_n(M') \vec m^n,
$$
and fix any $n$ such that $q_n\ne 0$.
Then $q_n(M')\ne 0$ for all $M'\in U'$ for a nonempty
Zariski open subset, $U'$, of $\field^{(k-1)\times \cg}$.
For any fixed
$M'\in U'$ we have $p(M',\vec m)$ is a nonzero polynomial in 
$\vec m$; hence for fixed $M'\in U'$ we have that
$(M',\vec m)\in U$ for a Zariski open subset
of $\vec m$ in $\field^\cg$.  

On the other hand, assuming that the maximum excesses in
Theorem~\ref{th:increase} are not both zero,
the generic maximum excess of $(L,G,\cg,\field,k-1)$
is positive.  Hence
$\ck_{M'}(L)$ has positive 
maximum excess for all $M'$ in some
nonempty, Zariski open subset, $U''$, of $\field^{(k-1)\times\cg}$.
But by Theorem~\ref{th:redundancy}, for any $M'\in U''$ we have that
$(M',\vec m)\notin U$ for $\vec m$ outside of a proper subspace
of $\field^\cg$.
But $U'$ and $U''$ must intersect (being two nonempty, Zariski open subsets
of an irreducible variety), which gives a contradiction.
\end{proof}

\begin{proof}[Proof of Theorem~\ref{th:L'}]
Let the generic maximum excess of $(L,G,\cg,\field,k)$ be $m_k$,
that of $(L,G,\cg,\field,k-1)$ be $m_{k-1}$, and that 
of $(L',G,\cg,\field,k-1)$ be $m'_{k-1}$.  Since
$k\le \rho(L)$ and hence $k-1\le\rho(L')$ (we can see
$\rho(L')\ge \rho(L)-1$ from equation~(\ref{eq:shrho})), 
we have that $m_k,m_{k-1},m'_{k-1}$ are all multiples of $|\cg|$.
The theorem is immediate if $m_k=0$, so we may assume $m_k>0$.
In this case Theorem~\ref{th:increase} implies that $m_{k-1}>m_k$,
and since these numbers are both multiples of $|\cg|$, we have
\begin{equation}\label{eq:greater_by_cg}
m_{k-1} \ge m_k + |\cg|.
\end{equation}
But the exact sequence in equation~(\ref{eq:second_induct_exact}) shows
that for any $M'\in\field^{(k-1)\times\cg}$ we have
\begin{equation}\label{eq:LL'}
{\rm m.e.}\bigl( \ck_{M'}(L') \bigr) \ge 
{\rm m.e.}\bigl( \ck_{M'}(L) \bigr) - |\cg|.
\end{equation}
Let $U,U'$, respectively, 
are the subsets of $M'\in\field^{(k-1)\times\cg}$ at which
$\ck_{M'}(L),\ck_{M'}(L')$, respectively, attain their generic value;
hence $U,U'$ are generic, and therefore so is $U\cap U'$.
Then applying equation~(\ref{eq:LL'}) to any $M'\in U\cap U'$ implies that
$$
m'_{k-1} \ge m_{k-1}- |\cg|.
$$
Combining this with equation~(\ref{eq:greater_by_cg}) gives $m'_{k-1}\ge m_k$,
which proves the theorem.
\end{proof}

\section{Proof of the SHNC}
\label{se:meproof}

In this section we prove the SHNC.  At this point we have all the tools
we need, except for one small detail.

\begin{lemma}\label{lm:remove_one}
Let $L$ be an arbitrary \'etale bigraph with
$\rho(L)>0$.  Then there exists an edge, $e\in E_L$, such that the
graph, $L'$, obtained by removing $e$ from $L$ has $\rho(L')=\rho(L)-1$.
\end{lemma}
\begin{proof}
For each $F\in E_L$ let $L_F$ denote the subgraph of $L$ obtained by
removing the edges in $F$ from $L$.  It is easy to see 
that for each $e\in E_L$ we have that
$\rho(L_{\{e\}})$ is either $\rho(L)$ or $\rho(L)-1$; this can be seen from
equation~(\ref{eq:shrho}), since removing $e$ from its connected 
component of $L$ leaves $h_1$ the same or reduces it by one;
alternatively, we can see this 
from the exact sequence
$$
0\to \underline\field_{L_{\{e\}}} \to 
\underline\field_L \to 
\underline\field_{L/L_{\{e\}}} \to 0,
$$
using the fact that $L/L_{\{e\}}$ is (edge supported and) of maximum
excess one.

For any two subgraphs, $L',L''$, of $L$ we have an exact sequence
$$
0\to \underline\field_{L'\cap L''}\to
\underline\field_{L'}\oplus\underline\field_{L''}\to
\underline\field_{L'\cup L''} \to 0.
$$
Hence
$$
\rho(L'\cap L'') \ge \rho(L')+\rho(L'')-\rho(L'\cup L'').
$$
Taking $F',F''$ to be disjoint subsets of $E_L$, we see that
if $\rho(L_{F'})=\rho(L_{F''})=\rho(L)$, then setting
$L'=L_{F'}$, $L''=L_{F''}$ yields
$$
\rho(L_{F'\cup F''}) \ge \rho(L),
$$
and so $\rho(L_{F'\cup F''})=\rho(L)$.
Hence, if $\rho(L_F)=\rho(L)$ for all $F\subset E_L$ of size one, then
by induction we can show this holds for $F\subset E_L$ of any size,
which is impossible (since removing all the edges of a graph leaves
it with $\rho=0$).  Hence there is at least one $e\in E_L$ for which
$\rho(L_{\{e\}})=\rho(L)-1$.
\end{proof}
Of course, one can give a purely graph theoretic proof of 
Lemma~\ref{lm:remove_one}; we now sketch such a proof.  
From equation~(\ref{eq:shrho}), it
suffices to show that if $L$ is connected with $h_1(L)\ge 2$ then
we can remove an edge from $L$ and reduce $h_1$ by one.  We claim
that it suffices to take any edge that remains after we repeatedly
prune the leaves of $L$.

\begin{definition} 
Let
$L$ be a subgraph of a Cayley bigraph, $G$ on a group, $\cg$.  
Let $\field$ be an algebraically closed field.
Then
by the {\em generic maximum excess of the $\rho$-kernel of type 
$(L,G,\cg,\field)$}
we mean the generic maximum excess of $(L,G,\cg,\field,\rho(L))$.
\end{definition}

\begin{theorem}\label{th:vanishing_rho}
Let 
$L$ be a subgraph of a Cayley bigraph, $G$ on a group, $\cg$.  
Let $\field$ be an algebraically closed field.
Then
the generic maximum excess of the $\rho$-kernel of type $(L,G,\cg)$
is zero.
\end{theorem}
\begin{proof}
Fix $G$ and $\cg$ and let us prove the theorem for all $L$ by induction
on $\rho(L)$.  The base case $\rho(L)=0$ follows by definition, since the
exact sequence
$$
0\to \ck\to \underline\field_L \cg \to \underline\field^0 \to 0
$$
implies
$$
{\rm m.e.}(\ck) \le {\rm m.e.}(\underline\field_L\cg) = 
\sum_{g\in \cg} \rho(Lg) = 0
$$
(since $\rho(Lg)=\rho(L)=0$ for all $g\in \cg$).
The inductive step of our induction on $\rho(L)$
is immediate from Theorem~\ref{th:L'} applied
to any $L'$ obtained from $L$ by removing a single edge
so that $\rho(L')=\rho(L)-1$; the existence of such an
$L'$ is given by Lemma~\ref{lm:remove_one}.
\end{proof}

\begin{proof}[Proof of Theorem~\ref{th:shnc}, the SHNC]
By the graph theoretic reformulation of the SHNC, it
suffices to show Theorem~\ref{th:main}.
By Theorem~\ref{th:sub_bi_enough}
it suffices to show that any subgraph, $L$, of a Cayley bigraph, $G$,
on a group, $\cg$,
is universal for the SHNC.
But by Theorem~\ref{th:vanishing_rho}, there exists a $\rho$-kernel,
$\ck=\ck_M(L)$ for $(L,G,\cg,\field)$ with vanishing maximum excess, for any
algebraically closed $\field$.
Hence we apply Theorem~\ref{th:motivate} to conclude that $L$
is universal for the SHNC.
\end{proof}

\section{Concluding Remarks}
\label{se:meconclude}

We finish this paper with a few concluding remarks.

In this chapter we have made no explicit reference to homology theories.
In \cite{friedman_sheaves} we have used the twisted homology to
prove that the maximum excess is a first quasi-Betti number; hence
the theorems in this paper ostensibly rely on homology theories.
However, we think it quite possible that one may able to prove
that the maximum excess is a first quasi-Betti number directly, or
give a direct proof of the inequalities we made use of in this paper.
For example, if ${\cal F}'\to{\cal F}$ is a monomorphism, then
since the maximum excess is a first quasi-Betti number we know that
$$
{\rm m.e.}({\cal F}')\le {\rm m.e.}({\cal F}).
$$
But this inequality is clear from the subsheaf formulation of
maximum excess in Theorem~\ref{th:me_as_subsheaf}.

We remark that we first proved the SHNC using twisted homology theory,
and then rewrote our proofs to use only maximum excess.
In fact, twisted homology theory offers some additional intuition
regarding the maximum excess.  Twisted homology theory shows that
(after pulling back appropriately, see \cite{friedman_sheaves}),
the maximum excess can be interpreted as the dimension of a certain
vector space of ``twisted harmonic one-forms'' of the sheaf.
When this dimension is $d>0$, one can impose $d'$
linear conditions on the twisted harmonic one-forms and still 
have a $d-d'$ dimensional space of one-forms.
This is how we view the variance in
$L$ of $\ck_{M'}(L)$, as in the 
exact sequence of equation~(\ref{eq:second_induct_exact}): to take
a space of one-forms on $\ck_{M'}(L)$ and obtain a one-form in
$\ck_{M'}(L')$, one has to impose $|\cg|\,|E_L\setminus E_{L'}|$
conditions on the one-forms, namely the conditions that they vanish
on the edges in $L\cg$ that do not lie in $L'\cg$.
Of course, one has to pullback by an appropriate covering map
to make this rigourous (see \cite{friedman_sheaves}), but
all the edge counts and dimension counts scale appropriately under
any covering.

The $k$-th power kernels in this paper are subsheaves of the 
constant sheaf $\underline\field\cg\isom\underline\field^{\cg}$.
We believe that subsheaves of constant sheaves satisfy some
stronger properties than general sheaves, regarding their homological
invariants and maximum excess.  It would nice to study this further.

Finally, we give a variant of our proof of the SHNC that involves
no homology theory and, in particular, avoids any use of 
Theorem~\ref{th:shmain}.
As before, let $L$ be any subgraph of a Cayley bigraph, $G$, on a group,
$\cg$, and let $\field$ be a field.
First note that using Appendix~\ref{ap:elem}, we can show that if
there is a $\rho$-kernel for $(L,G,\cg)$ with vanishing maximum excess,
then the SHNC holds for all pairs $(L,L')$, with $L'$ any subgraph of
$G$; Appendix~\ref{ap:elem} makes no use of homology or Theorem~\ref{th:shmain}.
(Appendix~\ref{ap:elem} is a bit tedious and long, however avoids use
of applying Theorem~\ref{th:etale_contagious} with $\alpha_1$ being
the maximum excess,
and the proof that the maximum excess is a scaling first quasi-Betti
number used Theorem~\ref{th:shmain}.)
Furthermore,
if the generic maximum excess of $(L,G,\cg,\field,\rho(L))$ were greater
than zero, then it would be at least $|\cg|$.  Then, by induction, for
$n=1,\ldots, \rho(L)$ we have that the generic maximum excess of 
$(L,G,\cg,\field,\rho(L)-n)$ would be at least $|\cg|(1+n)$, in view
of Theorems~\ref{th:generic_divisible}
and \ref{th:increase} (which makes no use of homology or
Theorem~\ref{th:shmain}).  But this is impossible for $n=\rho(L)$, since
the generic maximum excess of $(L,G,\cg,\field,0)$ is $|\cg|\rho(L)$,
because a $0$-th power kernel is plainly just $\underline\field_L\,\cg$,
which has maximum excess $\rho(L)|\cg|$.
Hence the SHNC holds for all pairs of subgraphs of Cayley bigraphs, and
hence holds for all pairs of \'etale bigraphs, by
Theorem~\ref{th:sub_bi_enough}.

\appendix
\chapter{A Direct View of $\rho$-Kernels}
\label{ap:elem}

In this appendix we give a direct combinatorial proof that the SHNC 
follows the vanishing
maximum excess of some $\rho$-kernel for each triple $(L,G,\cg)$.


In this section we give a direct proof that the vanishing generic
maximum excess of $\rho$-kernels for all subgraphs, $L$, of any
Cayley graph, $G$, implies the SHNC.  We shall not use exact sequences.
We shall require a few definitions, and some calculations to follow.
While this gives some extra intuition about $\rho$-kernels, this
section is not essential to the proof of the SHNC; we shall omit
some of the easy but tedious graph theoretic details.

\begin{definition} Let $L$ be a subgraph of a Cayley bigraph, $G$, on
a group, $\cg$.  As usual, for $P\in V_G\amalg E_G$, let
$\cg_L(P)$ be the set of $g\in\cg$ such that $Lg$ contains $P$.
By a {\em vertex family} on $(L,G,\cg)$ we mean a function, ${\cal U}$, 
from $V_G$ to $\cp(\cg)$,
the power set (i.e., set of subsets) of $\cg$, such that
for all $v\in V_G$ we have ${\cal U}(v)\subset\cg_L(v)$.
Similarly, an {\em edge family} on $(L,G,\cg)$ is a function
${\cal W}\from E_G\to \cp(\cg)$ such that ${\cal W}(e)\subset \cg_L(e)$
for all $e\in E_G$.
A vertex family, ${\cal U}$, and edge family, ${\cal W}$, are
{\em compatible} if for all $e\in E_G$ we have ${\cal W}(e)\subset
{\cal U}(te)\cap{\cal U}(he)$.
Given a vertex family, ${\cal U}$,
the {\em induced edge family}, ${\cal U}_E$, of ${\cal U}$ is the edge
family ${\cal U}_E$ given by
$$
{\cal U}_E(e) = {\cal U}(te)\cap{\cal U}(he)\cap \cg_L(e).
$$
\end{definition}
The following lemmas motivate the above definitions; we omit their proofs,
which are almost immediate.
\begin{lemma}
To each vertex family, ${\cal U}$ on $(L,G,\cg)$, and compatible edge family,
${\cal W}$, on $(L,G,\cg)$, there is a subgraph
$H\subset L\times_{B_2}G$ determined via
$$
V_H = \{(vg^{-1},v)\ |\ g\in{\cal U}(v) \}, \qquad
E_H = \{(eg^{-1},e)\ |\ g\in{\cal W}(e) \};
$$
conversely, any subgraph $H\subset L\times_{B_2}G$ arises from a unique
vertex family, ${\cal U}$, and compatible edge family, ${\cal W}$.
\end{lemma}
\begin{lemma}
For any vertex family, ${\cal U}$, on $(L,G,\cg)$ and compatible
edge family, ${\cal W}$, we have
$$
{\cal W}(e)\subset {\cal U}_E(e).
$$
In other words, ${\cal U}_E$ is the ``largest'' edge family compatible
with ${\cal U}$.
\end{lemma}

\begin{definition} 
Let $L$ be a subgraph of a Cayley bigraph, $G$, on a group, $\cg$, 
let $M\in\field^{\rho(L)\times\cg}$ be totally linearly independent,
and let $\ck=\ck_M$ be the resulting $\rho$-kernel.
By a {\em straight} subspace of $\ck(V)$ we mean a subspace
$$
U = \sum_{v\in V_G} U(v) \in \ck(V),
$$
such that for each $v\in V_G$, we have
\begin{equation}\label{eq:straight}
U(v)={\rm Free}_{{\cal U}(v)}
\end{equation}
for some ${\cal U}(v)\subset\cg$, with notation as in 
Definition~\ref{de:free}.
\end{definition}

Our goal for the rest of this section is to prove the following
theorem.

\begin{theorem}\label{th:elementary}
Let $L$ be a subgraph of a Cayley graph, $G$, on a group, $\cg$.
The following conditions are equivalent:
\begin{enumerate}
\item for all $L'\subset G$, the SHNC holds for
$(L,L')$; 
\item
for every vertex family, ${\cal U}$, on $(L,G,\cg)$ we have
$$
\sum_{e\in E_G} |{\cal U}_E(e)|_{\rho(L)} \le
\sum_{v\in V_G} |{\cal U}(v)|_{\rho(L)} ;
$$
and
\item for some or any field, $\field$, and some or any
totally independent $M\in\field^{\rho(L)\times\cg}$,
every straight subspace of
$\ck(V)$ with $\ck=\ck_M(L)$ has excess zero. 
\end{enumerate} 
\end{theorem}
We know by Theorem~\ref{th:sub_bi_enough}
that the SHNC holds iff it holds for all pairs $(L,L')$
that are subgraphs of a Cayley graph, $G$.  
Hence Theorem~\ref{th:vanishing_rho}, that implies
condition~(3) of this theorem, implies the SHNC.
\begin{proof}
Conditions~(2) and (3) are easily seen to be 
equivalent via equation~(\ref{eq:straight}) and
equation~(\ref{eq:free_dimension}).

If ${\cal U}$ is any vertex family on $(L,G,\cg)$, let the {\em positive
set} of ${\cal U}$ consist of those $v\in V_G$ for which 
$|{\cal U}(v)|>\rho(L)$ and of those $e\in E_G$ for which
$|{\cal U}_E(e)|>\rho(L)$.  We easily see that the positive set
forms a subgraph, $L'$, of $G$, and that the pairs $(Pg^{-1},P)$ such that
$P$ is in the positive set and $g$ lies in ${\cal U}(P)$ or ${\cal U}_E(P)$
(as is appropriate), forms a subgraph, $H$, of $L\times_{B_2}L'$.  We see that
$$
-\chi(H)=\sum_{e\in E_{L'}} |{\cal U}_E(e)| - \sum_{v\in V_{L'}} |{\cal U}(v)|
$$
$$
= \rho(L)\chi(L') +
\sum_{e\in E_{L'}} \bigl( |{\cal U}_E(e)|-{\rho(L)}\bigr) - 
\sum_{v\in V_{L'}} \bigl( |{\cal U}(v)|-{\rho(L)}\bigr) 
$$
$$
= \rho(L)\chi(L') +
\sum_{e\in E_{G}} |{\cal U}_E(e)|_{\rho(L)} - 
\sum_{v\in V_{G}} |{\cal U}(v)|_{\rho(L)} .
$$
Hence we may write
\begin{equation}\label{eq:summary}
-\chi(H)-\rho(L)\chi(L') =
\sum_{e\in E_{G}} |{\cal U}_E(e)|_{\rho(L)} -
\sum_{v\in V_{G}} |{\cal U}(v)|_{\rho(L)} .
\end{equation}
This equation is the main ingredient in the equivalence of conditions~(1)
and (3).
Let us now state some graph theoretic lemmas that will firmly
establish this equivalence.

\begin{lemma} For any digraphs $H\subset G$, we have
$$
-\chi(H)\le \rho(G);
$$
and equality holds if $H$ consists of all connected components, $X$, of $G$
with $h_1(X)>0$ and any of those with $h_1(X)=0$.
\end{lemma}
\begin{proof} The statement about equality is clear from the definition
of $\rho$ in equation~(\ref{eq:shrho}).  
The inequality can be established
graph theoretically by induction on the number of vertices and edges
in $G$ that are not in $H$.  Alternatively, see the end 
of Section~\ref{se:shme}.
\end{proof}

\begin{lemma} Let ${\cal U}$ be a vertex family for $(L,G,\cg)$, where
$L$ is a subgraph of a Cayley digraph, $G$, on a group, $\cg$.
Assume that
\begin{equation}\label{eq:sub_rho}
\sum_{e\in E_{G}} |{\cal U}_E(e)|_{\rho(L)} -
\sum_{v\in V_{G}} |{\cal U}(v)|_{\rho(L)} > 0.
\end{equation}
Then there is a vertex family, ${\cal U}'$, which satisfies this inequality
with ${\cal U}$ replaced by ${\cal U}'$, for which the positive set of
${\cal U}'$, $L'$, satisfies $-\chi(L')=\rho(L')$.
\end{lemma}
\begin{proof}
For any subgraph, $Y\subset G$ and vertex family ${\cal W}$ on $(L,G,\cg)$,
set
$$
f({\cal W},Y)=\sum_{e\in E_{Y}} |{\cal W}_E(e)|_{\rho(L)} -
\sum_{v\in V_{Y}} |{\cal W}(v)|_{\rho(L)}.
$$
Then clearly we have
$$
f({\cal U},L') = \sum_{X\in{\rm conn}(L')} f({\cal U},X),
$$
where ${\rm conn}(L')$ is the set of connected components of $L'$.  
But
equation~(\ref{eq:sub_rho}) says that $f({\cal U},G)>0$, and clearly
$f({\cal U},L')=f({\cal U},G)$.  Hence we have $f({\cal U},X)>0$ for 
some connected component, 
$X$, of $L'$; fix any such $X$.

We claim that $\rho(X)=-\chi(X)$.  Since $X$ is connected, this is
true unless $\chi(X)=1$; so it suffices to show that $\chi(X)=1$
is impossible.
If $\chi(X)=1$, then by repeatedly pruning the leaves of $X$, i.e.,
deleting a vertex of degree one and its incident edge from $X$, we
arrive at an isolated vertex.  But if $Y$ is any subgraph of $G$
with a vertex,
$v\in V_Y$, of degree one, and incident edge $e\in E_Y$, and if $Y'$
is $Y$ with $v$ and $e$ discarded, we claim that $f({\cal U},Y')\ge f({\cal U},Y)$;
indeed, ${\cal U}_E(e)\subset {\cal U}(v)$, so
$$
f({\cal U},Y') = f({\cal U},Y) + |{\cal U}(v)|_{\rho(L)} - |{\cal U}(e)|_{\rho(L)} \ge f({\cal U},Y).
$$
Hence, by repeatedly pruning $X$ we are left with $X''$ that is a single
vertex with no edges, so $f({\cal U},X'')\ge f({\cal U},X)>0$.  
But clearly $f({\cal U},X'')\le 0$ for $X''$ consisting of a single
vertex.  Hence 
$\chi(X)=1$ is impossible, and so $\chi(X)\le 0$ and
hence $\rho(X)=-\chi(X)$.

For any vertex family, ${\cal V}$ of $(L,G,\cg)$ and any subgraph
$Y\in G$ define a vertex family ${\cal V}|_Y$ via
$$
{\cal V}|_Y(v) = \left\{ \begin{array}{ll} {\cal V}(v) & \mbox{if $v\in V_Y$,}
\\ \emptyset & \mbox{otherwise.} \end{array}\right.
$$
for all $v\in V_G$.  We easily see that
\begin{equation}\label{eq:W_subset}
{\cal V}(e)\subset ({\cal V}|_L)_E(e)
\end{equation}
for all $e\in E_L$.
Hence
$$
f({\cal V}|_Y,G) = f({\cal V}|_Y,Y) \ge f({\cal V},Y),
$$
using equation~(\ref{eq:W_subset}).
In particular, for ${\cal V}={\cal U}$ and $Y=X$ we have
$$
f({\cal U}|_X,G) \ge f({\cal U},X)>0.
$$
So we take ${\cal U}'={\cal U}|_X$ and let $L'$ be its positive set.
Then $f({\cal U}',L')=f({\cal U}',G)>0$, and
$L'$ consists of $X$ plus possibly some addition edges, so $L'$ is connected
and $\chi(L')\le\chi(X)\le 0$, so $\rho(L')=-\chi(L')$.
\end{proof}

At this point condition~(1) of Theorem~\ref{th:elementary}
easily implies condition~(2).  For if
condition~(2) does not hold, then for some ${\cal U}$, and with $L'$ given
as its positive set, we may assume $\rho(L')=-\chi(L')$ we have
$$
\sum_{e\in E_{G}} |{\cal U}_E(e)|_{\rho(L)} - 
\sum_{v\in V_{G}} |{\cal U}(v)|_{\rho(L)}  > 0
$$
and hence
$$
\rho(L\times_{B_2}L')\ge -\chi(H) > -\rho(L)\chi(L')
=\rho(L)\rho(L').
$$
Hence the SHNC is false on a pair of subgraphs of $G$.

It remains to show that condition~(2) of Theorem~\ref{th:elementary}
implies condition~(1).  Again, we need some graph theoretic considerations.

\begin{lemma}
Assume the SHNC is false on a pair of subgraphs, $(L,L')$, of a Cayley
bigraph $G$ on a group $\cg$.
Then there is a subgraph, $L''\subset L'$, such that 
\begin{enumerate}
\item the SHNC is
false on $(L,L'')$, 
\item $L''$ is connected,
\item $-\chi(L'')=\rho(L'')$, and
\item there is a subgraph, $H\subset L\times_{B_2}L''$ such that
$-\chi(H)=\rho(L\times_{B_2}L'')$, and if ${\cal U}$ is the vertex
family associated to $H$, then we have
$$
f({\cal U},G) >0 .
$$
\end{enumerate}
\end{lemma}
If condition~(1) of Theorem~\ref{th:elementary} is false, then the 
hypothesis of the above lemma holds; but item~(4) of the lemma
means that condition~(2) of Theorem~\ref{th:elementary} is false.
Hence we conclude this section, and the proof of Theorem~\ref{th:elementary}
with the proof of the above lemma.

\begin{proof}
So assume the SHNC is false on a pair of subgraphs,
$(L,L')$.  Similar to before, the SHNC must therefore be false
on $(L,L'')$, where $L''$ is some 
connected component of $L$.  
Fix such a connected component, $L''$.

We 
cannot have $\rho(L'')=0$, for otherwise $\rho(L\times_{B_2}L'')=0$ and
the SHNC is not false on $(L,L'')$.
Hence we have $L''\subset L'$ is connected and $\rho(L'')>0$, whereupon we have
$-\chi(L'')=\rho(L'')$.

Now let $L''\subset L'$ be a minimal subgraph of $L'$ (with respect to
inclusion of subgraphs) with the properties
that $L''$ is connected, $\rho(L'')>0$, and the SHNC is false on $(L,L'')$.
We shall show that the lemma holds with this subgraph, $L''$;
we have already established the first three items in the conclusion.

Then take any $H\subgraph L\times_{B_2}L''$ such that
$-\chi(H)=\rho(L\times_{B_2}L'')$.  
We now make a number of claims regarding $L''$ that follow from
the minimality of $L''$.

First, we claim that $L''$ has no leaves, i.e., no vertices of degree one.
Otherwise, if $v$ is a vertex of degree one, and $e$ is its incident edge,
then there are at least as many vertices in $H$ over $v$ as there are over
$v$.  So letting $L'''$ be $L''$ with $v$ and $e$ discarded, we see that
$\rho(L''')=\rho(L'')$; but if $H'$ consists of the vertices and edges
of $H$ that do not lie over $e$ or $v$, then $H'$ is a subgraph of $H$
and $\rho(H')=\rho(H)$ (since we obtain $H'$ from $H$ by discarding
isolated vertices over $v$ or vertices over $v$ along with their single
incident edges,
lying over $e$).  Hence the SHNC would fail also on $(L,L''')$, contradicting
the minimality of $L''$.

Second, we claim that over each $e\in E_{L''}$ we 
there are at least $\rho(L)+1$
edges in $H$; if not, we delete $e$ from $L''$, obtaining $L'''\subset
L''$, and delete the
at most $\rho(L)$ edges over $e$ from $H$, 
obtaining $H'$ that lies over $L'''$;
this yields a strict subgraph, $L'''$ of 
$L''$ such that
$$
\rho(L\times_{B_2}L''')\ge -\chi(H')\ge-\chi(H)-\rho(L)=
\rho(L\times_{B_2}L'')-\rho(L)
$$
$$
>\rho(L)\bigl(\rho(L'')-1\bigr).
$$
But since $L''$ is pruned, we have $\rho(L''')=\rho(L'')-1$.  So,
once again, we have the SHNC fails on $(L,L''')$ for some a proper subgraph,
$L'''$,
of $L''$; this contradicts the minimality of $L''$.

Third, we claim that over each $v\in V_{L''}$ there are at least
$\rho(L)+1$ vertices in $H$.  Indeed, if $v$ is incident upon some edge,
$e$, in $L''$, then $e$ has at least $\rho(L)+1$ vertices in $H$ above
it, so $v$ does as well.  If $v$ is isolated in $L''$, i.e., incident
upon no edge, then $L''$ consists of only $v$, since $L''$ is connected;
but this contradicts the fact that $\rho(L'')>0$.

To $H$ is associated a vertex family, ${\cal U}$, and an edge family,
${\cal W}$.  According to the three claims established in the previous
three paragraphs, we have
$$
v\in V_{L''} \implies |{\cal U}(v)|\ge \rho(L)+1, 
\qquad
e\in E_{L''} \implies |{\cal W}(e)|\ge \rho(L)+1.
$$
Clearly also 
$$
v\notin V_{L''} \implies {\cal U}(v)=\emptyset,
\qquad
e\notin E_{L''} \implies {\cal W}(e)=\emptyset.
$$
It follows that, as before
$$
f({\cal U},G) \ge
\sum_{e\in E_{L''}} |{\cal W}(e)|_{\rho(L)}
-\sum_{v\in V_{L''}} |{\cal U}(e)|_{\rho(L)}
$$
$$
=-\chi(H)-\rho(L)(|E_{L''}|-|V_{L''}|)
=\rho(L\times_{B_2}L'')-\rho(L)\rho(L''),
$$
using the fact that $L''$ is connected and $\rho(L'')>0$
(so that $\rho(L'')=|E_{L''}|-|V_{L''}|$).
Hence $f({\cal U},G)>0$, which shows item~(4) in the conclusion
of the lemma.
\end{proof}
\end{proof}

\backmatter
\providecommand{\bysame}{\leavevmode\hbox to3em{\hrulefill}\thinspace}
\providecommand{\MR}{\relax\ifhmode\unskip\space\fi MR }
\providecommand{\MRhref}[2]{%
  \href{http://www.ams.org/mathscinet-getitem?mr=#1}{#2}
}
\providecommand{\href}[2]{#2}



\begin{thebibliography}{Min11b}

\bibitem[AL02]{amit_random}
Alon Amit and Nathan Linial, \emph{Random graph coverings {I}: General theory
  and graph connectivity}, Combinatorica \textbf{22} (2002), no.~1, 1--18.

\bibitem[AM69]{atiyah_book}
M.~F. Atiyah and I.~G. Macdonald, \emph{Introduction to commutative algebra},
  Addison-Wesley Publishing Co., Reading, Mass.-London-Don Mills, Ont., 1969.
  \MR{0242802 (39 \#4129)}

\bibitem[Arz00]{ar00}
G.~N. Arzhantseva, \emph{A property of subgroups of infinite index in a free
  group}, Proc. Amer. Math. Soc. \textbf{128} (2000), no.~11, 3205--3210.
  \MR{MR1694447 (2001b:20040)}

\bibitem[Ati76]{atiyah}
M.~F. Atiyah, \emph{Elliptic operators, discrete groups and von {N}eumann
  algebras}, Colloque ``{A}nalyse et {T}opologie'' en l'{H}onneur de {H}enri
  {C}artan ({O}rsay, 1974), Soc. Math. France, Paris, 1976, pp.~43--72.
  Ast\'erisque, No. 32--33. \MR{MR0420729 (54 \#8741)}

\bibitem[Bur71]{burns71}
Robert~G. Burns, \emph{On the intersection of finitely generated subgroups of a
  free group.}, Math. Z. \textbf{119} (1971), 121--130. \MR{MR0279166 (43
  \#4892)}

\bibitem[Del77]{sga4.5}
P.~Deligne, \emph{Cohomologie \'etale}, Lecture Notes in Mathematics, Vol. 569,
  Springer-Verlag, Berlin, 1977, S{\'e}minaire de G{\'e}om{\'e}trie
  Alg{\'e}brique du Bois-Marie SGA 4$\frac{1}{2}$, Avec la collaboration de J.
  F. Boutot, A. Grothendieck, L. Illusie et J. L. Verdier. \MR{MR0463174 (57
  \#3132)}

\bibitem[DF01]{dicks01}
Warren Dicks and Edward Formanek, \emph{The rank three case of the {H}anna
  {N}eumann conjecture}, J. Group Theory \textbf{4} (2001), no.~2, 113--151.
  \MR{MR1812321 (2002e:20051)}

\bibitem[Dic94]{dicks94}
Warren Dicks, \emph{Equivalence of the strengthened {H}anna {N}eumann
  conjecture and the amalgamated graph conjecture}, Invent. Math. \textbf{117}
  (1994), no.~3, 373--389. \MR{MR1283723 (95c:20034)}

\bibitem[Eve08]{everitt08}
Brent Everitt, \emph{Graphs, free groups and the {H}anna {N}eumann conjecture},
  J. Group Theory \textbf{11} (2008), no.~6, 885--899. \MR{MR2466915}

\bibitem[FMT06]{friedman_murty_tillich}
Joel Friedman, Ram Murty, and Jean-Pierre Tillich, \emph{Spectral estimates for
  abelian cayley graphs}, J. Comb. Theory Ser. B \textbf{96} (2006), no.~1,
  111--121.

\bibitem[Fri93]{friedman_geometric_aspects}
Joel Friedman, \emph{Some geometric aspects of graphs and their
  eigenfunctions}, Duke Math. J. \textbf{69} (1993), no.~3, 487--525.
  \MR{94b:05134}

\bibitem[Fri03]{friedman_relative}
\bysame, \emph{Relative expanders or weakly relatively {R}amanujan graphs},
  Duke Math. J. \textbf{118} (2003), no.~1, 19--35. \MR{MR1978881
  (2004m:05165)}

\bibitem[Fri05]{friedman_cohomology}
\bysame, \emph{Cohomology of grothendieck topologies and lower bounds in
  boolean complexity}, {\tt http://www.math.ubc.ca/\~{}jf}, also {\tt
  http://arxiv.org/abs/cs/0512008}, to appear.

\bibitem[Fri06]{friedman_cohomology2}
\bysame, \emph{Cohomology of grothendieck topologies and lower bounds in
  boolean complexity ii}, {\tt http://www.math.ubc.ca/\~{}jf}, also {\tt
  http://arxiv.org/abs/cs/0604024}, to appear.

\bibitem[Fri07]{friedman_linear}
\bysame, \emph{Linear transformations in boolean complexity theory}, CiE '07:
  Proceedings of the 3rd conference on Computability in Europe (Berlin,
  Heidelberg), Springer-Verlag, 2007, pp.~307--315.

\bibitem[Fri08]{friedman_alon}
\bysame, \emph{A proof of {A}lon's second eigenvalue conjecture and related
  problems}, Mem. Amer. Math. Soc. \textbf{195} (2008), no.~910, viii+100.
  \MR{MR2437174}

\bibitem[Fri11a]{friedman_sheaves_hnc}
Joel Friedman, \emph{Sheaves on graphs and a proof of the hanna neumann
  conjecture}, available at {\tt http://arxiv.org/pdf/1105.0129v1} and at {\tt
  http://www.math.ubc.ca/\~{}jf}.

\bibitem[Fri11b]{friedman_sheaves}
\bysame, \emph{Sheaves on graphs and their homological invariants}, available
  at {\tt http://arxiv.org/pdf/1104.2665v1} and at {\tt
  http://www.math.ubc.ca/\~{}jf}.

\bibitem[FT05]{friedman_tillich_generalized}
Joel Friedman and Jean-Pierre Tillich, \emph{Generalized {A}lon--{B}oppana
  theorems and error-correcting codes}, SIAM J. Discret. Math. \textbf{19}
  (2005), 700--718.

\bibitem[Ger83]{gersten83}
S.~M. Gersten, \emph{Intersections of finitely generated subgroups of free
  groups and resolutions of graphs}, Invent. Math. \textbf{71} (1983), no.~3,
  567--591. \MR{MR695907 (85m:05037b)}

\bibitem[GM03]{gelfand}
Sergei~I. Gelfand and Yuri~I. Manin, \emph{Methods of homological algebra},
  second ed., Springer Monographs in Mathematics, Springer-Verlag, Berlin,
  2003. \MR{MR1950475 (2003m:18001)}

\bibitem[Gro77]{gross}
Jonathan~L. Gross, \emph{Every connected regular graph of even degree is a
  {S}chreier coset graph}, J. Combinatorial Theory Ser. B \textbf{22} (1977),
  no.~3, 227--232. \MR{MR0450121 (56 \#8419)}

\bibitem[Har77]{hartshorne}
Robin Hartshorne, \emph{Algebraic geometry}, Springer-Verlag, New York, 1977,
  Graduate Texts in Mathematics, No. 52. \MR{57 \#3116}

\bibitem[Har92]{harris}
Joe Harris, \emph{Algebraic geometry}, Graduate Texts in Mathematics, vol. 133,
  Springer-Verlag, New York, 1992, A first course. \MR{1182558 (93j:14001)}

\bibitem[How54]{howson54}
A.~G. Howson, \emph{On the intersection of finitely generated free groups}, J.
  London Math. Soc. \textbf{29} (1954), 428--434. \MR{MR0065557 (16,444c)}

\bibitem[HS97]{hilton}
P.~J. Hilton and U.~Stammbach, \emph{A course in homological algebra}, second
  ed., Graduate Texts in Mathematics, vol.~4, Springer-Verlag, New York, 1997.
  \MR{1438546 (97k:18001)}

\bibitem[Imr77a]{imrich77}
Wilfried Imrich, \emph{On finitely generated subgroups of free groups}, Arch.
  Math. (Basel) \textbf{28} (1977), no.~1, 21--24. \MR{MR0439941 (55 \#12822)}

\bibitem[Imr77b]{imrich76}
\bysame, \emph{Subgroup theorems and graphs}, Combinatorial mathematics, {V}
  ({P}roc. {F}ifth {A}ustral. {C}onf., {R}oy. {M}elbourne {I}nst. {T}ech.,
  {M}elbourne, 1976), Springer, Berlin, 1977, pp.~1--27. Lecture Notes in
  Math., Vol. 622. \MR{MR0463016 (57 \#2980)}

\bibitem[Iva99]{ivanov99}
S.~V. Ivanov, \emph{On the intersection of finitely generated subgroups in free
  products of groups}, Internat. J. Algebra Comput. \textbf{9} (1999), no.~5,
  521--528. \MR{MR1719719 (2000k:20023)}

\bibitem[Iva01]{ivanov01}
\bysame, \emph{Intersecting free subgroups in free products of groups},
  Internat. J. Algebra Comput. \textbf{11} (2001), no.~3, 281--290.
  \MR{MR1847180 (2002e:20052)}

\bibitem[JKM03]{jitsukawa03}
Toshiaki Jitsukawa, Bilal Khan, and Alexei~G. Myasnikov, \emph{On the {H}anna
  {N}eumann conjecture}, 2003, Available as {\tt
  http://arxiv.org/abs/math/0302009}.

\bibitem[Kha02]{khan02}
Bilal Khan, \emph{Positively generated subgroups of free groups and the {H}anna
  {N}eumann conjecture}, Combinatorial and geometric group theory ({N}ew
  {Y}ork, 2000/{H}oboken, {NJ}, 2001), Contemp. Math., vol. 296, Amer. Math.
  Soc., Providence, RI, 2002, pp.~155--170. \MR{MR1921710 (2003e:20027)}

\bibitem[Lan02]{lang}
Serge Lang, \emph{Algebra}, third ed., Graduate Texts in Mathematics, vol. 211,
  Springer-Verlag, New York, 2002. \MR{1878556 (2003e:00003)}

\bibitem[L{\"u}c02]{luck02}
Wolfgang L{\"u}ck, \emph{{$L^2$}-invariants: theory and applications to
  geometry and {$K$}-theory}, Ergebnisse der Mathematik und ihrer Grenzgebiete.
  3. Folge. A Series of Modern Surveys in Mathematics [Results in Mathematics
  and Related Areas. 3rd Series. A Series of Modern Surveys in Mathematics],
  vol.~44, Springer-Verlag, Berlin, 2002. \MR{MR1926649 (2003m:58033)}

\bibitem[Min10]{mineyev10}
Igor Mineyev, \emph{The topology and analysis of the {H}anna {N}eumann
  {C}onjecture}, Preprint. Available at {\tt
  http://www.math.uiuc.edu/\~{}mineyev/math/art/shnc.pdf}.

\bibitem[Min11a]{mineyev11.2}
\bysame, \emph{Groups, graphs, and the {H}anna {N}eumann {C}onjecture},
  Preprint. Available at {\tt
  http://www.math.uiuc.edu/\~{}mineyev/math/art/gr-gr-shnc.pdf}.

\bibitem[Min11b]{mineyev11.1}
\bysame, \emph{Submultiplicativity and the {H}anna {N}eumann {C}onjecture},
  Preprint. Available at {\tt
  http://www.math.uiuc.edu/\~{}mineyev/math/art/submult-shnc.pdf}.

\bibitem[MW02]{meakin02}
J.~Meakin and P.~Weil, \emph{Subgroups of free groups: a contribution to the
  {H}anna {N}eumann conjecture}, Proceedings of the {C}onference on {G}eometric
  and {C}ombinatorial {G}roup {T}heory, {P}art {I} ({H}aifa, 2000), vol.~94,
  2002, pp.~33--43. \MR{MR1950872 (2003k:20028)}

\bibitem[Neu56]{hanna}
Hanna Neumann, \emph{On the intersection of finitely generated free groups},
  Publ. Math. Debrecen \textbf{4} (1956), 186--189. \MR{MR0078992 (18,11f)}

\bibitem[Neu57]{hanna_add}
\bysame, \emph{On the intersection of finitely generated free groups.
  {A}ddendum}, Publ. Math. Debrecen \textbf{5} (1957), 128. \MR{MR0093537 (20
  \#61)}

\bibitem[Neu90]{walter90}
Walter~D. Neumann, \emph{On intersections of finitely generated subgroups of
  free groups}, Groups---Canberra 1989, Lecture Notes in Math., vol. 1456,
  Springer, Berlin, 1990, pp.~161--170. \MR{MR1092229 (92b:20026)}

\bibitem[Neu07]{walter07}
Walter~D. Neumann, \emph{A short proof that positive generation implies the
  {H}anna {N}eumann {C}onjecture}, 2007, Available as {\tt
  http://arxiv.org/abs/math/0702395}, to appear.

\bibitem[Ser83]{servatius83}
Brigitte Servatius, \emph{A short proof of a theorem of {B}urns}, Math. Z.
  \textbf{184} (1983), no.~1, 133--137. \MR{MR711734 (85c:20019)}

\bibitem[sga72a]{sga4.1}
\emph{Th\'eorie des topos et cohomologie \'etale des sch\'emas. {T}ome 1:
  {T}h\'eorie des topos}, Springer-Verlag, Berlin, 1972, S\'eminaire de
  G\'eom\'etrie Alg\'ebrique du Bois-Marie 1963--1964 (SGA 4), Dirig\'e par M.
  Artin, A. Grothendieck, et J. L. Verdier. Avec la collaboration de N.
  Bourbaki, P. Deligne et B. Saint-Donat, Lecture Notes in Mathematics, Vol.
  269. \MR{50 \#7130}

\bibitem[sga72b]{sga4.2}
\emph{Th\'eorie des topos et cohomologie \'etale des sch\'emas. {T}ome 2},
  Springer-Verlag, Berlin, 1972, S\'eminaire de G\'eom\'etrie Alg\'ebrique du
  Bois-Marie 1963--1964 (SGA 4), Dirig\'e par M. Artin, A. Grothendieck et J.
  L. Verdier. Avec la collaboration de N. Bourbaki, P. Deligne et B.
  Saint-Donat, Lecture Notes in Mathematics, Vol. 270. \MR{50 \#7131}

\bibitem[sga73]{sga4.3}
\emph{Th\'eorie des topos et cohomologie \'etale des sch\'emas. {T}ome 3},
  Springer-Verlag, Berlin, 1973, S\'eminaire de G\'eom\'etrie Alg\'ebrique du
  Bois-Marie 1963--1964 (SGA 4), Dirig\'e par M. Artin, A. Grothendieck et J.
  L. Verdier. Avec la collaboration de P. Deligne et B. Saint-Donat, Lecture
  Notes in Mathematics, Vol. 305. \MR{50 \#7132}

\bibitem[sga77]{sga5}
\emph{Cohomologie {$l$}-adique et fonctions {$L$}}, Lecture Notes in
  Mathematics, Vol. 589, Springer-Verlag, Berlin, 1977, S{\'e}minaire de
  G{\'e}ometrie Alg{\'e}brique du Bois-Marie 1965--1966 (SGA 5), Edit{\'e} par
  Luc Illusie. \MR{MR0491704 (58 \#10907)}

\bibitem[ST96]{st1}
H.~M. Stark and A.~A. Terras, \emph{Zeta functions of finite graphs and
  coverings}, Adv. Math. \textbf{121} (1996), no.~1, 124--165. \MR{MR1399606
  (98b:11094)}

\bibitem[Sta83]{stallings83}
John~R. Stallings, \emph{Topology of finite graphs}, Invent. Math. \textbf{71}
  (1983), no.~3, 551--565. \MR{MR695906 (85m:05037a)}

\bibitem[Tar92]{tardos92}
G{\'a}bor Tardos, \emph{On the intersection of subgroups of a free group},
  Invent. Math. \textbf{108} (1992), no.~1, 29--36. \MR{MR1156384 (93c:20048)}

\bibitem[Tar96]{tardos96}
\bysame, \emph{Towards the {H}anna {N}eumann conjecture using {D}icks' method},
  Invent. Math. \textbf{123} (1996), no.~1, 95--104. \MR{MR1376247 (97b:20029)}

\bibitem[Wis05]{wise05}
Daniel~T. Wise, \emph{The coherence of one-relator groups with torsion and the
  {H}anna {N}eumann conjecture}, Bull. London Math. Soc. \textbf{37} (2005),
  no.~5, 697--705. \MR{MR2164831 (2006f:20037)}

\end{thebibliography}
\end{document}